\RequirePackage{fix-cm}

\documentclass[smallextended]{svjour3} 

\smartqed 

\usepackage{amsmath}

\usepackage{lipsum}
\usepackage{amsfonts}
\usepackage{graphicx}
\usepackage{epstopdf}
\usepackage{algorithmic}
\usepackage{amssymb}
\usepackage{mathtools}
\usepackage{bm}
\usepackage{stmaryrd}
\usepackage{braket}
\usepackage{subcaption}
\usepackage{multirow}
\usepackage[utf8]{inputenc}
\usepackage{pgfplots}
\pgfplotsset{compat=newest}
\usepgfplotslibrary{groupplots}
\usepgfplotslibrary{dateplot}
\usepackage{nicefrac}
\usepackage{amsopn}
\usepackage{tikz}
\usepackage{arydshln}
\usepackage{booktabs}
\usepackage{todonotes}
\definecolor{darkgreen}{rgb}{0, .5, 0}
\usepackage[colorlinks=true,
            linkcolor=darkgreen,
            urlcolor=blue,
            citecolor=darkgreen]{hyperref}

\makeatletter
\def\cl@chapter{}
\makeatother

\usepackage[capitalize]{cleveref}

\crefname{equation}{}{}
\crefname{table}{Table}{Tables}
\crefname{figure}{Figure}{Figures}
\crefname{section}{Section}{Sections}
\crefname{appendix}{Appendix}{Appendix}
\crefname{subsection}{Section}{Sections}
\crefname{subsubsection}{Section}{Sections}

\DeclareMathOperator*{\argmin}{arg\,min}
\DeclareMathOperator{\proj}{proj}
\DeclareMathOperator{\conv}{conv}
\DeclareMathOperator{\ext}{ext}

\newcommand{\bbN}{\mathbb{N}}
\newcommand{\bbR}{\mathbb{R}}
\newcommand{\Z}{\mathbb{Z}}
\newcommand{\lrp}[1]{\left(#1\right)}

\newcommand{\dsum}{\sum}
\newcommand{\R}{\mathbb{R}}

\newcommand{\intsto}[1]{\llbracket #1 \rrbracket}
\newcommand{\even}{\text{even}}
\newcommand{\CPLEX}{\texttt{CPLEX}}
\newcommand{\NN}{\texttt{NN}}
\newcommand{\HL}{\texttt{CDA}}
\newcommand{\CDA}{\texttt{CDA}}
\newcommand{\GUROBI}{\texttt{GRB}}
\newcommand{\GUROBISHIFT}{\texttt{GRB-S}}
\newcommand{\BARON}{\texttt{BRN}}
\newcommand{\BARONSHIFT}{\texttt{BRN-S}}
\newcommand{\NMDT}{\texttt{NMDT}}
\newcommand{\TNMDT}{\texttt{T-NMDT}}
\newcommand{\brackets}[1]{\llbracket #1 \rrbracket}
\renewcommand{\L}{L}

\newcommand{\tred}[1]{\texttt{\textbf{#1}}}

\newcommand{\defeq}{\vcentcolon=}
\newcommand{\eqdef}{=\vcentcolon}

\newcommand \lrbr[1]{\llbracket#1\rrbracket}

\newcommand{\jnoteo}[1]{}

\newcommand{\bnoteo}[1]{}

\newcommand{\rnoteo}[1]{}

\newcommand\colhlight[1]{\tikz[overlay, remember picture,baseline=-\the\dimexpr\fontdimen22\textfont2\relax]\node[rectangle,fill=blue!50,fill opacity = 0.5,text opacity =1] {$#1$};} 
\newcommand\colhlighty[1]{\tikz[overlay, remember picture,baseline=-\the\dimexpr\fontdimen22\textfont2\relax]\node[rectangle,fill=yellow,fill opacity = 0.5,text opacity =1] {$#1$};} 

\newcommand{\floor}[1]{\left\lfloor #1 \right\rfloor}

\newcommand{\IP}{{\text{\fontsize{5}{5}\selectfont IP}}}
\newcommand{\LP}{{\text{\fontsize{5}{5}\selectfont LP}}}

\newcommand{\IPtiny}{{\text{\normalfont \fontsize{6}{6}\selectfont IP}}}
\newcommand{\LPtiny}{{\text{\normalfont \fontsize{6}{6}\selectfont LP}}}
\newcommand{\PLP}{P^{\LPtiny}} 
\newcommand{\PIP}{P^{\IPtiny}} 
\newcommand{\QLP}{Q^{\LPtiny}} 
\newcommand{\QIP}{Q^{\IPtiny}} 

\newcommand{\BHH}[1]{\texttt{NN-R#1}}

\newcommand\hlight[1]{\tikz[overlay, remember picture,baseline=-\the\dimexpr\fontdimen22\textfont2\relax]\node[rectangle,fill=blue!50,rounded corners,fill opacity = 0.2,draw,thick,text opacity =1] {$#1$};} 
\newcommand\hlighty[1]{\tikz[overlay, remember picture,baseline=-\the\dimexpr\fontdimen22\textfont2\relax]\node[rectangle,fill=yellow!50,rounded corners,fill opacity = 0.2,draw,thick,text opacity =1] {$#1$};}

\DeclareMathOperator    \argmax        {arg\,max}

\renewcommand{\tred}[1]{\textcolor{red}{#1}}
\newcommand{\alert}[1]{\tred{#1}}

\journalname{Journal of Global Optimization}

\spdefaulttheorem{hypothesis}{Hypothesis}{}{}
\crefname{hypothesis}{Hypothesis}{Hypotheses}
\spdefaulttheorem{observation}{Observation}{}{}
\crefname{observation}{Observation}{Observation}

\title{Compact mixed-integer programming formulations in quadratic optimization \thanks{This work was supported by AFOSR (grant FA9550-21-0107) and ONR (Grant N00014-20-1-2156). Any opinions, findings, and conclusions or recommendations expressed in this material are those of the authors and do not necessarily reflect the views of the Office of Naval Research or the Air Force Office of Scientific Research.}
}

\date{Received: date / Accepted: date}

\author{Benjamin Beach
\and Robert Hildebrand
\and Joey Huchette}

\institute{Benjamin Beach 
           \and
           Robert Hildebrand \at
              Grado Department of Industrial and Systems Engineering, Virginia Tech\\ 
              \email{\{bben6,rhil\}@vt.edu}
           \and Joey Huchette \at Department of Computational and Applied Mathematics, Rice University \\
           \email{joehuchette@rice.edu}
}

\begin{document}

\maketitle

\begin{abstract}
We present a technique for producing valid dual bounds for nonconvex quadratic optimization problems. The approach leverages an elegant piecewise linear approximation for univariate quadratic functions due to Yarotsky ~\cite{Yarotsky-2016}, formulating this (simple) approximation using mixed-integer programming (MIP). Notably, the number of constraints, binary variables, and auxiliary continuous variables used in this formulation grows logarithmically in the approximation error. Combining this with a diagonal perturbation technique to convert a nonseparable quadratic function into a separable one, we present a mixed-integer convex quadratic relaxation for nonconvex quadratic optimization problems. We study the strength (or \emph{sharpness}) of our formulation and the tightness of its approximation.  Further, we show that our formulation represents feasible points via a Gray code. We close with computational results on problems with quadratic objectives and/or constraints, showing that our proposed method i) across the board outperforms existing MIP relaxations from the literature, and ii) on hard instances produces better bounds than exact solvers within a fixed time budget.
\end{abstract}


\section{Introduction}

We are interested in methods to solve optimization problems with quadratic objectives and/or constraints. Consider the following generic problem with a quadratic objective:
\begin{equation} \label{eqn:generic-problem}
    \min_{x \in X}\quad h(x) \defeq x' Q x + c \cdot x,
\end{equation}
where $X \subseteq \R^n$ is some nonempty feasible region described by side constraints. When the quadratic objective matrix $Q$ is not positive semidefinite, this is a difficult nonconvex optimization problem. We will focus on techniques to (approximately) reformulate nonconvex quadratic functions like the objective of \eqref{eqn:generic-problem}.

Quadratic optimization problems naturally arise in a number of important applications across science and engineering (see~\cite{Furini:2019,Hao:1982} and references therein). In the presence of nonconvexity, such problems are in general very difficult to solve from both a practical and theoretical perspective~\cite{Pardalos:1991}. As a result, there has been a steady stream of research developing new algorithmic techniques to solve quadratic optimization problems, and variants thereof~(see \cite{Burer:2012a} for a survey).

Our approach to approximately solving problems of the form \eqref{eqn:generic-problem} will be to reformulate the objective of \eqref{eqn:generic-problem} using mixed-integer programming (MIP). Given some diagonal matrix $D$, we can equivalently write \cref{eqn:generic-problem} as
\begin{subequations} \label{eqn:generic-problem-D}
\begin{align}
    \min_{x \in X}\quad& h^D(x,y) \coloneqq x'(Q + D)x + c \cdot x - D  y \label{eqn:generic_quad_obj-D}\\
        \text{s.t.} \quad & y_i = x_i^2 \quad\quad  i \in \intsto{n}, \label{eqn:generic_quad_D_eqns}
\end{align}
\end{subequations}
where $\llbracket n \rrbracket \coloneqq \{1,\ldots,n\}$. If $D$ is chosen such that $Q + D$ is positive semidefinite, the quadratic objective will be convex, meaning that all the nonconvexity of this problem has been isolated in the univariate quadratic equations $y_i = x_i^2$. This technique is sometimes called ``diagonal perturbation''~\cite{Dong:2018}.

In this work, we present a compact, tight MIP formulation for the graph of a univariate quadratic term: $\Set{ (x,y) | l \leq x \leq u, \: y = x^2 }$. We derive our formulation by adapting an elegant result of Yarotsky~\cite{Yarotsky-2016}, who shows that there exists a simple neural network function that approximates $y = x^2$ exponentially well (in terms of the size of the network) over the unit interval. The resulting neural network can be interpreted as a function $F_\L : \bbR \to \bbR$ that is build compositionally from a number of simple piecewise linear functions. There is a long and rich strain of research on MIP formulations for piecewise linear functions that serve as approximations for more complex nonlinear functions~\cite{Croxton:2003,Dantzig:1960,Huchette:2017,Lee:2001,Magnanti:2004,Padberg:2000,Vielma:2010}, with recent work focusing particularly on modeling neural networks~\cite{Anderson:2019,Bunel:2019,Serra:2018a,Serra:2018,Tjeng:2017,Huchette-2019}.

We show that this approximation for univariate quadratic terms leads to a \emph{relaxation} for optimization problems with quadratic objectives and/or constraints, meaning that it provides valid dual bounds for the true quadratic problem. We will show that our proposed formulation is \emph{sharp}, meaning that its LP relaxation projects to the convex hull of all feasible points. Further, we show that the formulation is in fact \emph{hereditarily sharp}, meaning that this sharpness property holds throughout the branch and bound tree. The key to reaching this result is connecting the binary reformulation to the \emph{reflected Gray code}, a well-studied binary sequence in electrical engineering.

\subsection{Literature review}
Our approach hews most closely to that of Dong and Luo~\cite{Dong-Luo-2018} and Saxena et al.~\cite{Saxena:2008}. The diagonal perturbation approach we follow have been applied throughout the years in a number of settings; for example, nonconvex quadratic optimization (with or without integer variables) \cite{Billionnet2012,Billionnet2016,Elloumi2019,Galli2014,sven-MIQCQP}, more general nonlinear~\cite{Frangioni:2006,Frangioni:2007} optimization with binary variables, and general nonlinear optimization~\cite{Adjiman:1998a,Adjiman:1998,Androulakis:1995}.

A string of recent work on optimization methods for nonconvex quadratic problems has focused on methods for relaxing bilinear terms using piecewise McCormick envelopes \cite{CastilloCastillo2018,Castro2015c,Castro2015-Chem,Misener2012,Nagarajan:2019,Castro2021}. These piecewise envelopes can be formulated using mixed-integer programming in multiple ways, typically resulting in either a linear- or logarithmic-sized MIP formulation. Moreover, this piecewise relaxation can be refined dynamically to produce a tighter relaxation in a region of interest without resulting in an unduly large MIP formulation~\cite{CastilloCastillo2018,Nagarajan:2019}. In a similar vein, a paper of Galli and Letchford~\cite{Galli2018} presents a binarization heuristic for ``box QP'' problems, leveraging a structural result of Hansen et al.~\cite{hansen}, and compares classical convexification techniques~\cite{Fortet1960,Glover1975,Hammer1970} within the heuristic.

An interesting recent paper of Xia et al.~\cite{Xia2020} reformulates optimization problems with quadratic objectives and linear constraints into MIPs via the KKT conditions. The approach outperforms commercial solvers on certain classes of instances; however, it does not seem to perform favorably on boxQP problems, and in general requires the careful computation of ``big-$M$'' coefficients which may lead to loose LP relaxations.

\subsection{Outline} In \cref{sec:MIP-Formulation} we describe our MIP approximation for $y=x^2$. In \cref{sec:gcodes} we prove some properties of Gray codes that will be useful for proving the results in \cref{sec:fstrength}. In \cref{sec:fstrength}, we show that our formulation is strong (i.e. sharp), and establish the connection between our MIP approximation and the reflected Gray code. In \cref{sec:chull}, we show how to derive some facets of the full convex hull of our MIP approximation, with connections to the parity polytope. In \cref{sec:area-section}, we present a relaxation version of our MIP approximation, derive the total area of the relaxation, and compare against the relaxation of Dong and Luo~\cite{Dong-Luo-2018}. Finally, in \cref{sec:computations}, we numerically compare our relaxation with other competing methods, including other relaxations such as CDA~\cite{Dong-Luo-2018} and NMDT~\cite{Castro2015c}, as well as state-of-the-art solvers with quadratic support like Gurobi, CPLEX, and BARON.

\section{A piecewise-linear approximation for univariate quadratic terms}
\label{sec:MIP-Formulation}
In this section, we present our mixed-integer programming relaxation for \eqref{eqn:generic-problem}. We start by describing the construction of Yarotsky, which is a piecewise linear neural network approximation for the univariate quadratic function $F(x) = x^2$. We then formulate the graph of this piecewise-linear function using mixed-integer programming, and use it to build a tight under-approximation for the quadratic optimization problem \eqref{eqn:generic-problem}.

For ease of notation, for any integers $i \le j$, we define $\lrbr{i,j} \defeq \{i, i+1, \dots, j\}$, and for integers $i \ge 1$ we define $\lrbr{i} \defeq \{1, 2, \dots, i\}$. 
\begin{figure}
    \centering
\begin{tabular}{lr}
        \begin{tikzpicture}
\pgfplotsset{%
    width=0.45\textwidth,
}
\definecolor{color0}{rgb}{0.12156862745098,0.466666666666667,0.705882352941177}
\definecolor{color1}{rgb}{1,0.498039215686275,0.0549019607843137}
\definecolor{color2}{rgb}{0.172549019607843,0.627450980392157,0.172549019607843}

\begin{axis}[
legend cell align={left},
legend columns=3,
legend style={at={(0.5,1.2)}, anchor=north, draw=white!80.0!black},
tick align=outside,
tick pos=left,
x grid style={white!69.01960784313725!black},
xmin=-0.05, xmax=1.05,
xtick style={color=black},
xtick={0,0.125,0.25,0.375,0.5,0.625,0.75,0.875,1},
xticklabels={$0$,$\tfrac18$,$\tfrac14$,$\tfrac38$,$\tfrac12$,$\tfrac58$,$\tfrac34$,$\tfrac78$,$1$},
y grid style={white!69.01960784313725!black},
ymin=-0.05, ymax=1.05,
ytick style={color=black},
ytick={0,0.125,0.25,0.375,0.5,0.625,0.75,0.875,1},
yticklabels={0,1/8,1/4,3/8,1/2,5/8,3/4,7/8,1}
]
\addplot [semithick, color0]
table {%
0 0
0.5 1
1 0
};
\addlegendentry{$G_1$}
\addplot [semithick, color1]
table {%
0 0
0.25 1
0.5 0
0.75 1
1 0
};
\addlegendentry{$G_2$}
\addplot [semithick, color2]
table {%
0 0
0.125 1
0.25 0
0.375 1
0.5 0
0.625 1
0.75 0
0.875 1
1 0
};
\addlegendentry{$G_3$}
\end{axis}

\end{tikzpicture}  
&
        \begin{tikzpicture}
\pgfplotsset{%
    width=0.45\textwidth,
}
\definecolor{color0}{rgb}{0.12156862745098,0.466666666666667,0.705882352941177}
\definecolor{color1}{rgb}{1,0.498039215686275,0.0549019607843137}
\definecolor{color2}{rgb}{0.172549019607843,0.627450980392157,0.172549019607843}
\definecolor{color3}{rgb}{0.83921568627451,0.152941176470588,0.156862745098039}
\definecolor{color4}{rgb}{0.580392156862745,0.403921568627451,0.741176470588235}

\begin{axis}[
legend cell align={left},
legend columns=5,
legend style={at={(0.5,1.2)}, anchor=north, draw=white!80.0!black},
tick align=outside,
tick pos=left,
x grid style={white!69.01960784313725!black},
xmin=-0.05, xmax=1.05,
xtick style={color=black},
xtick={0,0.125,0.25,0.375,0.5,0.625,0.75,0.875,1},
xticklabels={$0$,$\tfrac18$,$\tfrac14$,$\tfrac38$,$\tfrac12$,$\tfrac58$,$\tfrac34$,$\tfrac78$,$1$},
y grid style={white!69.01960784313725!black},
ymin=-0.05, ymax=1.05,
ytick style={color=black},
ytick={0,0.125,0.25,0.375,0.5,0.625,0.75,0.875,1},
yticklabels={0,1/8,1/4,3/8,1/2,5/8,3/4,7/8,1}
]
\addplot [semithick, color0]
table {%
0 0
1 1
};
\addlegendentry{$F_0$}
\addplot [semithick, color1]
table {%
0 0
0.5 0.25
1 1
};
\addlegendentry{$F_1$}
\addplot [semithick, color2]
table {%
0 0
0.25 0.0625
0.5 0.25
0.75 0.5625
1 1
};
\addlegendentry{$F_2$}
\addplot [semithick, color3]
table {%
0 0
0.125 0.015625
0.25 0.0625
0.375 0.140625
0.5 0.25
0.625 0.390625
0.75 0.5625
0.875 0.765625
1 1
};
\addlegendentry{$F_3$}
\addplot [semithick, color4]
table {%
0 0
0.01 0.0001
0.02 0.0004
0.03 0.0009
0.04 0.0016
0.05 0.0025
0.06 0.0036
0.07 0.0049
0.08 0.0064
0.09 0.0081
0.1 0.01
0.11 0.0121
0.12 0.0144
0.13 0.0169
0.14 0.0196
0.15 0.0225
0.16 0.0256
0.17 0.0289
0.18 0.0324
0.19 0.0361
0.2 0.04
0.21 0.0441
0.22 0.0484
0.23 0.0529
0.24 0.0576
0.25 0.0625
0.26 0.0676
0.27 0.0729
0.28 0.0784
0.29 0.0841
0.3 0.09
0.31 0.0961
0.32 0.1024
0.33 0.1089
0.34 0.1156
0.35 0.1225
0.36 0.1296
0.37 0.1369
0.38 0.1444
0.39 0.1521
0.4 0.16
0.41 0.1681
0.42 0.1764
0.43 0.1849
0.44 0.1936
0.45 0.2025
0.46 0.2116
0.47 0.2209
0.48 0.2304
0.49 0.2401
0.5 0.25
0.51 0.2601
0.52 0.2704
0.53 0.2809
0.54 0.2916
0.55 0.3025
0.56 0.3136
0.57 0.3249
0.58 0.3364
0.59 0.3481
0.6 0.36
0.61 0.3721
0.62 0.3844
0.63 0.3969
0.64 0.4096
0.65 0.4225
0.66 0.4356
0.67 0.4489
0.68 0.4624
0.69 0.4761
0.7 0.49
0.71 0.5041
0.72 0.5184
0.73 0.5329
0.74 0.5476
0.75 0.5625
0.76 0.5776
0.77 0.5929
0.78 0.6084
0.79 0.6241
0.8 0.64
0.81 0.6561
0.82 0.6724
0.83 0.6889
0.84 0.7056
0.85 0.7225
0.86 0.7396
0.87 0.7569
0.88 0.7744
0.89 0.7921
0.9 0.81
0.91 0.8281
0.92 0.8464
0.93 0.8649
0.94 0.8836
0.95 0.9025
0.96 0.9216
0.97 0.9409
0.98 0.9604
0.99 0.9801
1 1
};
\addlegendentry{$F$}
\end{axis}

\end{tikzpicture}
\end{tabular}
    \caption{\emph{Left:} The intermediary sawtooth functions $G_i=2^{2i}(F_{i-1}-F_{i})$. \emph{Right:}
    The approximation for $F(x)=x^2$ of Yarotsky by functions $F_i$~\cite[Figure 2]{Yarotsky-2016}.}
    \label{fig:ReLU}
\end{figure}
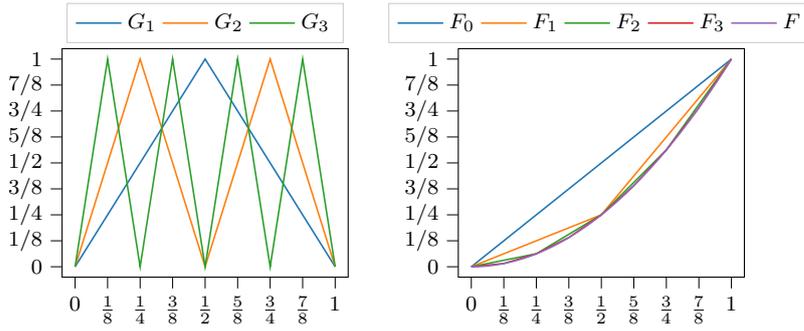
\subsection{The construction of Yarotsky}
\label{ssec:yarotsky}
For fixed $\L \in \mathbb{N}$, we wish to model the function $F_L(x)$, defined as the piecewise linear interpolant to $y=x^2$ on the interval $[0,1]$ at $2^L+1$ uniformly spaced breakpoints:
\begin{equation}
    F_L(x) = \tfrac{2i-1}{N} (x-\tfrac i N)+\tfrac{i^2}{N^2} \quad \text{ if } x \in [\tfrac {i-1}{N},\tfrac{i}{N}] \text{ for some } i \in \lrbr{2^L}.
\end{equation}
Define the \emph{sawtooth functions} $G_i\colon [0,1] \to [0,1]$ as $G_i = 2^i(F_{i-1}-F_i)$. Yarotsky~\cite{Yarotsky-2016} shows that $G_i$ can be defined recursively as 
\begin{subequations} \label{eqn:sawtooth}
\begin{align}
    G_0(x) &= x, \\
    G_i(x) &= 
    \begin{cases} 
        2G_{i-1}(x) & G_{i-1}(x) < 1/2 \\ 
        2(1-G_{i-1}(x)) & G_{i-1}(x) \geq 1/2
    \end{cases} \quad i \in \llbracket \L \rrbracket,
\end{align}
\end{subequations}
and, furthermore, that
\begin{equation} \label{eqn:quadratic-approx}
    F_\L(x) = x - \sum_{i=1}^\L 2^{-2i}G_i(x).
\end{equation}
Yarotsky further shows that $F(x)$ approximates $x^2$ to a pointwise error of $|x^2 - F_L(x)| \le 2^{-2\L-2}$~\cite[Proposition 2]{Yarotsky-2016}.\footnote{Furthermore, Yarotsky~\cite{Yarotsky-2016} observes that it is straightforward to represent each of the sawtooth functions as a composition of the standard ReLU activation function $\sigma(x) = \max\{0,x\}$. For example, $G_1(x) = 2\sigma(x) - 4 \sigma(x-\frac 12) + 2 \sigma(x-1)$. In this way, $F_\L$ can be written as a neural network with a very particular choice of architecture and weight values.} We include an illustration of $G_\L$ and $F_\L$ for different values of $\L$ in \cref{fig:ReLU}(b). Crucially, we will later make use of the fact that $F_\L(x) \geq F(x)$ for each $0 \leq x \leq 1$, i.e. $F_\L$ is an overestimator for $F$.

\subsection{A MIP formulation for $F_\L$} \label{ssec:formulation}
We now turn our attention to constructing a mixed-integer programming formulation for $F_\L$. As \eqref{eqn:quadratic-approx} tells us that $F_\L$ depends linearly on the sawtooth functions $G_i$, we turn our attention to formulating the piecewise-linear equations \eqref{eqn:sawtooth} using MIP. 

For the remainder of the section we will use $g_i$ as decision variables in our optimization formulation corresponding to the output of the $i$-th sawtooth function $G_i$. Therefore, $g_0 = x$, and for each of the other sawtooth functions $G_i$ for $i \in \llbracket \L \rrbracket$, we introduce a binary decision variable $\alpha_i$. Given some input $x$, these binary variables serve to indicate which piece of the sawtooth the input lies on:
\begin{subequations} \label{eqn:binary-implications}
\begin{align}
    \alpha_i &= 0 \Longrightarrow \left(g_i = 2g_{i-1}\right) \wedge \left(0 \leq g_{i-1} \leq 1/2\right) \\
    \alpha_i &= 1 \Longrightarrow \left(g_i = 2(1-g_{i-1})\right) \wedge \left(1/2 \leq g_{i-1} \leq 1\right)
\end{align}
\end{subequations}
Define the set $S_i := \Set{ (g_{i-1},g_i,\alpha_i) \in [0,1] \times [0,1] \times \{0,1\} | \eqref{eqn:binary-implications} }$ for each $i\in \llbracket \L \rrbracket$. It is not difficult to see that a convex hull formulation for $S_i$ is given by
\begin{subequations} \label{eqn:ideal-formulation-one-layer}
    \begin{align}
        2(\alpha_i - g_{i-1}) &\le g_{i} \le 2(1 - g_{i-1}), \label{eq_Sia}\\
        2(g_{i-1} - \alpha_i) &\le g_{i} \le 2 g_{i-1}. \label{eq_Sib} \\
        (g_{i-1},g_i,\alpha_i) &\in [0,1] \times [0,1] \times \{0,1\}.
    \end{align}
\end{subequations}
Chaining these formulations together for each $i$, we construct a MIP formulation for $\mathcal G_\L := \Set{ (x,y) \in [0,1] \times [0,1] | y = F_\L(x) }$, the graph of the neural network approximation function $F_\L$.

\begin{proposition} \label{prop:graph-formulation}
    Fix some $\L \in \bbN$. A MIP formulation for $(x,y) \in \mathcal G_\L$ is
    \begin{equation} \label{eqn:graph-formulation}
    \begin{array}{ll}
        g_0 = x \\
        (g_{i-1}, g_i, \alpha_i) \in S_i & i \in \llbracket \L \rrbracket \\
        y = x - \dsum_{i=1}^\L 2^{-2i}g_i.
    \end{array}
    \end{equation}
\end{proposition}

We emphasize that this formulation is extremely compact: it requires only $\mathcal{O}(L)$ binary variables, auxiliary continuous variables, and constraints. As noted in \cref{ssec:yarotsky}, $F_\L$ approximates $F$ to within $\mathcal{O}(2^{-\L})$ pointwise, which implies that the size of our formulation scales logarithmically in the desired accuracy.

It is a straightforward extension of \cref{prop:graph-formulation} to consider more general interval domains $x \in [l,u]$ on the inputs. In particular, introducing two auxiliary variables $\tilde{x}, \tilde{y} \in [0,1]$, we formulate $\tilde{y} = F(\tilde{x}) = \tilde{x}^2$ using \eqref{eqn:graph-formulation}, and then map them to the $(x,y)$ variables via the linear transformation
\begin{align*}
    x = l + (u-l) \tilde{x}, \, \, \ 
    y = l^2 + 2l(u-l)\tilde{x} + (u-l)^2\tilde{y}.
\end{align*}

\subsection{Tying it all together}
We are now prepared to construct our mixed-integer programming approximation for \eqref{eqn:generic-problem}. For the objective of \eqref{eqn:generic-problem}, compute a nonnegative diagonal matrix $D$ such that $Q + D$ is positive semidefinite.\footnote{This can be accomplished in a number of ways: for example, by computing the minimum eigenvalue of $D$, or by solving a semidefinite programming problem~\cite{Dong-Luo-2018}.} Then, for a given $\L$, the approximation for \eqref{eqn:generic-problem} is:
\begin{subequations} \label{eqn:approx-problem}
    \begin{alignat}{2}
        \min_{x \in X,y}\quad& h^D(x,y) \equiv x'(Q + D)x + c \cdot x - D  y  \\
        \text{s.t.} \quad & (x_i,y_i) \in \mathcal{G}_\L \quad & i \in \llbracket n \rrbracket. \label{eqn:approx-problem-3}
    \end{alignat}
\end{subequations}
Using the formulation \eqref{eqn:graph-formulation} for the constraint \eqref{eqn:approx-problem-3}, this yields a mixed-integer convex quadratic reformulation of the problem (ignoring the potential structure of $X$). This formulation requires at most $nL$ binary variables and $\mathcal{O}(nL)$ auxiliary continuous variables and linear constraints. Furthermore, recall that we may set $L = \mathcal{O}(\log(1/\varepsilon))$ to attain an approximation of accuracy $\varepsilon$ for the equations \eqref{eqn:generic_quad_D_eqns}.

Consider any $\hat{x} \in X$, along with any $\hat{y}$ such that $(\hat{x},\hat{y})$ satisfies \eqref{eqn:approx-problem-3}. Since $F_\L$ overestimates $F$, for each $i \in \llbracket n \rrbracket$ we have $\hat{x}_i^2 \leq \hat{y}_i$. Therefore, $h^D(\hat{x},\hat{y}) \leq h(\hat{x})$. Since there always will exist such a $\hat{y}$ for any $\hat{x} \in X$, \eqref{eqn:approx-problem} offers a valid dual bound on the optimal cost of \eqref{eqn:generic-problem}.

Note that this approach can readily be adapted to handle quadratic constraints. In particular, this transformation will offer a \emph{relaxation} of the quadratically constrained problem. Note that the error bound derived above is with respect to the quadratic constraint that is being relaxed. It may not translate into an error bound on the objective value of the optimization problem, a known phenomena in the global optimization literature~\cite{Dey:2015}.

\section{Gray Codes and Binary Representation}\label{sec:gcodes}
In this section, we introduce the reflected Gray code, and prove some of its useful properties. 

For the remainder of this work, we will work with two notions of expressing integers as vectors in $\{0,1\}^*$.  First, we consider the standard binarization with $L$ bits.  That is, for an integer $i \in \llbracket 0, 2^L-1 \rrbracket$ we define $\bm \beta^i \in \{0,1\}^L$ such that 
\begin{equation}
i = \sum_{j = 1}^{L} 2^{L-j} \beta^i_j.
\label{eq:bindef}
\end{equation}
Next, we define the \emph{reflected Gray code} sequence, which is a sequence of binary representations of integers that is extremely well-studied in electrical engineering and engineering~\cite{Savage:1997}. Notably, each adjacent pair in the sequence differs in exactly one bit. As presented in \cite{Foss1954} and references therein, the $L$-bit reflected Gray code $\bm \alpha^i\in \{0,1\}^L$ representing the integer $i$ can be described by the recursion 
\begin{align}
\label{eq:gray_code_defn}
\alpha^i_1 &= \beta^i_1\\
\alpha^i_j &:= \beta^i_j \oplus \beta^i_{j-1}   & \text{ for all } j = 2, \dots, L,
\end{align}
where we use $\oplus$ to denote addition modulo $2$.  By inverting the relation, we obtain the formula
\begin{align}
\label{eq:alternative-definition}
    \beta^i_j &= \alpha^i_1 \oplus \alpha^i_{2} \oplus \cdots \oplus \alpha^i_{j} & \text{ for } j =1, \dots, L.
\end{align}
In this way, flipping any $\alpha_j$ bit implies that we flip all less significant bits $\beta_k$ for $k \geq j$.  See \cref{fig:reflected_gray_code} for an illustration of how to build the reflected Gray code, which we will henceforth refer to as `the Gray code'.
\begin{figure}
    \centering
    \begin{tabular}{ccc}
$L = 1$ & $L= 2$ & $L = 3$\\
\hline \\
$
\begin{array}{l|c}
\text{Code} \ \ & \text{Number}\\
\hline
\ \ \colhlighty{0} \ \ & 0\\
\ \ \colhlight{1}& 1\\
\end{array}
$
&
$
\begin{array}{l|c}
\text{Code} \ \ & \text{Number}\\
\hline
0\ \ \  \colhlighty{0} \ \ & 0\\
0\ \ \ \colhlighty{1} & 1\\
1\ \ \ \colhlight{1} & 2\\
1\ \ \  \colhlight{0} & 3
\end{array}
$
&
$
\begin{array}{l|c}
\text{Code} \ \ & \text{Number}\\
\hline
0\ \ \  \colhlighty{0}\hspace{0.45cm} \colhlighty{0} \ \ & 0\\
0\ \ \  \colhlighty{0} \hspace{0.45cm} \colhlighty{1} & 1\\
0\ \ \  \colhlighty{1}\hspace{0.45cm} \colhlighty{1} & 2\\
0\ \ \  \colhlighty{1}\hspace{0.45cm}\colhlighty{0} & 3\\
1\ \ \ \colhlight{1}\hspace{0.45cm} \colhlight{0} & 4\\
1\ \ \ \colhlight{1}\hspace{0.45cm} \colhlight{1} & 5\\ 
1\ \ \ \colhlight{0}\hspace{0.45cm} \colhlight{1} & 6\\
1\ \ \ \colhlight{0}\hspace{0.45cm} \colhlight{0}  &7
\end{array}
$
\end{tabular}
    \caption{Building the reflected Gray code.  The reflected Gray code with $L+1$ bits is build from the reflected Gray code on $L$ bits by appending 0s in front of it, then reflecting the Gray code sequence, and then appending 1s in front of it.}
    \label{fig:reflected_gray_code}
\end{figure}

One key property of any Gray code is that successive integer representations differ by only 1 bit, i.e.,
\begin{equation}
\label{eq:one_bit_flip}
    \|\bm \alpha^i - \bm \alpha^{i+1}\|_1  = 1.
\end{equation}
That is, only one bit changes between adjacent binary vectors in the sequence.  We show a similar property holds if we restrict the set of integers we work with by fixing some of the bits in the Gray code vector. To help prove this property, we note the following well-known property of the reflected Gray code in this work.

\begin{lemma}\label{lem:gray_code_rec_def}
For each $i \in \lrbr{0,2^L-1}$, let $\bm{ \tilde{\alpha}}^i$ be the $L$-bit Gray code for $i$, and let $\bm \alpha^j$ be the $L+1$-bit Gray code for some $i \in \lrbr{0, 2^{L+1}-1}$.
\begin{enumerate}
    \item If $j \in \lrbr{0,2^L-1}$, then $\bm \alpha^j = [0,\bm{\tilde \alpha}^j].$
    \item If $j \in \lrbr{2^L,2^{L+1}-1}$, then $\bm \alpha^j = [1,\bm{\tilde \alpha}^{i}]$, where $i = 2^{L+1}-j-1$.
\end{enumerate}
\end{lemma}
\begin{proof}
	First, for each $i \in \lrbr{0,2^L-1}$, let $\bm{\tilde \beta}^i$ be the corresponding $L$-bit binarization. Similarly, for each $j \in \lrbr{0,2^{L+1}-1}$, and let $\bm \beta^j$ be the corresponding $L+1$-bit binarization. Then, by \cref{eq:gray_code_defn} have that $\alpha^j_1 = \beta^j_1=0$ and $\bm \beta^i = [0, \bm{\tilde \beta}^i]$. Applying \cref{eq:gray_code_defn} recursively, this yields $\bm \alpha^i = [0, \bm{\tilde \alpha}^i ]$, as desired.
    
Now, let $i \in \lrbr{2^L,2^{L+1}-1}$, and let $\tilde{i} = 2^{L+1}-i-1$. Since $\tilde{i} \in \lrbr{0,2^L-1}$ we  have as before that $\bm \alpha^{\tilde i} = [0, \bm{\tilde \alpha}^{\tilde i}].$ We wish to show that $\bm \alpha^{i} = [1, \bm{\tilde \alpha}^{\tilde i}].$ Now note that
\begin{align*}
	\tilde{i}& = 2^{L+1}-i-1 = 2^{L+1}-\dsum_{j=1}^{L+1}2^{L+1-j} \beta_j -1 \\
    &=  2^{L+1} - 1 - \dsum_{j=1}^{L+1}2^{L+1-j} + \dsum_{j=1}^{L+1}2^{L+1-j} (1-\beta_j ) \\
    &= \dsum_{j=1}^{L+1}2^{L+1-j} (1-\beta_j )
\end{align*}
That is, in the binarization for $\tilde{i}$, we have $\bm \beta^{\tilde i} = \bm \beta^i \oplus [1,\dots,1]$, so that every bit has been flipped. Observing \cref{eq:alternative-definition}, we see that this can be induced by enforcing $\alpha_1^{\tilde i} = 1-\alpha_1^i$, with all other $\alpha_j^{\tilde i} = \alpha_j^i$: flipping the first $\alpha$-bit induces a flip in all $\beta$-bits. Thus, we obtain that $\bm \alpha^{i} = [1, \bm{\tilde \alpha}^{\tilde i}]$, as desired.
\end{proof}

\begin{lemma} \label{lem:gcode_res}
Let  $J \subseteq \lrbr{L}$ and $\overline {\bm \alpha} \in \{0,1\}^J$.  Let $X = \{x \in \lrbr{0,2^L-1} : \bm \alpha^x_J = \bm {\bar\alpha}\}$. We will write $X$ as $X=\{x_1,\dots,x_t\}$, ordered such that $x_j < x_{j+1}$. Let $I = \llbracket L \rrbracket \setminus J$. Then $\bm \alpha_I^{x_j}$ is a reflected Gray code for the indices $j$ over $X$. That is, for any $j \in 1, \dots, t$, we have
\begin{equation}
    \|{\bm \alpha_I}^{x_j} - {\bm \alpha_I}^{x_{j+1}}\|_1 = 1. \label{eq:rgc_fix_1}
\end{equation}
Furthermore, if $|I| \ge 1$, there exists a $\bm \gamma \in \{0,1\}^{|I|}$ such that, for all $j \in \lrbr{0,t}$, we have
\begin{equation}
    \bm \alpha_I^{x_j} \oplus \bm \gamma = \bm \alpha^j_{\lrbr{L-|I|+1,L}}.\label{eq:rgc_fix_2}
\end{equation}
That is, the modified Gray code after fixing some bits is the original reflected Gray code on $|I|$ bits, with some bits flipped.
\end{lemma}
\begin{proof}
    We will prove this by induction on $L$. To enable the use of \cref{lem:gray_code_rec_def}, we will also prove that, if $|I| \ge 1$, then for all $j \in \lrbr{0,t}$, $i=t-j$, we have 
    \begin{equation} \label{eq:rgc_fix_3}
        x^j = 2^{L} - x^i - 1.
   \end{equation}
    \textbf{Base case:} $L=1$
    
    For $L=1$, the possibilities are trivial, as there is only one bit. If we do not fix the bit, then we have $\alpha^j=[j]$ for each $j \in \{0,1\}$; this sequence of two vectors is trivially a Gray code sequence. This yields the original reflected Gray code for $L=1$, and so $\bm \gamma = [0]$. Finally, we have $t=1$, and $x_j = j$, yielding, for $i=1-j$, $x_j = 1-x_i = 2^1 - x^i - 1$, as required.
    
    On the other hand, if we do fix the bit, then there are no pairs of consecutive bits, and \cref{eq:rgc_fix_1} holds by default. In this case, we have $|I|=0$, so that the other results do not apply.
    
    \textbf{Inductive step}: 
    
    Let $L=k$, and suppose the desired properties hold for $L=k-1$. Let $J$, $\overline{\bm \alpha}$ be a choice of fixed bits for $L=k$. First, note that if $|J| = L$, then all bits are fixed and the \cref{eq:rgc_fix_1} holds by default, while the others do not apply, as $|I|=0$.
    
    Next, suppose $I = \{1\}$, so that the newly added bit is the first unfixed bit. Then we have $t=1$, and $\bm \alpha^{x_j} = [j, \bm{\bar \alpha}]$, with $\bm \alpha^{x_j}_I = [j]$. Thus, defining $\bm \gamma = [0]$, we have that \cref{eq:rgc_fix_1} and \cref{eq:rgc_fix_2} hold trivially. Finally, by \cref{lem:gray_code_rec_def} and the uniqueness of the $L$-bit Gray code for $j$, we have that $x_0 = 2^L - x_1 - 1$ and $x_1 = 2^L - x_0 - 1$, as required.
    
    Otherwise, suppose $|J| \le L-1$, with $I \neq \{1\}$. Then there are three cases: $1 \notin J$, or $1 \in J$ and either $\bar{\alpha}_1=1$ or $\bar{\alpha}_1=0$. Regardless of this choice, the corresponding choices $\tilde{J}$ and $\tilde{\bm \alpha}$ for $L=k-1$ can be attained by defining $\tilde{J} = J \setminus \{1\}$, and defining $\tilde{I}$ and $\tilde{\bm \alpha}$ accordingly. Consider the corresponding sequence $\tilde{X}$ for $L=k-1$. Then, by the inductive hypothesis, the desired results hold for $\tilde{X}$, and as $|\tilde{I}| \ge 1$, we can define a corresponding vector $\tilde{\bm \gamma} \in \{0,1\}^{|\tilde{I}|}$ so that \cref{eq:rgc_fix_2} holds.
    
    \textbf{Case 1}: $1 \in J$ with $\bar{\alpha}_1=0$. In this case, define $\tilde{J} = J \setminus \{1\}$ and $\tilde \alpha = \bar \alpha_{\lrbr{2,L}}$, and define $\tilde{X}$ and $\tilde{\bm \gamma}$ accordingly, noting that $|X| = |\tilde{X}| = t=2^{k-|J|-1}-1$. Let $j \in \lrbr{t}$. Then we have by \cref{lem:gray_code_rec_def} that ${\bm \alpha}^{x_j} = [0,{\tilde{\bm \alpha}}^{\tilde{x}_j}]$, so that ${\bm \alpha_I}^{x_j} = {\tilde{\bm \alpha}_I}^{\tilde{x}_j}$. Thus, \cref{eq:rgc_fix_1,eq:rgc_fix_2,eq:rgc_fix_3} hold by the induction hypothesis, with $\bm \gamma = \tilde{\bm \gamma}$.
    
    \textbf{Case 2}: $1 \in J$ with $\overline{\alpha}_1=1$. In this case, define $\tilde{J} = J \setminus \{1\}$ and $\tilde \alpha = \bar \alpha_{\lrbr{2,L}}$, and define $\tilde{X}$ accordingly, noting that $|X| = |\tilde{X}| = t$. Let $j \in \lrbr{t}$. Then, since $t-j=2^{k-|J|-1}-j-1$, we have by \cref{lem:gray_code_rec_def} that ${\bm \alpha}^{x_j} = [1,{\tilde{\bm \alpha}}^{\tilde{x}_{t-j}}]$, so that ${\bm \alpha_I}^{x_j} = {\tilde{\bm \alpha}}_I^{\tilde{x}_{t-j}}$. That is, the sequence of ${\bm \alpha}^{x_j}_I$'s the sequence of ${\tilde{\bm \alpha}}_I^{\tilde{x}_{j}}$'s, but in reversed order. Thus, \cref{eq:rgc_fix_1,eq:rgc_fix_3} hold by the induction hypothesis, as reversing the order of a sequence has no impact on results for consecutive or centrally reflected terms. Furthermore, by \cref{lem:gray_code_rec_def} and the induction hypothesis, we have that reversing the order of the sequence corresponds with flipping only the first bit $\alpha^{x_j}_1$, so that we can define $\bm \gamma = \tilde{\bm \gamma} \oplus [1, 0, \dots, 0]$ to obtain \cref{eq:rgc_fix_2}.
    
    \textbf{Case 3}: $1 \notin J$, so that the first bit is unfixed. In this case, define $\tilde{J} = J$ and $\tilde{\bm \alpha} = \bar{\bm \alpha}$, and define $\tilde{X}$ accordingly. Then $\tilde{t}=|\tilde{X}| = 2^{k-|J|-1}$. Let $j \in \lrbr{0,t}$ and let $i = t-j = 2^{k-|J|} - j - 1$, so that $j = 2^{k-|J|} - i - 1$ Then, by \cref{lem:gray_code_rec_def}, we can construct $X$ as follows:
    
    \begin{enumerate}
    \item If $j \in \lrbr{0,2^{k-|J|-1}-1}$, then ${\bm \alpha}^{x_j} = [0,{\tilde{\bm \alpha}}^{\tilde{x}_j}]$
    \item If $j \in \lrbr{2^{k-|J|-1},2^{k-|J|}-1}$, then ${\bm \alpha}^{x_j} = [1,{\tilde{\bm \alpha}}^{\tilde{x}_{i}}]$.
    \end{enumerate}
    
     \sloppy This yields \cref{eq:rgc_fix_3} immediately. Further, define $\bm \gamma = [0, \tilde{\bm \gamma}]$. Then for $j \in \lrbr{0,2^{k-|J|-1}-1}$, we have 
    $$\bm \alpha_I^{x_j} \oplus \bm \gamma = [0,{\tilde{\bm \alpha}}_I^{\tilde{x}_j}] \oplus [0, \tilde{\bm \gamma}] =  [0,{\tilde{\bm \alpha}_I}^{\tilde{x}_j} \oplus {\tilde{\bm \gamma}}] = [0, \bm \alpha^j_{\lrbr{L-|I|+2,L}}] = \bm \alpha^j_{\lrbr{L-(|I|)+1,L}},$$
    yielding \cref{eq:rgc_fix_2}. For $j \in \lrbr{2^{k-|J|-1},2^{k-|J|-1}-1}$, defining $i = t-j = 2^{k-|J|} - j - 1 \in \lrbr{\tilde{t}+1,t}$, we have by \cref{lem:gray_code_rec_def} that
    $$\bm \alpha_I^{x_j} \oplus \bm \gamma = [1,{\tilde{\bm \alpha}}_I^{\tilde{x}_i}] \oplus [0, \tilde{\bm \gamma}] =  [1,{\tilde{\bm \alpha}_I}^{\tilde{x}_i} \oplus {\tilde{\bm \gamma}}] = [1, \bm \alpha^i_{\lrbr{L-|I|+2,L}}] = \bm \alpha^j_{\lrbr{L-(|I|)+1,L}},$$
    again yielding \cref{eq:rgc_fix_2}.
    
    Now, \cref{eq:rgc_fix_1} trivially holds for all $j \neq 2^{k-|J|-1}-1$ as in cases 1 and 2, since the first bits of ${\bm \alpha}^{x_j}$ and ${\bm \alpha}^{x_{j+1}}$ match, and exactly one other bit differs by the induction hypothesis. Otherwise, if $j = 2^{k-|J|-1}-1$, then $i = 2^{k-|J|}-(2^{k-|J|-1}-1)-1 = 2^{k-|j|} = j+1$, and thus $x_j$ and $x_{j+1}$ differ by exactly the first bit $\alpha_1$ as indicated above. Thus, \cref{eq:rgc_fix_1} holds.
    From these cases, \cref{eq:rgc_fix_1,eq:rgc_fix_2,eq:rgc_fix_3} hold by induction. 
    \end{proof}
 
Next we show one more property of the code that occurs when extending the Gray code by 1 bit.
\begin{proposition}
\label{lem:gray_code_double}
For an integer $i \in \llbracket 0, 1, \dots, 2^{L} -1\rrbracket $, let $\bm{\alpha}^i$ and $\bm{\beta}^i$ be the $L$-bit Gray code and binary representations of $i$.
Let $\bm{\tilde \alpha}^{2i}$ and $\bm{\tilde \alpha}^{2i+1}$ be the $(L+1)$-bit Gray codes of the integers $2i$ and $2i+1$.

Then
\begin{align*}
    \bm{\tilde \alpha}^{2i} = \left[\alpha^i_1, \dots, \alpha^i_L, \beta^i_L\right]\ \text{ and } \ 
        \bm{\tilde \alpha}^{2i+1} = \left[\alpha^i_1, \dots, \alpha^i_L, 1-\beta^i_L\right].
\end{align*}

\end{proposition}
\begin{proof}
First, note that
\begin{align*}
    \bm{\tilde \beta}^{2i} = \left[\beta_1, \dots, \beta_L, 0\right] \ and \ 
    \bm{\tilde \beta}^{2i+1} = \left[\beta_1, \dots, \beta_L, 1\right].
\end{align*}
Then
\begin{equation}
\tilde \alpha^{2i}_{L+1} = \tilde\beta^{2i}_{L+1} \oplus  \tilde\beta^{2i}_{L} =0 \oplus \beta^{i}_{L} = \beta^{i}_{L}
\  \text{ and } \ 
\tilde \alpha^{2i+1}_{L+1} = \tilde \beta^{2i}_{L+1} \oplus  \tilde \beta^{2i}_{L} =1 \oplus \beta^{i}_{L} = 1-\beta^{i}_{L}.
\end{equation}

Furthermore, $\tilde \beta^{2i}_j = \beta^{2i+1}_j = \beta^i$ for all $j= 1, \dots, L$, by definition, we have that $\tilde \alpha^{2i}_j = \tilde \alpha^{2i+1}_j = \alpha^i_j$ for all $j=1, \dots, L$. 
\end{proof}

\section{Formulation Strength}\label{sec:fstrength}
The \emph{strength} of a MIP formulation is a commonly used metric to assess its potential computational performance. We will work with the three following notions of strength.
\begin{definition}
    Consider a set $U \subseteq \R^n$. For a formulation $P^{\IPtiny}=\{(\bm u,\bm v,\bm z) \in \PLP : \bm z \in \Z^L\}$, where $\PLP \subseteq \R^{n+d+L}$ and $\proj_{\bm u}(P^{\IPtiny}) = U$, we say the the formulation $\PLP$ is
    \begin{itemize}
    	 \item \emph{sharp} if $\proj_{\bm u}(\PLP) = \conv(U)$.
         \item \emph{hereditarily sharp} if, for all $I \subseteq \lrbr{L}$ and $\bm{\bar{z}}_I \in \Z^{|I|}$, we have $$\proj_{\bm u}(\PLP|_{\bm{\bar{z}}_I = \bm{z}_I}) = \conv(\{{\bm u} \in U : (\bm u,\bm v,\bm z) \in \PLP\}|_{\bm{\bar{z}}_I = \bm{z}_I}).$$
         \item \emph{ideal} if $\proj_{\bm{z}}(\ext(\PLP)) \subseteq \Z^L$.
    \end{itemize}
\end{definition}
These definitions closely follow those in \cite{Huchette:2018}, except we define hereditary sharpness explicitly in terms of the current branch.

In this section, we explore the strength of our MIP formulation \cref{eqn:graph-formulation}. We also draw an interesting connection between how our formulation represents feasible points through a Gray code: in essence, feasible points are represented in the formulation by their $L$ most significant digits in a binary expansion. 

For our particular problem, define
\begin{equation}
        \begin{array}{rl}
        	\PLP &\defeq \{(x, y, \bm \alpha, \bm g) \in [0,1] \times \R^+ \times [0,1]^L \times [0,1]^{L+1} :\cref{eqn:graph-formulation}\}\\
        	\PIP &\defeq \{(x, y, \bm \alpha, \bm g) \in [0,1] \times \R^+ \times \{0,1\}^L \times [0,1]^{L+1} :\cref{eqn:graph-formulation}\}
        \end{array}
        \end{equation}
and
\begin{equation}
\begin{array}{rl}
    		\QLP &\defeq \proj_{x,y}(\PLP)\\
    \QIP &\defeq \proj_{x,y}(\PIP).
    \end{array}
    \label{eq:sharpness_sets}
\end{equation}

First, we show that $\QIP$ is, unfortunately, not ideal.

\begin{example}
    The formulation $\PIP$ approximating $y=x^2$ is not ideal.
\end{example}
\begin{proof}
Consider $L=2$, $x=0.25$, $\bm \alpha = [\tfrac 12, 1]$, $\bm g = [\tfrac 12, 1]$, and $y = 0.25 - 2^{-2}(\tfrac 12) - 2^{-4}(0.25) = \tfrac{11}{64}$. This point is chosen to maximize $g_2$ along a facet $g_2 \le 2 \cdot (2_x-\alpha_1)$ of the convex hull. It is an extreme point because it is incident with six facets: $g_2 \le 1$, $\alpha_2 \le 1$, $g_2 \le 2g_1$, $g_1 \le 2x$, $\alpha_1-x \le g_2$, and $y = x - 2^{-2}g_1 + 2^{-4}g_2$. Thus, $\PLP$ has a fractional extreme point, and so is not ideal.
\end{proof}
Next, we will show that $\QIP$ is sharp. To assist deriving this result, we define the generic sawtooth function $G\colon~\R \rightarrow \R$ as
\begin{equation} \label{eq:gdef}
    g_i = G(g_{i-1}) \defeq \begin{cases}
        2g_{i-1} & g_{i-1} < \tfrac{1}{2},\\
        2(1-g_{i-1}) & g_{i-1} \ge \tfrac{1}{2}.
    \end{cases}    
\end{equation}
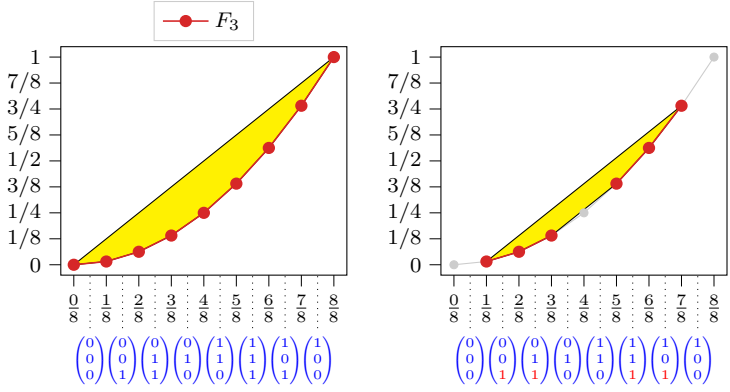
\begin{figure}
    \centering
    \begin{tikzpicture}
\pgfplotsset{%
    width=0.45\textwidth,
}
\definecolor{color0}{rgb}{0.12156862745098,0.466666666666667,0.705882352941177}
\definecolor{color1}{rgb}{1,0.498039215686275,0.0549019607843137}
\definecolor{color2}{rgb}{0.172549019607843,0.627450980392157,0.172549019607843}
\definecolor{color3}{rgb}{0.83921568627451,0.152941176470588,0.156862745098039}
\definecolor{color4}{rgb}{0.580392156862745,0.403921568627451,0.741176470588235}

\begin{axis}[
legend cell align={left},
legend columns=5,
legend style={at={(0.5,1.2)}, anchor=north, draw=white!80.0!black},
tick align=outside,
tick pos=left,
x grid style={white!69.01960784313725!black},
xmin=-0.05, xmax=1.05,
xtick style={color=black},
xtick={0,0.125,0.25,0.375,0.5,0.625,0.75,0.875,1},
xticklabels={$\tfrac08$,$\tfrac18$,$\tfrac28$,$\tfrac38$,$\tfrac48$,$\tfrac58$,$\tfrac68$,$\tfrac78$,$\tfrac88$},
extra x tick style = {blue,font=\tiny, dotted,},
extra x ticks = {0.0625, 0.1875, 0.3125, 0.4375, 0.5625, 0.6875, 0.8125, 0.9375},
extra x tick labels = {
$\begin{pmatrix}0\\0\\0\end{pmatrix}$,
$\begin{pmatrix}0\\0\\1\end{pmatrix}$,
$\begin{pmatrix}0\\1\\1\end{pmatrix}$,
$\begin{pmatrix}0\\1\\0\end{pmatrix}$,
$\begin{pmatrix}1\\1\\0\end{pmatrix}$,
$\begin{pmatrix}1\\1\\1\end{pmatrix}$,
$\begin{pmatrix}1\\0\\1\end{pmatrix}$,
$\begin{pmatrix}1\\0\\0\end{pmatrix}$},
extra tick style={tickwidth=20,tick style = red},
y grid style={white!69.01960784313725!black},
ymin=-0.05, ymax=1.05,
ytick style={color=black},
ytick={0,0.125,0.25,0.375,0.5,0.625,0.75,0.875,1},
yticklabels={0,1/8,1/4,3/8,1/2,5/8,3/4,7/8,1}
]
\draw[fill = yellow] 
(0, 0) --
(0.125,  0.015625) --
(0.25,  0.0625) --
(0.375,  0.140625) --
(0.5,  0.25) --
(0.625,  0.390625) --
(0.75,  0.5625) --
(0.875,  0.765625) --
(1,  1) --
(0, 0);

\addplot [semithick, color3, mark=*,mark options={scale=1.0,color3}]
table {%
0 0
0.125 0.015625
0.25 0.0625
0.375 0.140625
0.5 0.25
0.625 0.390625
0.75 0.5625
0.875 0.765625
1 1
};
\addlegendentry{$F_3$}

\end{axis}

\end{tikzpicture}  \ \  \begin{tikzpicture}
\pgfplotsset{%
    width=0.45\textwidth,
}
\definecolor{color0}{rgb}{0.12156862745098,0.466666666666667,0.705882352941177}
\definecolor{color1}{rgb}{1,0.498039215686275,0.0549019607843137}
\definecolor{color2}{rgb}{0.172549019607843,0.627450980392157,0.172549019607843}
\definecolor{color3}{rgb}{0.83921568627451,0.152941176470588,0.156862745098039}
\definecolor{color4}{rgb}{0.580392156862745,0.403921568627451,0.741176470588235}

\begin{axis}[
legend cell align={left},
legend columns=5,
legend style={at={(0.5,1.2)}, anchor=north, draw=white!80.0!black},
tick align=outside,
tick pos=left,
x grid style={white!69.01960784313725!black},
xmin=-0.05, xmax=1.05,
xtick style={color=black},
xtick={0,0.125,0.25,0.375,0.5,0.625,0.75,0.875,1},
xticklabels={$\tfrac08$,$\tfrac18$,$\tfrac28$,$\tfrac38$,$\tfrac48$,$\tfrac58$,$\tfrac68$,$\tfrac78$,$\tfrac88$},
extra x tick style = {blue,font=\tiny, dotted,},
extra x ticks = {0.0625, 0.1875, 0.3125, 0.4375, 0.5625, 0.6875, 0.8125, 0.9375},
extra x tick labels = {
$\begin{pmatrix}0\\0\\0\end{pmatrix}$,
$\begin{pmatrix}0\\0\\\alert{1}\end{pmatrix}$,
$\begin{pmatrix}0\\1\\\alert{1}\end{pmatrix}$,
$\begin{pmatrix}0\\1\\0\end{pmatrix}$,
$\begin{pmatrix}1\\1\\0\end{pmatrix}$,
$\begin{pmatrix}1\\1\\\alert{1}\end{pmatrix}$,
$\begin{pmatrix}1\\0\\\alert{1}\end{pmatrix}$,
$\begin{pmatrix}1\\0\\0\end{pmatrix}$},
extra tick style={tickwidth=20,tick style = red},
y grid style={white!69.01960784313725!black},
ymin=-0.05, ymax=1.05,
ytick style={color=black},
ytick={0,0.125,0.25,0.375,0.5,0.625,0.75,0.875,1},
yticklabels={0,1/8,1/4,3/8,1/2,5/8,3/4,7/8,1}
]
\addplot [gray!40,mark=*,mark options={scale=0.75}]
table {%
0 0
0.125 0.015625
0.25 0.0625
0.375 0.140625
0.5 0.25
0.625 0.390625
0.75 0.5625
0.875 0.765625
1 1
};

\draw[fill = yellow] 
(0.125,  0.015625) --
(0.25,  0.0625) --
(0.375,  0.140625) --
(0.625,  0.390625) --
(0.75,  0.5625) --
(0.875,  0.765625) --
(0.125,  0.015625);

\addplot [semithick, color3, mark=*,mark options={scale=1.0,color3}]
table {%
0.125 0.015625
0.25 0.0625
0.375 0.140625
};
\addplot [semithick, color3, mark=*,mark options={scale=1.0,color3}]
table {%
0.625 0.390625
0.75 0.5625
0.875 0.765625
};
\end{axis}

\end{tikzpicture}
    \caption{\emph{Left:} The piecewise linear approximation $\QIP$ in red and the linear relaxation $\QLP$ in yellow filled in.  The formulation is sharp because $\QLP$ is the convex hull of $\QIP$. The vectors $\bm \alpha \in \{0,1\}^3$ below are the corresponding binary $\alpha$ variables for any $x$ value on that interval $(\tfrac{i}{8}, \tfrac{i+1}{8})$. \emph{Right:}  The branch $\QIP|_{\alpha_3 = 1}$ in red and the linear relaxation $\QLP|_{\alpha_3 = 1}$ in yellow.  This demonstrates hereditary sharpness since it holds that $\QLP|_{\alpha_3 = 1}$ is the convex hull of  $\QIP|_{\alpha_3 = 1}$.  }
    \label{fig:sharpness}
\end{figure}
\begin{theorem}
\label{thm:sharp}
    The formulation $\PIP$ is sharp for $\QIP$.
\end{theorem}
\begin{proof}
For sharpness, we wish to show that $\QLP=\conv(\QIP)$. Clearly $\QLP \supseteq \conv(\QIP)$ from validity of our formulation; therefore, we focus on showing that $\QLP \subseteq \conv(\QIP)$. We start by fixing some $\bar{x} \in [0,1]$. The result then follows if we can show that the ``slice'' of $\QLP$ at $\bar{x}$, $\QLP|_{x = \bar x} \coloneqq \Set{y | (\bar{x},y) \in \QLP}$, is contained in $\proj_{x}(\conv(\QIP))$. Since $\QLP|_{x = \bar{x}} \subset [0,1]$ and is convex, it suffices to show that both its maximum value and minimum value are contained in $\conv(\QIP)$.

\sloppy First, let $y^* = \max \{ y \in \QLP|_{x = \bar{x}}\}$. From $\QLP$, we know that $y = x - \sum_{i=1}^L 2^{-2i} g_i \leq x$ since $g_i \geq 0$ for all $i$.  Hence, $y^* \leq \bar x$. Next, we observe that, as $(0,0),(1,1) \in \QIP$, convexity in turn implies that $(\bar{x},\bar{x}) \in \conv(\QIP) \subseteq \QLP$. Therefore, $y^* = \bar{x}$ by maximality, and so  $(\bar{x},y^*) \in \conv(\QIP)$. 

Next, let $y^* = \min \{ y \in \QLP|_{x = \bar{x}}\}$. From definition, it follows that there is some $\bm g^*$ and ${\bm \alpha}^*$ such that $(\bar x, y^*, \bm{g}^*, \bm{\alpha}^*) \in \PLP$. Define $G$ as in \cref{eq:gdef}. If $g^*_i =G(g^*_{i-1})$ for each $i \in \llbracket L \rrbracket$, then we immediately conclude that there exists some $\tilde{\bm \alpha} \in \{0,1\}^L$ such that $(\bar{x},y^*,\bm g^*, \bm \tilde{\bm \alpha}) \in P^I$. This, in turn, implies that $(\hat{x},y^*) \in \QIP$.

Otherwise, take $i \in \llbracket L \rrbracket$ as the largest index such that $g^*_i > G(g^*_{i-1})$. Then, recursively define $\tilde{\bm g}$ such that $\tilde{g}_j = \begin{cases} g^*_j & j < i \\ G(\tilde{g}_{j-1}) & j \geq i \end{cases}$ for each $j \in \{0,\ldots,\L\}$. Further, take $\tilde{y} = \bar{x} - \sum_{i=1}^L 2^{-2i}\tilde{g}_i$, with $\tilde{\alpha}_i = g_{i-1}$ for all $i$, inducing only lower bounds of $0$ on all $g_j$. Then $(\tilde y, \bar x, \bm{\tilde g}, \bm {\tilde \alpha}) \in \QLP$.  We now show that $\tilde{y} < y^*$, contradicting the minimality of $y^*$.

Let $\varepsilon = \tilde{g}_i - g^*_i$.  Note that $G\colon \R \to \R$ is Lipschitz continuous with Lipschitz constant 2.  That is, for any $g,g' \in [0,1]$, we have that $|G(g) - G(g')| \leq 2|g - g'|$.
Hence, since $|g^*_i-\tilde{g}_i| \leq \varepsilon$ we conclude inductively that $|g^*_{i+k}-\tilde{g}_{i+k}| \leq 2^k\varepsilon$ for each $k \in \llbracket L-i \rrbracket$. Thus, we have
\begin{align*}
y^* - \tilde{y} &= \sum_{j=i}^L 2^{-2j} (\tilde{g}_j - g_j^*) \\
&\ge 2^{-2i}(\tilde{g}_i - g^*_i) + \sum_{j=i+1}^L 2^{-2j} |g_j^* - \tilde{g}_j| \\
&\ge 2^{-2i}\varepsilon - \textstyle \sum_{j=i+1}^L 2^{-2j} 2^{j-i} \varepsilon \\
&= 2^{-2i}\varepsilon(1 - 2^i \sum_{j=i+1}^L 2^{-j}) \\
&= 2^{-2i}\varepsilon(1 - 2^{i} (2^{-i} - 2^{-L})) \\
&= \varepsilon(2^{-i-l}) > 0.
\end{align*}
Therefore, $\tilde{y} < y^*$, contradicting the minimality of $y^*$. From this, we conclude that no such $i$ exists, completing the proof.
\end{proof}

We now show the correspondence between our formulation and the Gray code. It is helpful to note that, in $\PIP$, we have $g_i = G(g_{i-1})$, where $G$ is as defined in \cref{eq:gdef}.
\begin{theorem}[The MIP formulation follows a Gray code]\label{thm:rgc}
    Take $(x,\bm g,\bm \alpha) \in \proj_{x,\bm g, \bm \alpha}(\PIP|_{g_L \in (0,1)})$ for some fixed value $g_L \in (0,1)$. Fix some $j \in \lrbr{0,L-1}$, and define $i_j \coloneqq \floor{2^{L-j} g_j}$. Then $[\alpha_{j+1}, \dots, \alpha_L]$ is the $L-j$-bit Gray code $\bm{\alpha}^{i_j}$ for $i_j$, and $g_j \in (\tfrac{i_j}{2^{L-j}}, \tfrac{i_j+1}{2^{L-j}})$.
\end{theorem}
\begin{proof}
    First, we note that, for each $j \le L+1$, we have
    $$
        g_{j} = \begin{cases}
            \tfrac{1}{2}g_{j+1} & \alpha_{j+1} = 0,\\
            1-\tfrac{1}{2}g_{j+1} & \alpha_{j+1} = 1.
        \end{cases}
    $$
    
    We will proceed by induction in a manner similar to the proof for \cref{lem:gcode_res}. 
    
    \textbf{Base Case:} $j=L-1$
    
    In this case, if $\alpha_L = 0$, then $g_{L-1} = \tfrac{1}{2} g_L \in (0,\tfrac{1}{2})$, and so $i_{L-1} = \floor{2 g_{L-1}} = 0$, whose 1-bit Gray code is $[0] = [\alpha_L]$, as desired. On the other hand, if $\alpha_L = 1$, then $g_{L-1} = 1-\tfrac{1}{2} g_L \in (\tfrac{1}{2},1)$, and so $i_{L-1} = \floor{2 g_{L-1}} = 1$, whose 1-bit Gray code is $[1] = [\alpha_L]$, as desired.
    
    \textbf{Inductive step}: Let $j \in \lrbr{0,L-2}$ and suppose the  statement holds $j+1$. 
    
    If $\alpha_{j+1} = 0$, then $g_j = \tfrac{1}{2} g_{j+1} \in \left(\tfrac{i_{j+1}}{2^{L-j}},\tfrac{i_{j+1}+1}{2^{L-j}}\right)$ and $2^{L-j}g_j \in (i_{j+1},i_{j+1}+1)$. In this case, we have $i_j = \floor{2^{L-j} g_j} = i_{j+1} \in \lrbr{0,2^{L-j-1}-1}$. Thus, by \cref{lem:gray_code_rec_def}, we have that the $L-j$-bit Gray code for $i_j$ is given by $\bm \alpha^{i_j} = [0 ~ \bm \tilde{\bm \alpha}^{i_{j+1}}]$, where $\tilde{\bm \alpha}^{i_j}$ is the ($L-j-1$)-bit Gray code for $i_j$, which is $[\alpha_{j+2} \dots \alpha_L]$ by the induction hypothesis. This yields $\bm \alpha^{i_j} = [0 ~ \alpha_{j+2} \dots \alpha_L] = [\alpha_{j+1} \dots \alpha_L]$ as desired. Further, observing the bounds we derived for $g_j$, we have $g_j \in (\tfrac{i_j}{2^{L-j}},\tfrac{i_j+1}{2^{L-j}})$.
    
    On the other hand, if $\alpha_{j+1} = 1$, then $g_j = 1-\tfrac{1}{2} g_{j+1} \in (1-\tfrac{i_{j+1}+1}{2^{L-j}},1-\tfrac{i_{j+1}}{2^{L-j}})$ and $2^{L-j}g_j \in (2^{L-j}-i_{j+1}-1,2^{L-j}-i_{j+1})$. Thus, we have that $i_j = 2^{L-j} - i_{j+1} - 1 \in \lrbr{2^{L-j-1}, 2^{L-j}-1}$. Thus, by \cref{lem:gray_code_rec_def}, computing $\tilde{i}_j = 2^{L-j}-1 - i_j = i_{j+1}$, we have that the ($L-j$)-bit Gray code for $i_j$ is $\bm \alpha^{i_j} = [1 ~ \tilde{\bm \alpha}^{\tilde{i}_j}]$, where $\tilde{\bm \alpha}^{\tilde{i}_j}$ is the ($L-j-1$)-bit Gray code for $\tilde{i}_j$, which by the induction hypothesis is given as $[\alpha_{j+2} \dots \alpha_L]$. This yields $\bm \alpha^{i_j} = [1 ~ \alpha_{j+2} \dots \alpha_L] = [\alpha_{j+1} \dots \alpha_L]$ as desired. Further, observing the bounds we derived for $g_j$, we have $g_j \in (\tfrac{i_j}{2^{L-j}},\tfrac{i_j+1}{2^{L-j}})$.
    
    Thus, with these cases, the result holds by induction.
\end{proof}

Note that, if $g_L \in \{0,1\}$ (so that $x \in 2^{-L} \Z \cap (0,1)$), the same Gray codes from \cref{thm:rgc} can be used as when $g_L \in (0,1)$. The primary difference is that there two choices for this Gray code when $x \in 2^{-L} \Z \cap (0,1)$. This dichotomy stems from the fact that we will obtain $g_j = \tfrac{1}{2}$ from some $j$, introducing ambiguity in the choice of $\alpha_j$.

\subsection{Hereditary Sharpness}\label{ssec:hsharp}
In this section, we prove hereditary sharpness of the formulation $\PIP$. 

Let $L$ be a nonnegative integer, and define the sets $\PIP, \PLP$, $\QIP$, and $\QLP$ as before. Suppose we fix some subset of the binary variables to $\bm \alpha_I = \bm{\bar \alpha}_I \in \{0,1\}^{|I|}$ for some set $I \subseteq L$. Furthermore, define $\PIP_{\bm{\bar \alpha}_I} \defeq \PIP|_{\bm \alpha_I = \bm{\bar \alpha}_I}\defeq \{(x,y,\bm g, \bm \alpha) \in \PIP |  \alpha_i = \bar \alpha_i \text{ for } i \in I\}$, and similarly for $\QLP$, $\QIP$, and $\PLP$. We then wish to show $\QLP_{\bm{\bar \alpha}_I} = \conv(\QIP_{\bm{\bar \alpha}_I})$. We will use this notation for the remainder of this section. An example demonstrating the hereditary sharpness of the formulation is shown in \cref{fig:sharpness}

To study this relationship in more detail, we in particular wish to study $\PLP_{\bm{\bar \alpha}_I}$. Let $(x, y, \bm g, \bm \alpha) \in \PLP_{\bm{\bar \alpha}_I}$. For each $i \in I$, then $g_{i-1}$ relates to $g_i$ via the linear function
\begin{equation}\label{eq:alpha-fix-constr}
    g_i = 2g_{i-1} (1-\bar{\alpha}_i) + 2(1-g_{i-1}) \bar{\alpha}_i.
\end{equation}
Alternatively, for each $i \in \llbracket L \rrbracket \backslash I$, the relationship can be written as 
$$
         2|g_{i-1} - \alpha_i| \le g_i \le \min\{2g_{i-1}, 2(1-g_{i-1})\}.
$$
In this form, simply setting each $\alpha_i = g_{i-1}$ yields the least restrictive possible lower-bound of $0$ on $g_i$ in terms of $g_{i-1}$. Thus, making this choice, we find that $\proj_{x,y,\bm g}(\PLP_{\bm{\bar \alpha}_I})$ can be expressed via the constraints
\begin{equation}
    \begin{array}{rll}
         y &= g_0 - \dsum_{i=1}^L 2^{-2i} g_i\\
         g_0 &= x\\
         g_i &\le 2g_{i-1} & i \in \lrbr{L}, i \notin I\\
         g_i &\le 2(1-g_{i-1}) & i \in \lrbr{L}, i \notin I\\
         g_i &= 2g_{i-1} (1-\bar{\alpha}_i) + 2(1-g_{i-1}) \bar{\alpha}_i & i \in I\\
         g_i &\in [0,1] & i \in \lrbr{L}
    \end{array}
    \label{eq:NN_ck1}
\end{equation}
while, after combining the linear constraints with the new constraints \cref{eq:alpha-fix-constr} fixing variables $\alpha_i$ for $i \in I$, $\PIP_{\bm{\bar \alpha}_I}$ can be written as
\begin{equation}
    \begin{array}{rll}
         y &= g_0 - \dsum_{i=1}^L 2^{-2i} g_i\\
         g_0 &= x\\
         2(g_{i-1} - \alpha_i) & \le g_i \le 2g_{i-1} & i \in \lrbr{L}, i \notin I\\
         2(\alpha_i - g_{i-1}) & \le g_i \le 2(1-g_{i-1}) & i \in \lrbr{L}, i \notin I\\
         g_i &= 2g_{i-1} (1-\bar{\alpha}_i) + 2(1-g_{i-1}) \bar{\alpha}_i & i \in I\\
         g_i &\in [0,1] & i \in \lrbr{L}\\
         \alpha_i &\in \{0,1\} & i \in \lrbr{L}, i \notin I
    \end{array}
    \label{eq:NN_ck2}
\end{equation}
For convenience, for all $i \in I$, define 
\begin{equation}
G_i(g_{i-1}, \alpha_i) \defeq 2g_{i-1} (1-\alpha_i) + 2(1-g_{i-1}) \alpha_i.
\end{equation}
For $i \in \lrbr{L}$, define the shorthand $G_i(g_{i-1}) \defeq G_i(g_{i-1}, \bar \alpha_i)$ if $i \in I$ and $G_i = G$ otherwise. Then by the construction of $\PIP_{\bm{\bar \alpha}_I}$, we have that $\bm g \in \proj_{\bm g} (\PIP_{\bm{\bar \alpha}_I})$ if and only if for all $i \in \lrbr{L}$, we have $g_i = G_i(g_{i-1})$ and $g_i \in [0,1]$.

\sloppy For all $i \in \lrbr{0,L}$, the below proposition explores how to compute the feasible region for $g_i$, $\proj_{g_i}(\PLP_{\bm{\bar \alpha}_I}) \eqdef [a_i,b_i]$, while establishing that $\proj_{g_i}(\PLP_{\bm{\bar \alpha}_I}) = \conv(\proj_{g_i}(\PIP_{\bm{\bar \alpha}_I}))$. 

\begin{lemma}[Bounds in Projection] \label{lem:bndcomp}
     For all $i \in \lrbr{0,L}$ and $I \subseteq \llbracket L \rrbracket$, we have $\proj_{g_i}(\PLP_{\bm{\bar \alpha}_I}) = \conv(\proj_{g_i}(\PIP_{\bm{\bar \alpha}_I})) \eqdef [a_i,b_i] \neq \emptyset$. Furthermore, $[a_L,b_L] = [0,1]$, and $[a_{i-1},b_{i-1}]$ can be computed from $[a_i,b_i]$ as
        \begin{equation}
        [a_{i-1}, b_{i-1}] = 
        \begin{cases}
            [\tfrac{1}{2} a_i,\tfrac{1}{2} b_i] & \text{if $i \in I$ and $\bar{\alpha}_i = 0$},\\
            [1-\tfrac{1}{2} b_i, 1-\tfrac{1}{2}a_i] & \text{if $i \in I$ and $\bar{\alpha}_i = 1$},\\
            [\tfrac{1}{2} a_i, 1-\tfrac{1}{2}a_i] & \text{if $i \notin I$}.
        \end{cases}
    \end{equation}
 Note that in the last case $a_i \le \tfrac{1}{2}$ and $b_i \ge \tfrac{1}{2}$.
\end{lemma}

\begin{proof}
    We proceed by induction.
    
    \textbf{Base Case}: $i=L$
    
    In this case, \cref{thm:rgc} establishes that, even if $\hat{I} = \lrbr{L}$ with some corresponding $\bm{\hat{\alpha}}_{\hat{I}} \in \{0,1\}^L$, we have $[0,1] = \proj_{g_i}(\PIP_{\bm{\hat{\alpha}}_{\hat{I}}})$, yielding
    $$\begin{array}{rl} [0,1] &= \proj_{g_L}(\PIP_{\bm{\hat{\alpha}}_{\hat{I}}}) \subseteq \conv(\proj_{g_L}(\PIP_{\bm{\hat{\alpha}}_{\hat{I}}})) \\
    &\subseteq \conv(\proj_{g_L}(\PIP_{\bm{\bar{\alpha}}_I})) \subseteq  \proj_{g_L}(\PLP_{\bm{\bar{\alpha}}_I}) \subseteq [0,1],
    \end{array}$$
    and so $\conv(\proj_{g_L}(\PIP_{\bm{\bar \alpha}_I})) = \proj_{g_L}(\PLP_{\bm{\bar \alpha}_I}) = [0,1]$, as required.
    
    \textbf{Inductive step}: 
    
    Let $i \in \lrbr{L}$, and suppose that $\conv(\proj_{g_i}(\PIP_{\bm{\bar \alpha}_I})) = \proj_{g_i}(\PLP_{\bm{\bar \alpha}_I}) = [a_i,b_i]$. Then, observing \cref{eq:NN_ck1} and \cref{eq:NN_ck2}, we find that there are three cases.
    
    \textit{Case 1}: If $i \in I$ and $\bar \alpha_i = 0$, then in both $\PLP_{\bm{\bar \alpha}_I}$ and $\PIP_{\bm{\bar \alpha}_I}$, we have $g_{i-1} = \tfrac{1}{2} g_i$, yielding 
    $$
    \conv(\proj_{g_{i-1}}(\PIP_{\bm{\bar \alpha}_I})) \subseteq \proj_{g_{i-1}}(\PLP_{\bm{\bar \alpha}_I}) = [\tfrac{1}{2} a_i,\tfrac{1}{2} b_i].
    $$
    Furthermore, $a_i,b_i \in \proj_{g_i}(\PIP_{\bm{\bar \alpha}_I})$ yields $\tfrac{1}{2} a_i, \tfrac{1}{2} b_i \in \proj_{g_{i-1}}(\PIP_{\bm{\bar \alpha}_I})$. This implies $[\tfrac{1}{2} a_i, \tfrac{1}{2}b_i] \subseteq \conv(\proj_{g_{i-1}}(\PIP_{\bm{\bar \alpha}_I}))$, yielding 
    $$
    \begin{array}{rl}
        \conv(\proj_{g_{i-1}}(\PIP_{\bm{\bar \alpha}_I})) = \proj_{g_{i-1}}(\PLP_{\bm{\bar \alpha}_I}) = [1-\tfrac{1}{2} a_i, 1\tfrac{1}{2}b_i]
    \end{array}
    $$
    Note that $[a_i,b_i] \neq \emptyset \Rightarrow [\tfrac{1}{2} a_i, \tfrac{1}{2}b_i]\neq \emptyset$, and that $[a_i,b_i] \subseteq [0,1] \Rightarrow [\tfrac{1}{2} a_i, \tfrac{1}{2}b_i] \subseteq [0,1]$. 
    
    \textit{Case 2}: If $i \in I$ and $\bar \alpha_i = 1$, then in both $\PLP_{\bm{\bar \alpha}_I}$ and $\PIP_{\bm{\bar \alpha}_I}$, we have $g_i = 1-\tfrac{1}{2} g_i$, yielding
    $$\conv(\proj_{g_{i-1}}(\PIP_{\bm{\bar \alpha}_I})) \subseteq \proj_{g_{i-1}}(\PLP_{\bm{\bar \alpha}_I}) = [1-\tfrac{1}{2} b_i,1-\tfrac{1}{2} a_i].$$
    Furthermore, $a_i,b_i \in \proj_{g_i}(\PIP_{\bm{\bar \alpha}_I})$ yields $1-\tfrac{1}{2} b_i, 1-\tfrac{1}{2} a_i \in \proj_{g_{i-1}}(\PIP_{\bm{\bar \alpha}_I})$. This implies $[1-\tfrac{1}{2} b_i, 1-\tfrac{1}{2}a_i] \subseteq \conv(\proj_{g_{i-1}}(\PIP_{\bm{\bar \alpha}_I}))$, yielding 
    $$
    \begin{array}{rl}
        \conv(\proj_{g_{i-1}}(\PIP_{\bm{\bar \alpha}_I})) = \proj_{g_{i-1}}(\PLP_{\bm{\bar \alpha}_I}) = [1-\tfrac{1}{2} b_i, 1-\tfrac{1}{2}a_i].
    \end{array}
    $$
    Note that $[a_i,b_i] \neq \emptyset \Rightarrow [1-\tfrac{1}{2} b_i, 1-\tfrac{1}{2}a_i]\neq \emptyset$, and $[a_i,b_i] \subseteq [0,1] \Rightarrow [1-\tfrac{1}{2} b_i, 1-\tfrac{1}{2}a_i] \subseteq [0,1]$.
    
    \textit{Case 3}: If $i \notin I$, then $\PIP_{\bm{\bar \alpha}_I}$ and $\PLP_{\bm{\bar \alpha}_I}$ model different relations between $g_i$ and $g_{i-1}$. 
    
    In $\PLP_{\bm{\bar \alpha}_I}$, we have
    $$
        \begin{array}{rl}
             g_i &\le 2g_{i-1}\\
             g_i &\le 2(1-g_{i-1}),
        \end{array}
    $$
    which can be written as
    $$
        \begin{array}{rl}
             g_{i-1} &\ge \tfrac{1}{2} g_i \ge \tfrac{1}{2} a_i\\
             g_{i-1} &\le 1 - \tfrac{1}{2} g_i \le 1 - \tfrac{1}{2} a_i.
        \end{array}
    $$
    Further, as $a_i \in \proj_{g_i}(\PLP_{\bm{\bar \alpha}_I})$, we have that 
    $$\proj_{g_{i-1}}(\PLP_{\bm{\bar \alpha}_I}) = [\tfrac{1}{2} a_i,1 - \tfrac{1}{2} a_i]),$$ 
    where $a_i \in [0,1]$ implies $\tfrac{1}{2} a_i \in [0,\tfrac{1}{2}]$, so that $[\tfrac{1}{2} a_i,1 - \tfrac{1}{2} a_i] \neq \emptyset$. Thus, 
    $$\conv(\proj_{g_{i-1}}(\PIP_{\bm{\bar \alpha}_I})) \subseteq \proj_{g_{i-1}}(\PLP_{\bm{\bar \alpha}_I}) = [\tfrac{1}{2} a_i,1 - \tfrac{1}{2} a_i].$$
    
    To show $[\tfrac{1}{2} a_i,1 - \tfrac{1}{2} a_i] \subseteq \conv(\proj_{g_{i-1}}(\PIP_{\bm{\bar \alpha}_I})$, we have only to show that the endpoints are in $\proj_{g_{i-1}}(\PIP_{\bm{\bar \alpha}_I})$. To this end, let $g_i = a_i$. In $\PIP_{\bm{\bar \alpha}_I}$, we can choose either $\alpha_i = 0$ or $\alpha_i=1$. If $\alpha_i = 0$, we have that $g_{i-1} = \tfrac{1}{2} g_i = \tfrac{1}{2} a_i$, while if $\alpha_i = 1$, we have $g_{i-1} = 1-\tfrac{1}{2} g_i = 1-\tfrac{1}{2} a_i$. Thus, we have $[\tfrac{1}{2} a_i,1 - \tfrac{1}{2} a_i] \subseteq \conv(\proj_{g_{i-1}}(\PIP_{\bm{\bar \alpha}_I}))$, yielding 
    $$
    \begin{array}{rl}
        \conv(\proj_{g_{i-1}}(\PIP_{\bm{\bar \alpha}_I})) = \proj_{g_{i-1}}(\PLP_{\bm{\bar \alpha}_I}) = [\tfrac{1}{2} a_i,1 - \tfrac{1}{2} a_i] \neq \emptyset,
    \end{array}
    $$
    as desired. 
\end{proof}

Note that $\conv(\proj_{x}(\PIP_{\bm{\bar \alpha}_I})) = \proj_{x}(\PLP_{\bm{\bar \alpha}_I})$ is a direct corollary of the above lemma. This fact will be helpful during the proof of hereditary sharpness. Next, to prove hereditary sharpness, we want to show that if we fix some $g_i$ to either $a_i$ or $b_i$, then for each $j>i$, we have that $g_j$ has only one feasible solution $e_j$ in $\PLP_{\bm{\bar \alpha}_I}$, where $e_j \in \{a_i,b_i\}$.

\begin{lemma}[Single solution at endpoints] \label{lem:unique_bnd_sln}
    For all $i \in \lrbr{0,L}$ and $\hat{g_i} \in \{a_i,b_i\}$, we have for all $j \ge i$ that $\proj_{g_j}(\PLP_{\bm{\bar \alpha}_I}|_{g_i=\hat{g}_i}) = \{e_j\}$, where $e_j \in \{a_j,b_j\}$.
\end{lemma}
\begin{proof}
    We will proceed by induction. Let $\hat{g_i} \in \{a_0,b_0\}$. Fortunately, the base case is trivial, since we have chosen $g_i \in \{a_i,b_i\}$.
    
Thus, for an induction, let $j \in \lrbr{i+1,L}$, and assume that $\proj_{g_{j-1}}(\PLP_{\bm{\bar \alpha}_I}|_{g_i = \hat{g}_i}) = \{e_{j-1}\} \subset \{a_{j-1},b_{j-1}\}$. Then there are three cases.
    
    \textbf{Case 1}: If $j \in I$ and $\bar{\alpha}_{j} = 0$, then we have $g_{j} = 2 g_{j-1}$, so that  $\proj_{g_j}(\PLP_{\bm{\bar \alpha}_I}|_{g_i = \hat{g}_i}) = \{2 e_{j-1}\} \eqdef \{e_j\}$. Now, by \cref{lem:bndcomp}, we have that $a_{j} = 2 a_{j-1}$ and $b_{j} = 2 b_{j-1}$, so that $e_{j-1} \in \{a_{j-1},b_{j-1}\} \Rightarrow e_j \in \{a_j,b_j\}$.
    
    \textbf{Case 2}: If $j \in I$ and $\bar{\alpha}_{j} = 1$, then we have $g_{j} = 2 (1-g_{j-1})$, so that $\proj_{g_j}(\PLP_{\bm{\bar \alpha}_I}|_{g_j = \hat{g}_j}) = \{2 (1-e_{j-1})\} \eqdef \{e_j\}$. Now, by \cref{lem:bndcomp}, we have that $a_{j} = 2 (1-b_{j-1})$ and $b_{j} = 2 (1-a_{j-1})$, so that $e_{j-1} \in \{a_{j-1},b_{j-1}\} \Rightarrow e_j \in \{a_j,b_j\}$.
    
    \textbf{Case 3}: If $j \notin I$, then we have by \cref{lem:bndcomp} $a_{j} = 2a_{j-1} = 2(1-b_{j-1})$, with $a_j \le \tfrac{1}{2}$ and $b_j \ge \tfrac{1}{2}$. Further, we have in $\PLP_{\bm{\bar \alpha}_I}$ that
    $$
        \begin{array}{cc}
            g_{j} \le 2 g_{j-1} = 2 e_{j-1},\\
            g_{j} \le 2 (1-g_{j-1}) = 2 (1-e_{j-1}).
        \end{array}
    $$
    Now, if $e_{j-1} = a_{j-1}$, then $2 e_{j-1} = 2 a_{j-1} \le 1$, while $2(1-e_{j-1}) = 2(1-a_{j-1}) \ge 1$. Thus, these bounds consolidate to $a_j \le g_j \le 2 a_{j-1} = a_j$, which implies $g_j = a_j$. 
    
    On the other hand, if $e_{j-1} = b_{j-1}$, then $2 e_{j-1} = 2 b_{j-1} \ge 1$, while $2(1-e_{j-1}) = 2(1-b_{j-1}) \le 1$. Thus, these bounds consolidate to $a_j \le g_j \le 2(1-b_{j-1}) = a_j$, which implies $g_j = a_j$.
    Thus, in this case, we have $\proj_{g_j}(\PLP_{\bm{\bar \alpha}_I}|_{g_i = \hat{g}_i}) = \{a_j\} \eqdef\{e_j\}$. This completes the proof. 
\end{proof}

Now, computing all $a_{i}$ and $b_{i}$ via \cref{lem:bndcomp}, we obtain the following form for $\proj_{x,y,\bm g}\{\PLP|_{I, \bm{\bar \alpha}_I}\}$:
\begin{equation}
    \begin{array}{rll}
         y &= g_0 - \dsum_{i=1}^L 2^{-2i} g_i\\
         g_0 &= x\\
         g_i &\le 2g_{i-1} & i \in \lrbr{L} \backslash I \\
         g_i &\le 2(1-g_{i-1}) & i \in \lrbr{L} \backslash I\\
         g_i &= 2g_{i-1} (1-\bar{\alpha}_i) + 2(1-g_{i-1}) \bar{\alpha}_i & i \in I\\
         g_i &\in [a_i,b_i] & i \in \lrbr{L}.
    \end{array}
    \label{eq:NN_ck3}
\end{equation}
Now, note that, for each $i \in \lrbr{L}$, $g_{i}$ has negative coefficient in first equation of \eqref{eq:NN_ck3}. Note also that this is the only constraint in \eqref{eq:NN_ck3} involving $y$. Thus, if we are to minimize over $y$, then we implicitly maximize $g_{i}$, and so $g_{i}$ will tend towards its upper bound. Furthermore, it turns out that, in any $y$-minimal or $y$-maximal solution in $\PLP_{\bm{\bar \alpha}_I}$ given some fixed value of $x$, each $g_i$ can be explicitly computed from $g_{i-1}$ using only the constraints from \cref{eq:NN_ck3} directly connecting $g_i$ and $g_{i-1}$. We refer to this as the \emph{greedy} solution property described in \cref{lem:greedy} below, and it holds due to the rapid decay of coefficients of $g_i$'s. 

For each $i \in \lrbr{L} \setminus I$, let $b_i$ be as in \cref{lem:bndcomp}, and define
\begin{equation}
    u_i(g_{i-1}) \coloneqq \min\{b_i, 2 g_{i-1}, 2(1-g_{i-1})\}. \label{eq:uidef}
\end{equation}
\begin{lemma} \label{lem:greedy}
Let $a_i, b_i$ as in \cref{lem:bndcomp}, and let $\hat x \in [a_0,b_0]$. Define
\begin{subequations}
\begin{align}
 (y^{*}, \bm g^{*}) &=  \argmin \{y     \,  : \,      (y,\bm g) \in \proj_{y,\bm g}(\PLP_{I, \bm{\bar \alpha}_I}|_{x=\hat{x}}) \}, \label{eq:greedy_min_opt}\\
  (y^\star, \bm g^\star) &=  \argmax \{y       \,  : \,       (y,\bm g) \in \proj_{y,\bm g}(\PLP_{I, \bm{\bar \alpha}_I}|_{x=\hat{x}})\}.
\end{align}
\end{subequations}
Then, for $i \notin I$, we have
\begin{align*}
  g^{*}_i &= u_i(g_{i-1}) \\
  g^\star_i &=  a_i.
\end{align*}
That is, one of the upper bounds is tight for each $g_i$ when maximizing $y$, while the domain lower bound is tight for $g_i$ when minimizing $y$. 
\end{lemma}

\begin{proof}
Let $\hat{x} \in [a_0,b_0]$. Then $\PLP_{\bm{\bar \alpha}_I}|_{x=\hat{x}} \neq \emptyset$ by \cref{lem:bndcomp}.

We begin by proving that $g^*_i = u_i(g^*_{i-1})$ for all $i \notin I$.  Suppose for a contradiction that, for some subset $j \in J \subseteq \lrbr{L} \setminus I$ and $\varepsilon_j>0$, we have $g_{j}^* = u_{j}(g_{j-1}^*) - \varepsilon_j$. Let $i$ be maximal in $J$, so that $g_{j}^*=u_j(g_{j-1}^*)$ for all $j>i$. Then we have $i \notin I$, since otherwise we have $g_{i}^*= 2g_{i-1} (1-\bar \alpha_i) + 2(1-g_{i-1}) = u_j(g_{j-1}^*)$ from \cref{eq:NN_ck3}.

Let $\tilde{g}_{j} = u_{j}(g_{j-1})$ for all $j \ge i$. for convenience, define 
$$
\tilde{\bm g} = (g_0^*, \dots g_{i-1}^*, \tilde{g}_{i}, \dots, \tilde{g}_L).
$$
To show that this choice is feasible, i.e., $\tilde{\bm g} \in \proj_{\bm g}(\PLP_{\bm{\bar \alpha}_I})$, we make an inductive observation. First, note that $\tilde{g}_{i-1} = g^*_{i-1} \in [a_{i-1},b_{i-1}]$ by \cref{lem:bndcomp}, with $\tilde{\bm g}_{\lrbr{0,i-1}} \in \proj_{\bm g_{\lrbr{0,i-1}}}(\PLP_{\bm{\bar \alpha}_I})$. Furthermore, for any $j$, $\tilde{\bm g}_{\lrbr{0,j-1}} \in \proj_{\bm g_{\lrbr{0,j-1}}}(\PLP_{\bm{\bar \alpha}_I})$ implies that there exists some $g_j \in \proj_{g_j}(\PLP_{\bm{\bar \alpha}_I}|_{\bm g_{\lrbr{0,j-1}} = \tilde{\bm g}_{\lrbr{0,j-1}}}) \subseteq [a_j,b_j]$, and $\tilde{g}_j$ is the largest such value by construction, yielding $\tilde{g}_j \in [a_j,b_j]$ and $\bm\tilde{g}_{\lrbr{0,j}} \in \proj_{\bm g_{\lrbr{0,j}}}(\PLP_{\bm{\bar \alpha}_I})$.

Now, we have by definition that $\tilde{g}_{i} - g^*_{i} = \varepsilon_i$. Furthermore, observe that for all $j>i$, $u_j(g_{j-1})$ is Lipschitz continuous with Lipschitz constant $2$, yielding
$$
\begin{array}{rl}
    |\tilde{g}_{j} - g^*_{j}| = |u_j(\tilde{g}_{j-1}) - u_j(g^*_{j-1})| \le 2 |\tilde{g}_{j-1} - g^*_{j-1}|
\end{array}
$$
Applying this recursively yields, for all $j>i$,
$$
\begin{array}{rl}
    |\tilde{g}_{j} - g^*_{j}| &\le 2^{j-i} \varepsilon.
\end{array}
$$
Note that this implies that, for $j>i$, we have $\tilde{g}_{j} - g^*_{j} \ge -2^{j-i} \varepsilon$. Then we have
\begin{equation}
    \begin{array}{rl}
         y^* - \tilde{y} &= 2^{-2i}(\tilde{g}_i - g^*_i) + \sum_{j=i+1}^L 2^{-2j} (\tilde{g}_j - g_j^*)\\
                         &\ge 2^{-2i}\varepsilon + \sum_{j=i+1}^L 2^{-2j} (-2^{j-i} \varepsilon)\\
                         &= 2^{-i} \varepsilon\lrp{2^{-i} - \sum_{j=i+1}^L 2^{-j}}\\
                         &= 2^{-i} \varepsilon(2^{-i} - (2^{-i} - 2^{-L}))\\
                         &= 2^{-(i+L)} \varepsilon.
    \end{array}
\end{equation}
However, this is a contradiction: it implies $\tilde{y}<y^*$, but $y^*$ was optimal! Thus, all $g_i$ must take on their upper-bounds given $g_{i-1}$.

Next, we prove that $g^\star_i = a_i$ for all $i \notin I$. The idea behind the proof is identical, with a notable simplifying difference: there is only one (constant) lower-bound $a_i$ on each $g_i$ for $i \notin I$. Thus, when enforcing that each $g^\star_j = a_i$ for $j>i$, $j \notin I$ with $g_i = a_i + \varepsilon$, the shift from $g^\star_i$ to $\tilde{g}_i$ effects no change in $g_j$, $j > i$. That is, if $i+1 \notin I$, $\tilde{g}_{i+1} - g^\star_{i+1} = 0$, yielding $\tilde{g}_{j} - g^\star_{j} = 0$ for $j \ge i+1$, so that \cref{eq:pf_greedy_err} simplifies to
\begin{equation}
    \begin{array}{rll}
         \tilde{y} - y^\star &= 2^{-2i}(g^*_i - \tilde{g}_i) = 2^{-2i} \varepsilon,
    \end{array}
    \label{eq:pf_greedy_err}
\end{equation}
implying $\tilde{y} > y^\star$, a contradiction. On the other hand, if $i+1 \in I$, let $k = \min\{j \in L-I : j \ge i+2\}$. Then, for each $j \in \lrbr{i+1,\dots, k-1}$, we have $|\tilde{g}_i - g^\star_i| = 2^{j-i} \varepsilon$. Furthermore, as $k \notin I$, we have $\tilde{g}_{k} = a_k = g^\star_k$, so that $\tilde{g}_k = g^\star_k$ for all $j > k$, yielding
\begin{equation}
    \begin{array}{rl}
         \tilde{y} - y^\star &= 2^{-2i}(g^\star_i - \tilde{g}_i) + \dsum_{j = i+1}^{k-1} 2^{-2j} (\tilde{g}_j - g_j^*)\\
                         &\ge 2^{-2i}\varepsilon + \dsum_{j=i+1}^{k-1} 2^{-2j} (-2^{j-i} \varepsilon)\\
                         &= 2^{-i} \varepsilon\lrp{2^{-i} - \dsum_{j=i+1}^{k-1} 2^{-j}}\\
                         &= 2^{-i} \varepsilon(2^{-i} - (2^{-i} - 2^{-(k-1)}))\\
                         &= 2^{-(i+k-1)}, \varepsilon
    \end{array}
\end{equation}
again implying $\tilde{y} > y^\star$, a contradiction.
\end{proof}

\begin{observation}\label{obs:continuity}
    Since each $u_i$, defined in \cref{eq:uidef}, is a continuous piecewise linear function, \cref{lem:greedy} recursively implies that each $g^*_i$ is continuous in $x$, and is a function of $g^*_{i-1}$.
\end{observation} 

As a corollary to \cref{lem:unique_bnd_sln} and \cref{lem:greedy}, we have that, for the optimal solution to the minimization problem in \cref{lem:greedy}, we have that $g_j = b_j < G(g_{j-1})$ for exactly one value of $j$ if $\hat{x}$ is not feasible in $\PIP_{\bm{\bar \alpha}_I}$. 

\begin{corollary} \label{cor:gapsol}
    For all $\hat{x} \in \proj_x(\PLP_{\bm{\bar \alpha}_I}) \setminus \proj_x(\PIP_{\bm{\bar \alpha}_I})$, define $(y^*,\bm g^*)$ as in \cref{eq:greedy_min_opt}. Then we have for exactly one $j \in \lrbr{L} \setminus I$ that $g^*_j = b_j < G(g^*_{i-1})$.
    
    Furthermore, let $x^1,x^2 \in \proj_x{\PIP_{\bm{\bar \alpha}_I}}$ be such that for all $\lambda \in (0,1)$, we have  $\hat{x} \defeq \lambda x^1 + (1-\lambda) x^2 \notin \proj_x{\PIP_{\bm{\bar \alpha}_I}} = \emptyset$. Then the uniquely determined $j$ discussed above is the same for all $\hat{x} \in (x^1,x^2)$ where $(x^1,x^2)$ denotes the open interval between $x^1$ and $x^2$.
\end{corollary}
\begin{proof}
    Consider the first $j \in \lrbr{L} \setminus I$ for which $g^*_j = b_j$. Such a $j$ exists; otherwise, \cref{lem:greedy} would give $g^*_i = G_i(g_{i-1})$ for all $i \in \lrbr{L}$, a contradiction on the choice of $\hat{x}$ since there is a corresponding feasible choice of $\bm \alpha^* \in \{0,1\}^L$. Then, by \cref{lem:unique_bnd_sln}, the set $\proj_{\bm g_{\lrbr{j+1,L}}}(\PLP_{\bm{\bar \alpha}_I}|_{g_j = g^*_j})$ consists of only a single point, for which each $g_i \in \{a_i,b_i\}$ for $i \ge j+1$. Furthermore, by \cref{lem:unique_bnd_sln}, this point is also in $\proj_{\bm g_{\lrbr{j+1,L}}}(\PIP_{\bm{\bar \alpha}_I}|_{g_j = g^*_j})$ so that $g^*_i = G_i(g^*_{j-1})$ for all $i \ge j+1$. Further, if $g^*_j = G(g^*_{j-1}) = b_j$, then we would obtain $\bm g^* \in \proj_{\bm g} \PIP_{x  = \hat{x}}$, a contradiction on the choice of $\hat{x}$. Thus, as $g^*_j \le G(g^*_{j-1})$ by \cref{lem:greedy}, we must have that $g^*_j = b_j < G(g^*_{j-1})$.
    
    Now, let $x^1,x^2 \in \proj_x{\PIP_{\bm{\bar \alpha}_I}}$ be such that $(x^1,x^2) \cap \proj_x{\PIP_{\bm{\bar \alpha}_I}} = \emptyset$. Suppose that there is some $\hat{x} \in (x^1,x^2)$ such that, for all sufficiently small $\varepsilon>0$, we have that the value $j^1$ in for $x-\varepsilon$ is different from the value $j^2$ for $x+\varepsilon$. In this case, since the $g_j^*$ is continuous in $x$ and $g^*_{j-1}$, we have at $\hat{x}$ that both $g^*_{j_1} = b_{j_1} = G_{j_1}(g^*_{j_1-1})$ and $g^*_{j_2} = b_{j_2} = G_{j_2}(g^*_{j_2})$. Furthermore, for all other $i \in \lrbr{L} \setminus I$, we have by continuity that $g^*_i = G_i(g^*_{i-1})$, yielding $g^*_i = G_i(g^*_{i-1})$ for all $i$. This yields $\bm g^* \in \proj_{\bm g}{\PIP_{\bm{\bar \alpha}_I}}$, a contradiction on the choice of $\hat{x}.$
\end{proof}

With this greedy solution property and the following corollary, we are ready to prove the hereditary sharpness of $\PIP$ as a formulation for $\QIP$. It is helpful to note that, for the minimizer solutions in \cref{lem:greedy}, $g^*_i < b_i$ implies $g^*_i = G_i(g_{i-1})$. Furthermore, if $\bm g \in \proj_{\bm g}(\PLP_{\bm{\bar \alpha}_I})$ and $g^*_i = G_i(g_{i-1})$ for all $i \in \lrbr{L} \setminus I$, then we have $\bm g \in \proj_{\bm g}(\PIP_{\bm{\bar \alpha}_I})$.

\begin{theorem}\label{thm:hsharp}
$\QIP$ is hereditarily sharp.
\end{theorem}
\begin{proof}
Let $a_0, b_0$ as in \cref{lem:bndcomp}, so $[a_0, b_0] = \proj_x(\PLP_{\bm{\bar \alpha}_I}) = \conv(\proj_x(\PIP_{\bm{\bar \alpha}_I}))$.

Let $\hat x \in [a_0, b_0]$. We need to show that $\QLP_{\bm{\bar \alpha}_I} = \conv(\QIP_{\bm{\bar \alpha}_I})$.  Since both sets are compact convex sets in $\R^2$, it suffices to show that $\QLP_{\bm{\bar \alpha}_I}|_{x = \hat x} = \conv(\QIP_{\bm{\bar \alpha}_I})|_{x = \hat x}$.
In this vein, define we define the upper and lower bounds in $y$ for each set
\begin{align*}
    [y^{\LP}_l(\hat x), y^{\LP}_{u}(\hat x)] &= \left[\min\{ y : (x,y) \in \QLP_{\bm{\bar \alpha}_I}|_{x = \hat x}\},\max\{ y : (x,y) \in \QLP_{\bm{\bar \alpha}_I}|_{x = \hat x}\}\right],\\
    [y^{\IP}_{l}(\hat x), y^{\IP}_{u}(\hat x)] &=\\
    &\left[\min\{ y : (x,y) \in \conv(\QIP_{\bm{\bar \alpha}_I})|_{x = \hat x}\},\max\{ y : (x,y) \in \conv(\QIP_{\bm{\bar \alpha}_I})|_{x = \hat x}\}\right].
\end{align*}
We will show that the upper and lower bounds coincide.
    
    \paragraph{Lower bounds}
    We begin by showing that $y^{\IP}_{l}(x)=y^{\LP}_{l}(x)$. 
    Since $\hat x \in \proj_x(\PLP_{\bm{\bar \alpha}_I})$, we have two cases. If $\hat x \in \proj_x(\QIP_{\bm{\bar \alpha}_I})$, then due to sharpness by \cref{thm:sharp}, we have
    \begin{align*}
         \min\{y: y \in \QIP_{\bm{\bar \alpha}_I}|_{x = \hat x}\} &\ge \min\{y : y \in \QLP_{\bm{\bar \alpha}_I}|_{x = \hat x}\} \\
         &\ge \min\{y : y \in \QLP|_{x=\hat x}\} \\
         &= \min\{y : y \in \QIP|_{x = \hat x}\} \\
         &= \min\{y : y \in \QIP_{\bm{\bar \alpha}_I}|_{x = \hat x}\},
    \end{align*}
    and so the result holds. In order, the four relations hold by: relaxation; relaxation; sharpness; and the fact that $\PIP|_{x = \hat x}$ consists of a single point for feasible $\hat{x}$, so that the restriction does not change the optimal solution.
    
    Otherwise,  $\hat x \notin \proj_x(\QIP_{\bm{\bar \alpha}_I})$.  Let
    $$(y^*, \bm g^*) \defeq \argmin\{y : (y, \bm g) \in \proj_{y, \bm g}(\PLP_{\bm{\bar \alpha}_I}|_{x = \hat x})\}$$ 
    and let $x^1<\hat x$ and $x^2>\hat x$ be the closest lower and upper bounds to $\hat x$ in $\proj_x(\QIP_{\bm{\bar \alpha}_I})$ (such points exist by \cref{lem:bndcomp}). Let $(x^1,y^1,\bm g^1, \bm \alpha^1), (x^2,y^2,\bm g^2, \bm \alpha^2) \in \PIP_{\bm{\bar \alpha}_I}$ be chosen such that $\bm \alpha^1 \in \proj_{\bm \alpha}\PIP_{\bm{\bar \alpha}_I}|_{x \in (x^1 - 2^{-L}, x^1)}$ and $\bm \alpha^1 \in \proj_{\bm \alpha}\PIP_{\bm{\bar \alpha}_I}|_{x \in (x^2, x^2 + 2^{-L})}$. According to \cref{thm:rgc}, $\bm \alpha^1$ and $\bm \alpha^2$ are the Gray codes for some integers $i_j$ and $i_{j+1}$.  
    By \cref{lem:gcode_res}, we know that there exists exactly one index $k \in \lrbr{L}$ such that $\alpha^1_i = \alpha^2_i$ for all $i \neq k$ and $\alpha^1_k = 1 - \alpha^2_k$.
    
    It follows that $\bm g^1$ and $\bm g^2$ satisfy the equations
    \begin{align}
    \label{eq:upper-bound-equations}
        g_i &= G_i(g_{i-1}, \alpha^1_i) = 2g_{i-1} (1-\alpha^1_i) + 2(1-g_{i-1}) \alpha^1_i & i \in \lrbr{L} \setminus k.
    \end{align}
     Furthermore, by \cref{lem:greedy}, for the $y$-minimal solution at all three $x$-values, we have that all $g_i$, $i \in \lrbr{L} \setminus I$ take on $u_i(g_{i-1}) = \min\{2g_{i-1}, 2(1-g_{i-1}), b_i\}$. This function is linear if $i \in I$; otherwise, it is the minimum of three functions: two linear, and one constant. 
    
    Choose $\lambda \in [0,1]$ such that $\hat x = \lambda x^1 + (1-\lambda) x^2$ and define $\hat y$ and $\hat{\bm g}$ by $(\hat x,\hat{y},\hat{\bm g}) = \lambda(x^1,y^1,\bm g^1) + (1-\lambda) (x^2,y^2,\bm g^2)$.  By convexity of $\PLP_{\bm{\bar \alpha}_I}$, the point $(\hat x,\hat{y},\hat{\bm g})$ is in $\proj_{x,y,\bm g}(\PLP_{\bm{\bar \alpha}_I})$. We want to show that $(\hat x,\hat{y},\hat{\bm g})= (\hat x,y^*,\bm g^*)$. To do so, by \cref{lem:greedy}, we have only to show that all $\hat{g}_{i}=u_{i}(\hat{g}_{i-1})$. 
   
    By convexity, since $\bm g^1$ and $\bm g^2$ satisfy \cref{eq:upper-bound-equations}, it follows that $\hat g$ satisfies them as well.  Hence, $\hat g_i = G_i(\hat g_{i-1}, \alpha^1_i)$ for all $i \neq k$. Furthermore, we have by induction that $g_i^* = \hat{g}_i$ for all $i \le k-1$: (base case) $g_0^*=\hat{g}_0 = \hat{x}$; (inductive case) for $i \le k-1$, if $g^*_{i-1} = \hat{g}_{i-1}$, then by \cref{lem:greedy}, we have
    $$\hat{g}_i \le \max \{g_i : (x, y, \bm g, \bm \alpha) \in \PLP_{\bm{\bar \alpha}_I}|_{g_{i-1} = \hat{g}_{i-1}}\} = g^*_i = u_i(\hat g_{i-1}) \le G_i(\hat{g}_{i-1}, \alpha^1_i) = \hat{g}_i,$$ 
    so that $\hat{g}_i = g^*_i = G_i(\hat{g}_{i-1}, \alpha^1_i)$. Similarly, we have $\hat{g}_k \le g_k^*$. We will show that $\hat g_k = b_k$.
    
    To do so, we will first show that $g^*_k = b_k$. If this holds, then since $\hat{x}$ was arbitrary, and since $g_k^*$ is continuous in $\hat{x}$ by \cref{obs:continuity}, we have that $g^1_k = g^2_k = b_k$. Since $\hat g_k$ is a convex combination of $g^1_k$ and $g^2_k$, this yields $\hat g_k = b_k$. 
    
    To show $g^*_k = b_k$, we first note that, by \cref{cor:gapsol}, there exists some $j \in \lrbr{L} \setminus I$ such that, for all $\hat{x} \in (x^1,x^2)$, we have $g^*_j = b_j < G(g^*_{j-1})$, with $g^*_i < b_i$ for all $i<j$. Furthermore, since $g^*_i = G(g^*_{i-1})$ for $i \in \lrbr{k-1} \setminus I$, we must have that $j \ge k$. To establish that $j = k$, we will show that $g^*_k = b_k$ when $\lambda = \tfrac{1}{2}$, implying that $j>k$ is impossible, since then we would have $g^*_k < b_k$.
    
    Now, define $\tilde{I} = \lrbr{L} \setminus \{k\}$, and define $\tilde \alpha_I$ so that $\tilde \alpha_i = \alpha^1_i$ for all $i \in \tilde{I}$. Define $\tilde{a}_i$ and $\tilde{b}_i$ as in \cref{lem:bndcomp} for $\PLP_{\tilde \alpha_{\tilde I}}$. Let
    $$
    (\tilde{y}, \bm{\tilde{g}})\defeq \argmin\{y : (y, \bm g) \in \proj_{y, \bm g}(\PLP_{\tilde \alpha_{\tilde I}}|_{x = \hat x})\}.
    $$
    Then by \cref{lem:greedy}, we have $\tilde{g}_k = \min\{\tilde{b}_k, G(\tilde{g}_{k-1})\}$. However, $\tilde{g}_k = G(\tilde{g}_{k-1})$ would yield $\bm \tilde{g} \in \proj_{\bm g} (\PIP_{\tilde \alpha_{\tilde I}})$, a contradiction on the choice of $\hat{x}$. Thus, we have $\tilde{g}_k = \tilde{b}_k$. Since $\hat{x}$ was arbitrary and since $\tilde{g}_k$ is continuous in $\hat{x}$ by \cref{obs:continuity}, this implies that $g^1_k = g^2_k = \tilde{b}_k$. Now, this allows us to compute $g^1_{k-1}$ and $g^2_{k-1}$ given $\alpha^1_i$, which will allow us to compute $b^*_k$ for $\lambda = \tfrac{1}{2}$ via $g^*_{k-1} = \hat{g}_{k-1}$ and $g^*_k = u_k(g^*_{k-1})$.
    
    To compute $\hat{g}_{k-1}$, there are two cases. Either $\alpha^1=0$, yielding $g^1_{k-1} = \tfrac{1}{2} \tilde{b}_k$ and $g^2_{k-1} = 1-\tfrac{1}{2} \tilde{b}_k$, or $\alpha^1 = 1$, yielding $g^1_{k-1} = \tfrac{1}{2} (1-\tilde{b}_k)$ and $g^2_{k-1} = \tfrac{1}{2} \tilde{b}_k$. In either case, we have $g^1_{k-1} + g^2_{k-1} = 1$, and so
    $$g^*_{k-1} = \hat{g}_{k-1} = \tfrac{1}{2}(g^1_{k-1} + g^2_{k-1}) = \tfrac{1}{2},$$
    yielding by \cref{lem:greedy} that
    $$g^*_k = \min\{2 g^*_{k-1}, 2(1-g^*_{k-1}), b_k\} = \min\{1,1,b_k\} = b_k,$$
    as required. Thus, since $\hat g_k = b_k$ and $\hat{g}_i$ satisfies \cref{eq:upper-bound-equations} for all $i \neq k$, it follows that $\hat{g}_i = u_i(\hat{g}_{i-1})$ for all $i \in \lrbr{L}$. Hence, by \cref{lem:greedy}, we have $\bm g^* = \hat {\bm g}$.  
        
    \paragraph{Upper bounds}
    For the upper bounds, note that $(x,y) \in \QIP$ implies $y = F(x)$, where $F$ is a convex function. Furthermore, by \cref{lem:unique_bnd_sln}, we have that $y^{\IP}_{l}(x)=y^{\IP}_{u}(x)=y^{\LP}_{u}(e_0)=y^{\LP}_{l}(x)$ for $x \in \{a_0,b_0\}$ (as the extended-space solutions are unique in $\PLP_{\bm{\bar \alpha}_I}$ for $x \in \{a_0,b_0\}$). Thus, by the convexity of $F$, we have that $y^{\IP}_{u}=y^{\LP}_{u}$ if and only if for all $x \in [a_0,b_0]$, we have for some $\lambda \in [0,1]$ that $(x, y^{\LP}_{u}(x)) = \lambda(a_0,F(a_0)) + (1-\lambda)(b_0,F(b_0))$. Alternatively, since $y_{\LP_u}(x) = F(x)$ for $x \in [a_0,b_0]$, it suffices to show that $y^{\LP}_{u}(a_0) = F(a_0)$, $y^{\LP}_{u}(b_0) = F_{b_0}$, and that $y^{\LP}_{u}$ is a linear function of $x$.
    
    To this end, consider any $x \in \proj_{x}(\PLP_{\bm{\bar \alpha}_I})$, and consider $(y, \bm g) = \max_y\{(y, \bm g) \in \proj_{y, \bm g}(\PLP_{\bm{\bar \alpha}_I}|_{x})\}$. Then by \cref{lem:greedy}, we have for all $i \in \lrbr{L}$, $i \notin 1$ that $g_i = a_i$, which is a constant, while for all $i \in I$ we have that $g_i = G(g_{i-1},\bar \alpha_i)$. Thus, eliminating all $g_i's$, we find that, $y^*$ is defined as $y^* = x - t(x)$, where $t(x)$ is some linear function of $x$, and thus so is $y^*(x)$, as required.
\end{proof}

\subsection{A connection with existing MIP formulations}

Interestingly, Gray codes also naturally appear in the ``logarithmic'' MIP formulations for general continuous univariate piecewise linear functions due to Vielma et al.~\cite{Vielma2010,Vielma2009}. Consider applying this existing formulation\footnote{In actuality, any Gray code, not just the reflected Gray code studied in this paper, yields a (potentially distinct) logarithmic formulation for a univariate function. Here, we mean the one constructed with the reflected Gray code, which is the most common choice regardless.} to approximate the univariate quadratic term with the same $2^L+1$ breakpoints as discussed in \cref{sec:gcodes}. The resulting MIP formulation uses $L$ binary variables, which follow the same interpretation as the neural network formulation as discussed in \cref{thm:rgc}. Moreover, it requires $\mathcal{O}(L)$ linear constraints (excluding variable bounds), and is ideal, a stronger property than the sharpness shown in \cref{thm:hsharp}. However, it comes at the price of an additional $2^L+1$ auxiliary continuous variables, and so is unlikely to be practical without a careful handling through, e.g., column generation. Therefore, our formulation sacrifices strength to reduce this to $\mathcal{O}(L)$ auxiliary continuous variables.

\section{Convex hull characterization}\label{sec:chull}
We explore a facet characterization of the convex hull of our model.  Such a characterization could be used to improve and branch \& cut scheme when solving MIPs with our model.

Although \eqref{eqn:ideal-formulation-one-layer} offers a convex hull formulation for a single ``layer'' in our construction, the composition over multiple layers ($\L > 1$) in \eqref{eqn:graph-formulation} will in general fail describe the integer hull. We characterize additional valid inequalities for the integer hull of \eqref{eqn:graph-formulation} that are derived via a connection with the parity polytope.

We begin by rewriting the relationship between the variables associated with each layer with the quadratic recurrence relation 
\begin{equation} \label{eqn:recursion-1}
g_i = (1-2 \alpha_{i})(2g_{i-1} -1) + 1 \quad i \in \brackets{\L}.
\end{equation}
For convenience, let $h_i := 2g_i - 1$ and $a_i := (1-2\alpha_i)$ for each $i \in \brackets{\L}$. Then after some simple algebraic manipulation, \eqref{eqn:recursion-1} is equivalent to
\begin{equation*}
h_i = 2a_{i}h_{i-1} + 1 \quad i \in \brackets{\L}.
\end{equation*}
Expanding the recurrence relation, we have
\begin{align*}
h_\L =  2^{\L}\left(\prod_{i=1}^{\L} a_i\right) \left(h_0 + \sum_{i=1}^{\L-1} \frac{1}{2^i\prod_{j = 1}^i a_j}\right) + 1.
\end{align*}
Define $b_i := \prod_{j=1}^i a_j$ for each $i \in \brackets{\L}$. As each $a_j \in \{-1,+1\}$, each $b_i \in \{-1,+1\}$ as well, and so $b_i = 1 / b_i$. Hence,
\begin{equation} \label{eqn:recursion-2}
    h_\L = 2^{\L}b_{\L}\left(h_0 + \sum_{i=1}^{\L-1} 2^{-i}b_i\right)  + 1.
\end{equation}
Multiplying both sides of \eqref{eqn:recursion-2} by $b_\L \in \{-1,+1\}$ yields
\begin{equation} \label{eqn:recursion-3}
    h_\L b_\L = 2^\L\left( h_0 + \sum_{i=1}^{\L-1} 2^{-i}b_i \right) + b_\L.
\end{equation}
Combining this with the McCormick inequality $h_\L + b_L - 1 \leq h_\L b_\L$ that is valid for the bilinear left-hand side of \eqref{eqn:recursion-3} (recall $h_\L, b_\L \in [-1,+1]$), we derive the following valid inequalities:
\begin{subequations}
\begin{align}
        h_\L \leq 2^\L \left( h_0 + \sum_{i=1}^{\L-1} 2^{-i}b_i \right) + 1, \label{eqn:cut-1} \\
        h_\L \leq -2^\L \left( h_0 + \sum_{i=1}^{\L-1} 2^{-i}b_i \right) + 1. \label{eqn:cut-2}
\end{align}
\end{subequations}
which can be readily mapped back to the original space of variables.
\begin{proposition}
    The following inequalities are valid for \eqref{eqn:graph-formulation}:
    \begin{subequations}
    \begin{align}
        g_\L \leq 2^{\L-1} \left( 2g_0 - 1 + \sum_{i=1}^{\L-1} 2^{-i}\prod_{j=1}^i (1-2\alpha_j) \right) + 1 \label{eqn:cut-orig-space-1} \\
        g_\L \leq -2^{\L-1} \left( 2g_0 - 1 + \sum_{i=1}^{\L-1} 2^{-i}\prod_{j=1}^i (1-2\alpha_j) \right) + 1 \label{eqn:cut-orig-space-2}
    \end{align}
    \end{subequations}
    If $L=2$, then \eqref{eqn:cut-1} and \eqref{eqn:cut-2} are both linear inequalities in $g$ and $\alpha$.
\end{proposition}

Based on computational observations, for $\L=2$,  these are exactly the nontrivial facet-defining linear inequalities for the integer hull of \eqref{eqn:graph-formulation}. For $\L > 2$, we can produce a large class of valid inequalities by bounding the product variables $b_i$. In particular, bounds on these products can be derived from valid inequalities for the parity polytope.

\subsection{Parity inequalities}
The parity polytope is the convex hull of $P_n^{\even}$, the set of all $\alpha \in \{0,1\}^n$ whose components sum to an even number. It has $2^{n-1}$ facets of the form 
\begin{align} \label{eqn:parity-facets}
    \sum_{i \in \brackets{n}\setminus I} \alpha_i + \sum_{i \in  I} (1-\alpha_i) & \geq 1  & I \subseteq \brackets{n} \text{ s.t. } |I| \text{ is odd}.
\end{align}

Define $\beta_i \in \{0,1\}$, such that $b_i = (1-2 \beta_i)$. Then 
\begin{equation}
1 = b_j^2 = (1-2 \beta_j) \prod_{i=1}^j (1-2\alpha_i) = (-1)^{\beta_j + \sum_{i=1}^j \alpha_i}.
\end{equation}
Hence, for $\alpha_i \in \{0,1\}$, the sum  $\beta_j + \sum_{i=1}^j \alpha_i$ is even and therefore $(\beta_j, \alpha_1, \dots, \alpha_j) \in P_{j+1}^{\even}$.   Therefore, we can apply \eqref{eqn:parity-facets} to derive valid inequalities for feasible solutions to \eqref{eqn:graph-formulation} of the form
\begin{subequations}
\begin{align}
        \beta_j + \sum_{i \in \brackets{j}\setminus I} \alpha_i + \sum_{i \in I} (1-\alpha_i) & \geq 1  & I \subseteq \brackets{j} \text{ s.t. } |I| \text{ is odd},\\
            (1-\beta_j) + \sum_{i \in \brackets{j}\setminus I} \alpha_i + \sum_{i \in I} (1-\alpha_i) & \geq 1  & I \subseteq \brackets{j} \text{ s.t. } |I| \text{ is even}.
\end{align}
\end{subequations}
After recalling the definition $b_j = (1-2\beta_j)$ and rearranging, we are left with
\begin{subequations}
\begin{align}
\label{eq:beta-upper}
        b_j &\leq 2 \left( \sum_{i \in \brackets{j}\setminus I} \alpha_i + \sum_{i \in I} (1-\alpha_i) \right)-1  & I \subseteq \brackets{j} \text{ s.t. } |I|  \text{ is odd},\\
        -b_j &\leq  2\left( \sum_{i \in \brackets{j}\setminus I} \alpha_i + \sum_{i \in I} (1-\alpha_i)\right)  - 1&  I \subseteq \brackets{j} \text{ s.t. } |I| \text{ is even}.
            \label{eq:-beta-upper}
\end{align}
\end{subequations}
Hence, combining these upper bounds on $b_i$ with \eqref{eqn:cut-1} and the upper bounds on $-b_i$ with \eqref{eqn:cut-2} produces an exponential family of valid inequalities for feasible solutions to \eqref{eqn:graph-formulation}. We call these \emph{parity inequalities}. 

\begin{example}
For $\L = 4$, we compute some of the nontrivial facet-defining inequalities of \eqref{eqn:graph-formulation}, written in terms of the original variables:
$$
\begin{array}{rcrlcrlcrlcrlcrlcl}
 g_4 &\leq&&16  g_{0}  &-& 14  \alpha_{1} &+& 6  \alpha_{2} &+& 2  \alpha_{3},  \\
 g_4 &\leq&&16  g_{0}  &-& 2  \alpha_{1} &-& 6  \alpha_{2} &+& 2  \alpha_{3},   \\
 g_4 &\leq&-&16  g_{0}&  +& 14  \alpha_{1} &+& 6  \alpha_{2} &+& 2  \alpha_{3} &+&2, \\
 g_4 &\leq&-&16  g_{0}  &+& 2  \alpha_{1} &-& 6  \alpha_{2} &+& 2  \alpha_{3} &+&14. \\
\end{array}
$$
Each of these inequalities is, in fact, a parity inequality, and can be constructed by a suitable combination of either \eqref{eqn:cut-1} with inequalities from \eqref{eq:beta-upper}, or \eqref{eqn:cut-2} with inequalities from \eqref{eq:-beta-upper}.

For example, the facet $g_4 \leq 16  g_{0}  - 14  \alpha_{1} + 6  \alpha_{2} + 2  \alpha_{3}$ is equivalent to 
 \begin{equation*}
 h_4 
 \leq 16  h_{0} + 15 + 2(- 14  \alpha_{1} + 6  \alpha_{2} + 2  \alpha_{3}),
 \end{equation*}
which can be produced as a conic combination of the inequality
\begin{equation*}
    h_4 \leq 2^4 \left( h_0 + \frac{1}{2}b_1 + \frac{1}{4}b_2 + \frac{1}{8}b_3 \right) + 1
\end{equation*}
from \eqref{eqn:cut-1} with the inequalities
\begin{align*}
b_1 &\leq  2 \big( 
    (1-\alpha_1)
\big) -1,\\
b_2 &\leq  2 \big( 
    (1-\alpha_1) + \alpha_2
\big) -1,\\
b_3 &\leq2 \big( 
    (1-\alpha_1) + \alpha_2 + \alpha_3)
\big) -1
\end{align*}
from the family \eqref{eq:beta-upper} corresponding to $I=\{1\}$ for all $j=1,2,3$, respectively.
\end{example}

\subsection{Separation over exponentially many parity inequalities}
Since there may be exponentially many parity inequalities, we provide an algorithm to separate over them.  In particular, given a point $(g,\alpha) \subseteq [0,1]^{\L+1}\times [0,1]^\L$, we can determine if it lies in the intersection of the parity inequalities by computing the inequalities that give smallest upper bounds for $b_j$ for each $j \in \llbracket \L \rrbracket$.  To do so, for each $b_j$, we need to determine the set $I_j \subseteq \brackets{j}$ that minimizes the right-hand side of equation \eqref{eq:beta-upper} or \eqref{eq:-beta-upper}.  
This can be done, in fact, by optimizing over another parity polytope. That is, set $I_j := \{i \in \brackets{j} : z^*_i = 1\}$, where
$$
 z^* \in \argmin_{z \in \{0,1\}^j}\Set{ 
 \sum_{i=1}^j z_i \alpha_i +  (1-z_i)(1-\alpha_i) | \sum_{i=1}^j (1-z_i) \text{ is odd}}.
$$
Linear functions can be optimized over the parity polytope in polynomial time via a linear size extended formulation~\cite{Kaibel2013}, or by writing an integer program with a single integer variable~\cite{Bader2018}, or by a simple greedy-like algorithm. An analogous approach can be taken if the sum of $z_i$ should be odd.

\section{Area comparisons}
\label{sec:area-section} The approximation presented above is an over approximation of $y = x^2$.  This is sufficient for providing dual bounds due how the approximation is applied using the diagonal perturbation.  However, our formulation can also be altered slightly to provide an under approximation of $y=x^2$, and in particular, creates a covering of the curve with a union of polytopes. We describe two relaxations that are comparable to that of Dong and Luo~\cite{Dong-Luo-2018}. We then compare these models based on the combined area of the covering to see how these methods converge.

We construct our first relaxation, named \BHH{1}, from the constraints \eqref{eqn:ideal-formulation-one-layer} and
\begin{subequations}\label{eq:sawtooth-relax}
    \begin{align}
          y &\leq x - \sum_{i=1}^\L \tfrac{g_i}{2^{2i}} \label{eq:sawtooth-relax-ub}\\
          y &\geq \left(x - \sum_{i=1}^{j} \frac{g_i}{2^{2i}} \right) - \frac{1}{2^{2j+2}} && j \in \lrbr{L-1} \label{eq:sawtooth-relax-lb}\\
          y &\ge 2x-1\label{eq:sawtooth-relax-rhlb}\\
          y &\ge 0\label{eq:sawtooth-relax-lhlb}
    \end{align}  
\end{subequations}
We can form a tighter relaxation \BHH{2} by starting with \BHH{1}, then adding the cut \cref{eq:sawtooth-relax-lb} with $j=L$:
\begin{align}\label{eq:sawtooth-relax-lb2}
      y &\geq \left(x - \sum_{i=1}^{j} \frac{g_i}{2^{2i}} \right) - \frac{1}{2^{2j+2}} & j \in \lrbr{L}.
\end{align}

\cref{tab:volume21,tab:volume-01} compare the volume of our relaxed method with the method of Dong and Luo~\cite{Dong-Luo-2018} on the intervals $x \in [0,1]$ and $x \in [-2,1]$, respectively. As $L$ increases, the volume of our relaxation consistently shrinks by a factor of 4, which is strictly greater by a fair margin to the improvement rate observed for the method of Dong and Luo. We can formalize our rate of improvement in the following proposition.

\begin{table}
{\small
\begin{tabular}{l|llllllll}
Method & $\L = 0$ \  & $\L = 1$ & $\L = 2$ & $\L = 3$ & $\L = 4$  \\
\hline 
\HL{} & 0.25 & 0.0680 \textbf{(3.68)} & 0.0177\textbf{(3.84)} & 0.00448\textbf{(3.95)} & 0.00112 \textbf{(3.994)}  \\
\BHH{1} & 0.25 & 0.0625\textbf{(4)} & 0.0156\textbf{(4)} & 0.00391\textbf{(4)} & 0.000977\textbf{(4)} \\
\BHH{2} & 0.188 & 0.0469\textbf{(4)} & 0.0117\textbf{(4)} & 0.00293\textbf{(4)}
& 0.000732\textbf{(4)}
\end{tabular}}
\begin{tabular}{ccc}
    \includegraphics[scale = 0.18]{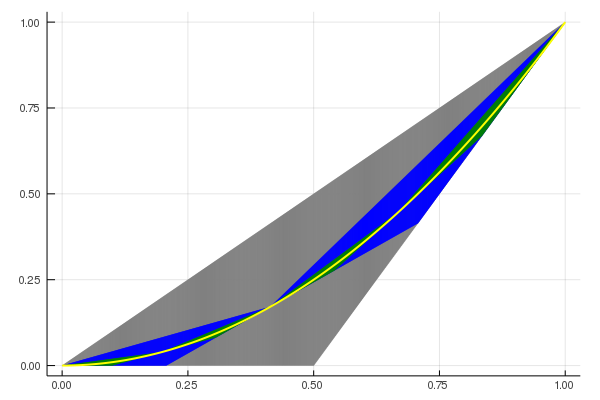} &
    \includegraphics[scale = 0.18]{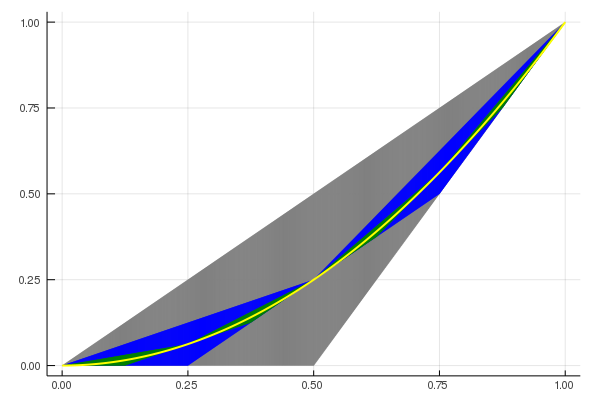} &
    \includegraphics[scale = 0.18]{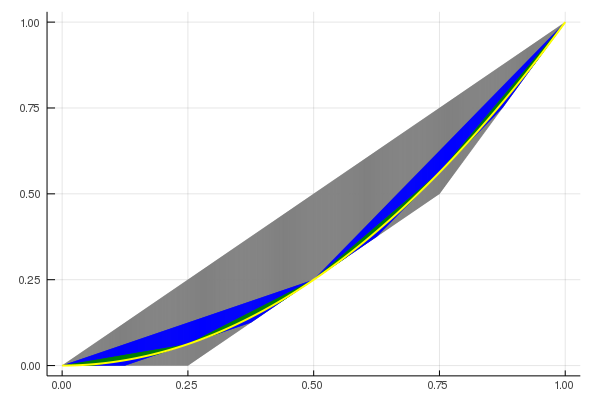}\\
    \HL{} & \BHH{1} & \BHH{2}
    \end{tabular}
\caption{Area/ratio table $x \in [0,1]$.  The Gray area is for $L=0$, the blue area is $L=1$, the green area (very small) is $L=2$, and the yellow curve is the curve $y=x^2$.   The area for each method was approximated numerically using a Riemann sum.}
\label{tab:volume-01}
\end{table}

\begin{table}
    \begin{tabular}{l|ccccccc}
    Method \ & $\L = 0$ \  & $\L = 1$ & $\L = 2$ & $\L = 3$ & $\L = 4$  \\
    \hline 
         \HL{} & 6.75 & 1.94\textbf{(3.47)} & 0.659\textbf{(2.95)} & 0.197\textbf{(3.34)} & 0.0530\textbf{(3.72)}\\
         \BHH{1} & 6.75 & 1.69\textbf{(4)} & 0.422\textbf{(4)} & 0.105\textbf{(4)} & 0.0264\textbf{(4)}\\
         \BHH{2} & 5.06 & 1.27\textbf{(4)} & 0.316\textbf{(4)} & 0.0791\textbf{(4)} & 0.0198\textbf{(4)} 
    \end{tabular}

    \begin{tabular}{ccc}
        \includegraphics[scale = 0.18]{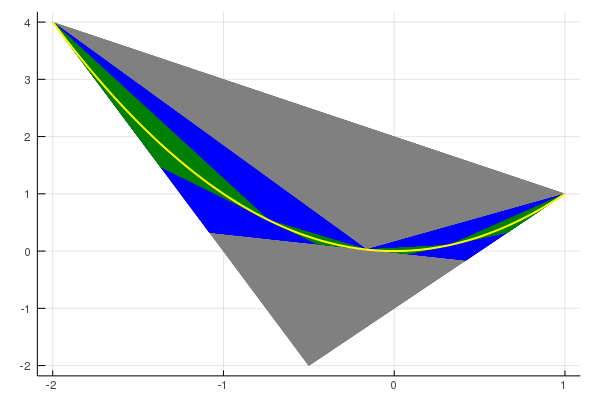} &
        \includegraphics[scale = 0.18]{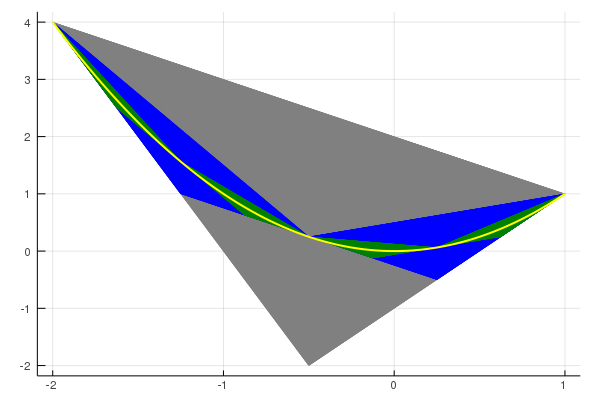} &
        \includegraphics[scale = 0.18]{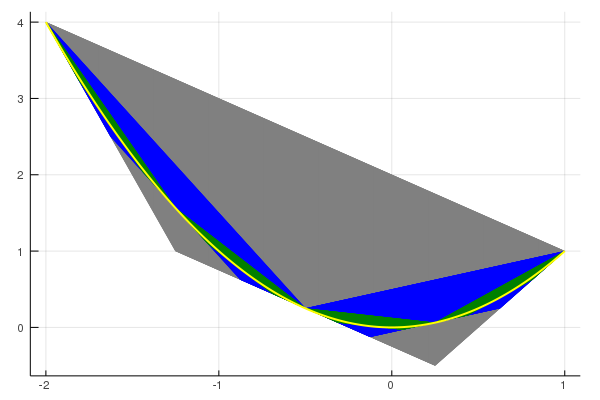}\\
        \HL{} & \BHH{1} & \BHH{2}
    \end{tabular}
    \caption{Area/ratio table $x \in [-2,1]$. Area/ratio table $x \in [0,1]$.  The Gray area is for $L=0$, the blue area is $L=1$, the green area (very small) is $L=2$, and the yellow curve is the curve $y=x^2$.   The area for each method was approximated numerically using a Riemann sum.  
    Interestingly, the \HL{} model converges at different rates depending on the domain, whereas our methods always converges at the same rate.}
    \label{tab:volume21}
\end{table}

\begin{proposition}\label{prop:area-approx}
The volume of our approximation decreases by a factor of 4 with each subsequent layer (i.e. as $\L$ increases).  Furthermore, the expected error at points $x$ sampled uniformly at random from the input interval domain is proportional to the total volume.
\end{proposition}

\cref{prop:area-approx} relies on the characterization of \BHH{1} as the piecewise McCormick relaxation of $y=x^2$ at uniformly-spaced breakpoints. For one piece $[x_1,x_2]$, this relaxation consists of the tangent lines, or outer-approximation cuts, at $x_1$ and $x_2$ for the lower bound, and the secant line between $x_1$ and $x_2$ for the upper bound. We have already established that the upper bound, the \NN{} approximation, is a piecewise interpolant to $x^2$ at the chosen breakpoints, yielding the secant line on each interval $[x_1,x_2]$ between interpolation points, and thus the upper-bound of the piecewise McCormick approximation. In \cref{lem:area-lb-oa} below, we show that the lower bound of the \BHH{1} and \BHH{2} approximations give the piecewise McCormick lower bounds for $2^L$ and $2^{L+1}$ uniformly-spaced breakpoints, respectively.

\begin{lemma} \label{lem:area-lb-oa}
    Define $T$ as the lower-bounding set in $(x,y)$ for the relaxation \BHH{2}, $T = \proj_{x,y}\{(x,y,\bm g, \bm \alpha) \in [0,1] \times \R \times [0,1]^L \times \{0,1\}^L : \cref{eq:sawtooth-relax-lb2,eq:sawtooth-relax-rhlb,eq:sawtooth-relax-lhlb,eqn:ideal-formulation-one-layer}\}$. Then $T$ can be constructed via $2^{L+1}+1$ uniformly-spaced outer-approximation cuts for $y \ge x^2$ on the interval $[0,1]$, based on the tangent lines to $x^2$ at $x_i = i \cdot 2^{-(L+1)}$, $i=0, \dots, 2^{(L+1)}$. That is, letting $p_{\hat{x}}(x)$ be the tangent line to $y=x^2$ at the point $\hat{x}$,
    \begin{equation}
        p_{\hat{x}}(x) = 2\hat{x}(x-\hat{x}) + \hat{x}^2 = \hat{x}(2x-\hat{x}),
    \end{equation}
    $T$ is equivalent to $\left\{(x,y) \in [0,1] \times \R : y \ge p_{\hat{x}}(x) \quad \forall i \in \{0, \dots, L\},~\hat{x} = i\cdot 2^{-(L+1)}\right\}$.
\end{lemma}

\begin{proof}
To begin, we note that, as shown by Yarotsky \cite{Yarotsky-2016}, we have that for each $l = 0, \dots, L$, $F_l(x)$ gives a piecewise-linear interpolant in $y=x^2$ at $2^l + 1$ uniformly-spaced points on $x \in [0,1]$. That is, for each $i \in \{0,1,\dots,2^l\}$, we have that $F_l(i \cdot 2^{-l}) = (i \cdot 2^{-l})^2$.

For $x_1<x_2$, Consider a single linear interpolant to $y=x^2$ on the interval $[x_1,x_2]$, given as
\begin{equation}
    \begin{array}{rl}
        h(x) &= \dfrac{x_2^2-x_1^2}{x_2-x_1}(x-x_1) + x_1^2\\
             &= (x_1+x_2)(x-x_1) + x_1^2.
    \end{array}
\end{equation}
Since $h(x)-x^2$ is concave, the deviation $h(x)-x^2$ is maximized when $\frac{d}{dx} (h(x)-x^2) = x_1+x_2-2x = 0$, yielding $x^* = \frac{x_1+x_2}{2}$. At this point, we have   
\begin{equation}
    \begin{array}{rll}
        h(x^*)-(x^*)^2 &= (x_1+x_2)\lrp{\dfrac{x_1+x_2}{2}-x_1} + x_1^2 - (\dfrac{x_1+x_2}{2})^2\\
                       &= \frac 12 (x_1+x_2)^2 - \frac 14 (x_1+x_2)^2 - (x_1+x_2)x_1 + x_1^2\\
                       &= \frac 14 (x_1+x_2)^2 - x_1 x_2\\
                       &= \frac 14 x_1^2 + \frac 12 x_1x_2 - x_1x_2 +  \frac 14 x_2^2\\
                       &= \frac 14 (x_2-x_1)^2.
    \end{array}
\end{equation}
Thus, we have that $(x^*)^2 = h(x^*) - \frac 14 (x_2-x_1)^2$. Since $x^2$ is a convex function, and since $h- \frac 14 (x_2-x_1)^2$ is tangent to $x^2$ at $x^*$ (as both slopes are $x_2-x_1$), this implies that for all $x \in [0,1]$, we have $(x^*)^2 \ge h(x) - \frac 14 (x_2-x_1)^2$.

Now, since $F_l(x)$ is a linear interpolant to $y=x^2$ on each interval $[i \cdot 2^{-l}, (i+1)\cdot 2^{-l}]$, with interval width $2^{-l}$, we have that the cuts $y \ge F_l(x) - \frac 14 2^{-2l} = F_l(x) - 2^{-2l-2}$ are valid for $y=x^2$, and furthermore are tangent to $x^2$ at the point $x=(i+\frac 12) \cdot 2^{-l}$, the midpoints of all interpolants. Thus, we obtain outer-approximation cuts at $x=(i+\frac 12) \cdot 2^{-l}$ for each $i=0, \dots, 2^l-1$.

We obtain $T$ by applying these cuts for $l=0, \dots, L$, combined with the outer-approximation cuts $y \ge 0$ (tangent at $x=0$) and $y \ge 2x-1$ (tangent at $x=1$). We now show that this yields $2^{L+1}+1$ uniformly-spaced outer-approximation cuts to $y=x^2$. This can be easily seen by expressing the outer-approximation points in binary. The points $x=(i+\frac 12) \cdot 2^{-l}$ for each $i=0, \dots, 2^l-1$ can be characterized as exactly the numbers in $(0,1)$ such that the $(l+1)$th binary decimal is a $1$, and all later binary decimals are zero. Alternatively, it is the set of points
$$P_l \defeq \left\{x \in [0,1] : \exists~ \mathbf{a} \in \{0,1\}^l \text{ s.t. } x = 2^{-(l+1)} + \sum_{i=1}^l 2^{-i} a_i\right\}.$$

We then have that the set of outer-approximation points, $x \in \{0,1\} \cup \bigcup_{l \in \llbracket 0, L\rrbracket} P_l$, is the set of all binary decimals in the interval $[0,1]$ for which the last $1$ occurs no later than the $L+1$st decimal place, which has cardinality $2^{L+1}+1$. This forms the set of uniformly-spaced points $i \cdot 2^{-(L+1)}$, $i = 0, \dots, 2^{L+1}$. Thus, we have that the set $T$ consists of a set of $2^{L+1}+1$ uniformly-spaced outer-approximation cuts to the function $y=x^2$, as required.
\end{proof}

With \cref{lem:area-lb-oa}, we establish that the \BHH{1} relaxation is equivalent to the piecewise McCormick relaxation of $y=x^2$ on the uniformly-spaced breakpoints $x_i = 2^{-L}i$, $i \in \lrbr{0,2^L}$. We now establish the area of any given piece of the relaxation.

\begin{lemma}\label{lem:pw-mccormick-area}
    The area of the McCormick Relaxation of $y = x^2$ on the interval $[x_1, x_2]$ is  $\tfrac{1}{4}(x_2 - x_1)^3$.
\end{lemma}
\begin{proof}
We wish to obtain a closed-form solution for the area of the resulting triangle region. To do so, we compute the vertices of this triangle, construct vectors between them, then compute the cross-product of these vectors.

The equations of the tangent lines are given by
\begin{equation*}
\begin{array}{ll}
    y = 2 x_1 (x - x_1) + x_1^2 = 2x_1 (x - \frac{x_1}{2}),\\
    y = 2 x_2 (x - x_2) + x_2^2 = 2x_2 (x - \frac{x_2}{2}).
\end{array}
\end{equation*}
We thus find that the intersection of these lines is given by the point $(\frac{x_1+x_2}{2},x_1x_2)$. Thus, the three vertices are given as $v_1 = (x_1,x_1^2)$, $v_2 = (\frac{x_1+x_2}{2},x_1x_2)$, and $v_3 = (x_2,x_2^2)$. From these vertices, we obtain two vectors
\begin{equation*}
    \begin{array}{cc}
         v_2-v_1 = [\frac 12 (x_2 - x_1), x_1(x_2-x_1)] = (x_2 - x_1) [\frac 12, x_1],\\
         v_3-v_2 = [\frac 12 (x_2 - x_1), x_2(x_2-x_1)] = (x_2 - x_1) [\frac 12, x_2].
    \end{array}
\end{equation*}
The area is given by the magnitude of the cross product as
\begin{equation*}
    A = \frac 12 |(v_2 - v_1) \times (v_3 - v_2)| = \frac 12 (x_2-x_1)^2 |\tfrac{1}{2} x_2 - \tfrac{1}{2} x_1| = \frac 14 (x_2 - x_1)^3.
\end{equation*}
As required.
\end{proof}

With \cref{lem:pw-mccormick-area}, we are now ready to show that the area of \BHH{1} is optimal among all piecewise McCormick relaxations with a fixed number of pieces.

\begin{proposition}
The minimum possible area covering $y=x^2$ on $x \in [a,b]$ with a sequence of $n$ McCormick Relaxations is $\tfrac 14 (b-a)^3/n^2$, and is achieved via uniformly spaced breakpoints.
\end{proposition}
\begin{proof}
Consider a general piecewise McCormick relaxation of $y=x^2$ on the interval $[a,b]$ with consecutive breakpoints $a=x_0\le x_1\le\dots\le x_n=b$. For any segment of consecutive breakpoints $x_i$ and $x_{i+1}$, the McCormick relaxation between those points is bounded by the secant line between $(x_i,x_i^2)$ and $(x_{i+1},x_{i+1}^2)$, and the tangent lines to $y=x^2$ at $x_i$ and $x_{i+1}$. For simplicity, let $i=1$ for this discussion. 

Thus, letting $y_i = x_{i} - x_{i-1},~i = 1 \dots n$, the problem of choosing the area-optimal breakpoints $a \le x_1, \le x_2 \le \dots \le x_{n-1} \le x_n$ reduces to solving
\begin{equation}
    \label{eq:nlp_area}
    \min_{y \in \R^n_+}
    \left\{\frac 14 \sum_{i=1}^n y_i^3 :  \sum_{i=1}^n y_i = b-a \right\}.
\end{equation}

It is then easy to show via the KKT conditions that, due to the convexity of $\sum_i y_i^3$ on positive support, all $y_i$ must be equal, yielding the uniformly-spaced solution $y_i = \frac{b-a}{n}$. Thus, the choice of breakpoints induced by our algorithm is optimal. \qed
\end{proof}
The \BHH{1} relaxation is exactly a union of McCormick Relaxations $2^n$ uniformly spaced breakpoints. Hence the total area is $\tfrac 14 (b-a)^3 2^{-2n}$. We now show that adding the inequality \cref{eq:sawtooth-relax-lb} with $j=L$ cuts off an extra fourth of the total area.
\begin{proposition}
The area of the relaxation in \cref{eq:sawtooth-relax} is $\tfrac{3}{16}(b-a)^3 2^{-2n}$.
\end{proposition}
\begin{proof}
We will compute the area removed by the addition of this cut on a general interval $[x_1, x_2]$, where $x_1 = 2^{-L}i$ for some $i \in \lrbr{0,2^L-1}$. As shown in \cref{lem:area-lb-oa}, the cut \cref{eq:sawtooth-relax-lb} with $j=L$ intersects the curve at $x = \frac{x_1+x_2}{2}$. Now, the area removed by the addition of this cut is the area of the triangle formed by the intersections of the tangent lines at $x_1$, $x_2$, and the midpoint $x_3 \defeq \frac{x_1+x_2}{2}$. These vertices, given as intersection points $v_{ij}$ for tangent lines for $x_i$ and $x_j$, are derived following the process in \cref{lem:pw-mccormick-area} as:
\begin{equation*}
\begin{array}{ll}
    v_{12} &= \lrp{\frac{x_1+x_2}{2},x_1x_2}\\
    v_{13} &= \lrp{\frac 12 \cdot \lrp{\frac{x_1 + x_2}{2} + x_1}, \frac 12 x_1 (x_1+x_2)} 
            = \lrp{\frac 14 (3 x_1 + x_2), \frac 12 x_1 (x_1+x_2)}\\
    v_{23} &= \lrp{\frac 12 \cdot \lrp{\frac{x_1 + x_2}{2} + x_2}, \frac 12 x_2 (x_1+x_2)} 
            = \lrp{\frac 14 (x_1 + 3 x_2), \frac 12 x_2 (x_1+x_2)}
\end{array}
\end{equation*}
From these points, we obtain vectors
\begin{equation*}
    \begin{array}{ll}
         v_{12} - v_{13} &= \lrp{\frac 14 (x_2 - x_1), \frac 12 x_1(x_2-x_1)}
                          = (x_2-x_1) \lrp{\frac 14, \frac 12 x_1}\\
         v_{23} - v_{12} &= \lrp{\frac 14 (x_2 - x_1), \frac 12 x_2(x_2-x_1)} 
                          = (x_2-x_1) \lrp{\frac 14, \frac 12 x_2}.
    \end{array}
\end{equation*}
Finally, we obtain cut area
\begin{equation*}
    A_{cut} = \frac12 | (v_{12} - v_{13}) \times (v_{23} - v_{12})| = \frac12(x_2 - x_1)^2 \cdot \frac18 \lrp{x_2 - x_1} = \frac{1}{16} (x_2-x_1)^3.
\end{equation*}
The result easily follows.
\end{proof}

\section{A computational study}\label{sec:computations}

We study the efficacy of our MIP relaxation approach on a family of nonconvex quadratic optimization problems. We compare 9 methods:
\begin{enumerate}
    \item \GUROBI{}: The native method in Gurobi v9.1.1 for nonconvex quadratic problems.
    \item \GUROBISHIFT{}: The native method in Gurobi v9.1.1, applied to the diagonalized shift reformulation of \eqref{eqn:generic-problem-D}.
    \item \BARON{}: Baron v21.1.13, using CPLEX v12.10 for the MIP/LP solver.
    \item \BARONSHIFT{}: Baron v21.1.13, using CPLEX v12.10 for the MIP/LP solver, applied to the diagonalized shift reformulation of \eqref{eqn:generic-problem-D}.
    \item \CPLEX{}: The native method in CPLEX v12.10 for nonconvex quadratic objectives.
    \item \HL{}: The algorithm of Dong and Luo~\cite{Dong-Luo-2018}. The number of layers $L$ will correspond to the parameter $\nu$ appearing in their paper.
    \item \NN{}: The new formulation \eqref{eqn:graph-formulation}.
    \item \NMDT{}: The ``normalized multi-parametric disaggregation technique'' (NMDT) of Castro~\cite{Castro2015c}. See \cref{app:NMDT} for a restatement in terms of the number of levels $L$.
    \item \TNMDT{}: A tightened variant of NMDT also described in \cref{app:NMDT}.
\end{enumerate}
We can cluster these methods into two families: five ``native'' methods (\GUROBI{}, \GUROBISHIFT{}, \BARON{}, \BARONSHIFT{}, and \CPLEX{}) that pass an exact representation of the quadratic problem to the solver, and four ``relaxations'' (\HL{}, \NN{}, \NMDT{}, and \TNMDT{}) which relax the quadratic problem using a MIP reformulation, which is then passed to an underlying solver. For each of these relaxations, we use Gurobi v9.1.1 as the underlying MIP solver.
Note that \GUROBI{}, \BARON{}, and \CPLEX{} are directly given \eqref{eqn:box-qp}, which is an optimization problem with linear constraints and a nonconvex quadratic objective. In contrast, \GUROBISHIFT{} and \BARONSHIFT{} are given the diagonalized reformulation of \eqref{eqn:box-qp} per \eqref{eqn:generic-problem-D}, which is an optimization problem with a convex quadratic objective and a series of nonconvex quadratic constraints.\footnote{CPLEX does not support nonconvex quadratic constraints of this form, so we do not include a corresponding approach with the diagonal shift.}

Our objective in this computational study is to measure the quality of the dual bound provided by the different methods. To place the methods on an even footing on the primal side, as initialization we run the nonconvex quadratic optimization method in Gurobi v9.1.1 with ``feasible solution emphasis'' to produce a good starting feasible solution. We then inject this primal objective bound as a ``cut-off'' for each method.

We will consider 4 metrics, which will be applied with respect to a given family of instances:
\begin{itemize}
    \item \underline{time}: The shifted geometric mean of the solve time in seconds (shift is minimum solve time in the family).
    \item \underline{gap}: The shifted geometric mean of the final relative optimality gap $\frac{|\texttt{db}-\texttt{bpb}|}{|\texttt{bpb}|}$, where $\texttt{db}$ is the dual bound provided by the method and $\texttt{bpb}$ is the best observed primal solution for the instance across all methods. Shift is taken as $\max\{10^{-4}, \text{minimum gap observed in the family}\}$.
    \item \underline{BB}: The number of instances in which the method either produced the best dual bound, or attained Gurobi's default optimality criteria of $\texttt{gap} < 10^{-4}$.
    Note that on a given instance, more than one method can potentially attain the best bound.
    \item \underline{TO}: The number of instances in which the solver \emph{times out} and terminates due to the time limit.
\end{itemize}

We note that even if the solver terminates within the time limit (with an ``optimal'' solver status), the optimality gap for \NN{} or \HL{} as reported in \cref{tab:computations} may be nonzero, due to the fact that these two methods serve as relaxations for the original boxQP problem.

We implement each model in the JuMP algebraic modeling language~\cite{Dunning:2015a}. To compute the shift used by the four relaxations, \GUROBISHIFT{}, and \BARONSHIFT{}, we use Mosek v9.2 to solve a semidefinite programming problem to produce the ``tightest'' diagonal matrix $D = \operatorname{diag}(\delta)$ such that $Q+D$ is positive semidefinite as in Dong and Luo~\cite{Dong-Luo-2018}:

\begin{equation} \label{eqn:shift-sdp}
    \min_{\delta \in \bbR^n} {\bf e} \cdot \delta \quad \text{ s.t. } \quad Q + \operatorname{diag}(\delta) \succeq 0.
\end{equation}

In \cref{sec:eigen-shift} we study an alternative, simpler method for computing this shift and its computational implications. Note that this time to solve the SDP is not included in the solve time numbers, but is relatively small (on the order of a few seconds for the largest instances) and is computation that is shared by most of the approaches. 

Each method is provided a time limit of 10 minutes. Computational experiments are performed on a machine with a 3.8 GHz CPU with 24 cores and 128 GB of RAM. Each solver is restricted to one thread, and all experiments for a given instance are run concurrently. Our code and the corresponding problem instances are publicly available at \url{https://github.com/joehuchette/quadratic-relaxation-experiments}.

\subsection{Baseline comparison} \label{sec:baseline}

We start by comparing our nine methods on 99 box constrained quadratic objective (\emph{boxQP}) optimization problem instances as studied in Chen and Burer~\cite{Chen2012} and Dong and Luo~\cite{Dong-Luo-2018}:
\begin{align} \label{eqn:box-qp}
    \min_{0 \leq x \leq 1}\quad& x' Q x + c \cdot x.
\end{align}
Despite its simple constraint structure, this is a nonconvex optimization problem when $Q$ is not positive semidefinite, and is difficult from both a theoretical and a practical perspective.

For this baseline study, we fix each of the relaxations to use $L=3$ layers; we will revisit this selection in \cref{sec:num-layers}. We leave as future work an implementation that iteratively refines the approximation to guarantee a pre-specified approximation error, as done by Dong and Luo~\cite{Dong-Luo-2018}.

We split these instances into three families: 63 ``solved'' instances on which each method terminates at optimality within the time limit, 18 ``unsolved'' instances on which each method terminates due to the time limit, and 18 ``contested'' instances on which some methods terminate and some do not. We present the computational results in \cref{tab:computations}, stratified by family. At a high level, we observe that \NN{} attains the ``best bound'' on 87 of 99 instances. We now survey each family in more detail. Alternatively, we stratify the results based on the size of the instances in \cref{app:computations-by-size}. 
\begin{table}[htbp]
    \centering
    \begin{tabular}{ccrrrr} \toprule
        family & method & time (sec) & gap & BB & TO \\ \midrule
        \multirow{9}{*}{solved}   & \BARON{}       & 0.51 & 0.00\% & \nicefrac{63}{63} & - \\
                                  & \CPLEX{}       & 0.68 & 0.00\% & \nicefrac{63}{63} & - \\
                                  & \GUROBI{}      & 0.37 & 0.00\% & \nicefrac{63}{63} & - \\ \cmidrule{2-6}
                                  & \BARONSHIFT{}  & 0.96 & 0.00\% & \nicefrac{63}{63} & - \\
                                  & \GUROBISHIFT{} & 0.63 & 0.00\% & \nicefrac{62}{63} & - \\ \cmidrule{2-6}
                                  & \HL{}          & 1.24 & 0.08\% & \nicefrac{5}{63} & - \\
                                  & \NN{}          & 0.39 & 0.01\% & \nicefrac{26}{63} & - \\ 
                                  & \NMDT{}        & 0.67 & 0.02\% & \nicefrac{19}{63} & - \\
                                  & \TNMDT{}       & 1.07 & 0.01\% & \nicefrac{24}{63} & - \\\midrule
        \multirow{9}{*}{contested}& \BARON{}       &  66.1 & 0.00\% & \nicefrac{12}{18} & \nicefrac{6}{18} \\
                                  & \CPLEX{}       &  46.0 & 0.00\% & \nicefrac{17}{18} & \nicefrac{1}{18} \\
                                  & \GUROBI{}      &  34.4 & 0.00\% & \nicefrac{15}{18} & \nicefrac{3}{18} \\ \cmidrule{2-6}
                                  & \BARONSHIFT{}  & 154.6 & 0.02\% & \nicefrac{10}{18} & \nicefrac{8}{18} \\
                                  & \GUROBISHIFT{} & 450.1 & 1.22\% & \nicefrac{2}{18} & \nicefrac{16}{18} \\ \cmidrule{2-6}
                                  & \HL{}          & 429.0 & 1.49\% & \nicefrac{0}{18} & \nicefrac{15}{18} \\
                                  & \NN{}          & 273.0 & 0.24\% & \nicefrac{2}{18} & \nicefrac{10}{18} \\ 
                                  & \NMDT{}        & 318.0 & 0.54\% & \nicefrac{1}{18} & \nicefrac{11}{18} \\
                                  & \TNMDT{}       & 446.5 & 0.83\% & \nicefrac{1}{18} & \nicefrac{14}{18} \\\midrule
        \multirow{9}{*}{unsolved} & \BARON{}       & - & 11.48\% & \nicefrac{0}{18} & - \\
                                  & \CPLEX{}       & - & 15.67\% & \nicefrac{3}{18} & - \\
                                  & \GUROBI{}      & - & 30.73\% & \nicefrac{0}{18} & - \\ \cmidrule{2-6}
                                  & \BARONSHIFT{}  & - & 11.86\% & \nicefrac{0}{18} & - \\
                                  & \GUROBISHIFT{} & - &  5.21\% & \nicefrac{0}{18} & - \\ \cmidrule{2-6}
                                  & \HL{}          & - &  5.33\% & \nicefrac{0}{18} & - \\
                                  & \NN{}          & - &  4.31\% & \nicefrac{15}{18} & - \\ 
                                  & \NMDT{}        & - &  4.59\% & \nicefrac{0}{18} & - \\
                                  & \TNMDT{}       & - &  5.04\% & \nicefrac{0}{18} & - \\\bottomrule
    \end{tabular}
    \caption{Baseline computational results on 99 boxQP instances.}
    \label{tab:computations}
\end{table}

\paragraph{Solved instances}
On the solved instances, all methods are able to terminate quickly--all in under two seconds, on average. The native methods are able to meet the termination criteria on nearly all instances, while the relaxation methods lag behind. Our new \NN{} method performs the best, attaining the termination gap criteria on roughly half of the instances, while \CDA{} performs the worst, attaining it on only 5 of 63 instances. We stress that, for these experiments, $L$ is set relatively low. In \cref{sec:num-layers} we revisit this decision, and observe that this gap can be closed on these easy instances by increasing $L$ at a nominal computational cost.

\paragraph{Contested instances}
On the contested instances the native solvers \BARON{}, \CPLEX{}, and \GUROBI{} perform best, producing the best bound in a clear majority of the 18 instances. Interestingly, the shifted variants \BARONSHIFT{} and \GUROBISHIFT{} perform worse, with Gurobi exhibiting a substantial degradation in performance as opposed to without the diagonal shift. In contrast, the relaxations time out on a majority of the instances. Taken together, we conclude that there is a transitional family of instances wherein the native solvers still excel, but which the more complex relaxations succumb to the curse of dimensionality.

\paragraph{Unsolved instances}
This family of instances tests the scenario where a method is given a fixed time budget and is asked to produce the best possible dual bound. On these 18 instances, \NN{} is the clear winner, producing the best bounds on 15 and the lowest mean gap. The other relaxations come relatively close in terms of termination gap, but do not attain the best bound on any instance. The native \GUROBI{} performs the worst of all methods in terms of gap closed, but applying the shift as in \GUROBISHIFT{} helps tremendously, producing gaps than are much lower than the other native solver methods and close to what the relaxations are able to provide.

\subsection{Varying the solver focus}
Modern solvers such as Gurobi expose high-level parameters for configuring the search algorithm for different goals. In this subsection, we configure Gurobi to focus on the best objective bound (\texttt{MIPFocus=3}). We summarize our results in \cref{tab:computations-mipfocus}. The story on the ``solved'' and ``contested'' instances remains roughly the same as in \cref{sec:baseline}. However, on the ``unsolved'' instances all methods perform better, with the largest improvement coming from the native Gurobi methods. Nonetheless, the \NN{} method is still attaining the best bound on 10 of 18 instances, with the \GUROBISHIFT{} method coming close in terms of gap closed, and is able to produce the best bound on the remaining 8 instances.

\begin{table}[htbp]
    \centering
    \begin{tabular}{ccrrrr} \toprule
        family & method & time (sec) & gap & BB & TO \\ \midrule
        \multirow{6}{*}{solved}   & \GUROBI{}      & 0.70 & 0.00\% & \nicefrac{63}{63} & - \\
                                  & \GUROBISHIFT{} & 0.64 & 0.00\% & \nicefrac{63}{63} & - \\ \cmidrule{2-6}
                                  & \HL{}          & 1.71 & 0.08\% & \nicefrac{5}{63} & - \\
                                  & \NN{}          & 0.48 & 0.01\% & \nicefrac{18}{63} & - \\ 
                                  & \NMDT{}        & 1.45 & 0.04\% & \nicefrac{17}{63} & - \\
                                  & \TNMDT{}       & 1.13 & 0.01\% & \nicefrac{24}{63} & - \\\midrule
        \multirow{6}{*}{contested}& \GUROBI{}      &  49.3 & 0.00\% & \nicefrac{14}{18} & \nicefrac{4}{18} \\
                                  & \GUROBISHIFT{} & 458.9 & 0.26\% &  \nicefrac{3}{18} & \nicefrac{15}{18} \\ \cmidrule{2-6}
                                  & \HL{}          & 470.2 & 1.31\% &  \nicefrac{0}{18} & \nicefrac{16}{18} \\
                                  & \NN{}          & 301.9 & 0.18\% & \nicefrac{4}{18} & \nicefrac{10}{18} \\ 
                                  & \NMDT{}        & 457.8 & 0.83\% &  \nicefrac{0}{18} & \nicefrac{15}{18} \\
                                  & \TNMDT{}       & 436.6 & 0.34\% &  \nicefrac{1}{18} & \nicefrac{14}{18} \\\midrule
        \multirow{6}{*}{unsolved} & \GUROBI{}      & - & 6.13\% & \nicefrac{0}{18} & - \\
                                  & \GUROBISHIFT{} & - & 3.58\% & \nicefrac{8}{18} & - \\ \cmidrule{2-6}
                                  & \HL{}          & - & 4.71\% & \nicefrac{0}{18} & - \\
                                  & \NN{}          & - & 3.34\% & \nicefrac{10}{18} & - \\ 
                                  & \NMDT{}        & - & 4.50\% & \nicefrac{0}{18} & - \\
                                  & \TNMDT{}       & - & 4.05\% & \nicefrac{0}{18} & - \\ \bottomrule
    \end{tabular}
    \caption{Computational results with Gurobi configured with \texttt{MIPFocus=3}.}
    \label{tab:computations-mipfocus}
\end{table}

\subsection{Varying the relaxation resolution} \label{sec:num-layers}
In the previous experiments, we fixed the number of layers for each relaxation at $L=3$. In this subsection, we study how varying this parameter affects each relaxation, in terms of both solve time and gap closed. In particular, we consider setting $L \in \{2,4,6,8\}$ for each relaxation method, and experiment with the same set of 99 boxQP instances as before. We summarize the results in \cref{tab:computations-num-layers}.

On the ``solved'' instances we observe that, unsurprisingly, increasing $L$ allows us to reach the best bound criteria on far more instances. Moreover, we observe that this results in only a nominal increase in computational cost; all methods terminate with a mean solve time of seconds, even with the finest discretization. We observe that \NN{} performs the best, in terms of mean solve time and ``best bound'' for each value of $L$ considered. Morever, \NN{} can attain the termination criteria on each instance with $L=8$, which is not the case for any other method. These results indicate that, on easy instances, increasing the resolution is cheap, and can attain the same dual bound quality as the native solvers in roughly the same time. In contrast, on the ``unsolved'' instances we observe that increasing $L$ results in \emph{higher} gaps across the board. This is unsurprising--increasing $L$ results in larger formulations, and on instances where the solver is already struggling this will quickly lead to performance degradation due to the ``combinatorial explosion'' effect. Moreover, even small values for $L$ offer a nontrivial refinement in a branch-and-bound setting over a tight convex relaxation. This result suggests that, for instances known to be hard, a reasonable strategy would be to set $L$ to a small value by default and then target finer discretizations on individual quadratic terms as-needed, through a dynamic refinement approach or otherwise.

\begin{table}[htbp]
    \centering
    \begin{tabular}{cccrrrr} \toprule
        family & method & $L$ & time (sec) & gap & BB & TO \\ \midrule
        \multirow{16}{*}{solved} & \multirow{4}{*}{\HL{}}     & 2 & 0.39 & 0.79\% &  \nicefrac{0}{63} & - \\
                                  &                           & 4 & 1.64 & 0.01\% & \nicefrac{24}{63} & - \\
                                  &                           & 6 & 2.84 & 0.00\% & \nicefrac{52}{63} & - \\
                                  &                           & 8 & 3.74 & 0.00\% & \nicefrac{61}{63} & - \\ \cmidrule{2-7}
                                  & \multirow{4}{*}{\NN{}}    & 2 & 0.30 & 0.04\% &  \nicefrac{8}{63} & - \\
                                  &                           & 4 & 0.41 & 0.00\% & \nicefrac{41}{63} & - \\
                                  &                           & 6 & 0.48 & 0.00\% & \nicefrac{58}{63} & - \\
                                  &                           & 8 & 0.55 & 0.00\% & \nicefrac{63}{63} & - \\ \cmidrule{2-7}
                                  & \multirow{4}{*}{\NMDT{}}  & 2 & 0.41 & 0.12\% &  \nicefrac{8}{63} & - \\
                                  &                           & 4 & 0.84 & 0.01\% & \nicefrac{31}{63} & - \\
                                  &                           & 6 & 1.17 & 0.00\% & \nicefrac{41}{63} & - \\
                                  &                           & 8 & 1.41 & 0.00\% & \nicefrac{46}{63} & - \\ \cmidrule{2-7}
                                  & \multirow{4}{*}{\TNMDT{}} & 2 & 0.58 & 0.05\% &  \nicefrac{8}{63} & - \\
                                  &                           & 4 & 1.24 & 0.00\% & \nicefrac{37}{63} & - \\
                                  &                           & 6 & 1.43 & 0.00\% & \nicefrac{53}{63} & - \\
                                  &                           & 8 & 1.59 & 0.00\% & \nicefrac{62}{63} & - \\ \midrule
      \multirow{16}{*}{contested} & \multirow{4}{*}{\HL{}}    & 2 & 136.18 & 1.10\% &  \nicefrac{2}{13} & \nicefrac{0}{13} \\
                                  &                           & 4 & 552.44 & 0.68\% &  \nicefrac{0}{13} & \nicefrac{10}{13} \\
                                  &                           & 6 & 595.58 & 1.33\% &  \nicefrac{0}{13} & \nicefrac{12}{13} \\
                                  &                           & 8 & 600.00 & 2.06\% &  \nicefrac{0}{13} & \nicefrac{13}{13} \\ \cmidrule{2-7}
                                  & \multirow{4}{*}{\NN{}}    & 2 & 206.03 & 0.16\% & \nicefrac{1}{13} & \nicefrac{3}{13} \\
                                  &                           & 4 & 268.56 & 0.04\% & \nicefrac{4}{13} & \nicefrac{3}{13} \\
                                  &                           & 6 & 293.47 & 0.07\% & \nicefrac{4}{13} & \nicefrac{5}{13} \\
                                  &                           & 8 & 319.46 & 0.08\% & \nicefrac{8}{13} & \nicefrac{5}{13} \\ \cmidrule{2-7}
                                  & \multirow{4}{*}{\NMDT{}}  & 2 & 208.66 & 0.42\% &  \nicefrac{0}{13} & \nicefrac{3}{13} \\
                                  &                           & 4 & 376.76 & 0.16\% &  \nicefrac{2}{13} & \nicefrac{5}{13} \\
                                  &                           & 6 & 447.43 & 0.18\% &  \nicefrac{4}{13} & \nicefrac{7}{13} \\
                                  &                           & 8 & 473.95 & 0.18\% &  \nicefrac{4}{13} & \nicefrac{7}{13} \\ \cmidrule{2-7}
                                  & \multirow{4}{*}{\TNMDT{}} & 2 & 344.76 & 0.33\% &  \nicefrac{0}{13} & \nicefrac{6}{13} \\
                                  &                           & 4 & 511.43 & 0.36\% &  \nicefrac{1}{13} & \nicefrac{9}{13} \\
                                  &                           & 6 & 505.40 & 0.24\% &  \nicefrac{3}{13} & \nicefrac{8}{13} \\
                                  &                           & 8 & 534.02 & 0.26\% &  \nicefrac{4}{13} & \nicefrac{8}{13} \\ \midrule
       \multirow{16}{*}{unsolved} & \multirow{4}{*}{\HL{}}    & 2 & - & 3.98\% & \nicefrac{0}{23} & - \\
                                  &                           & 4 & - & 4.72\% & \nicefrac{0}{23} & - \\
                                  &                           & 6 & - & 5.02\% & \nicefrac{0}{23} & - \\
                                  &                           & 8 & - & 5.29\% & \nicefrac{0}{23} & - \\ \cmidrule{2-7}
                                  & \multirow{4}{*}{\NN{}}    & 2 & - & 3.37\% & \nicefrac{23}{23} & - \\
                                  &                           & 4 & - & 3.53\% & \nicefrac{0}{23} & - \\
                                  &                           & 6 & - & 3.63\% & \nicefrac{0}{23} & - \\
                                  &                           & 8 & - & 3.72\% & \nicefrac{0}{23} & - \\ \cmidrule{2-7}
                                  & \multirow{4}{*}{\NMDT{}}  & 2 & - & 3.55\% & \nicefrac{0}{23} & - \\
                                  &                           & 4 & - & 3.92\% & \nicefrac{0}{23} & - \\
                                  &                           & 6 & - & 4.05\% & \nicefrac{0}{23} & - \\
                                  &                           & 8 & - & 4.16\% & \nicefrac{0}{23} & - \\ \cmidrule{2-7}
                                  & \multirow{4}{*}{\TNMDT{}} & 2 & - & 3.84\% & \nicefrac{0}{23} & - \\
                                  &                           & 4 & - & 4.36\% & \nicefrac{0}{23} & - \\
                                  &                           & 6 & - & 4.31\% & \nicefrac{0}{23} & - \\
                                  &                           & 8 & - & 4.30\% & \nicefrac{0}{23} & - \\ \bottomrule
    \end{tabular}
    \caption{Computational results with varying discretization levels.}
    \label{tab:computations-num-layers}
\end{table}

\subsection{Varying the diagonal perturbation} \label{sec:eigen-shift}

We now turn our attention to how the diagonal shift that is used by the relaxations, \BARONSHIFT{}, and \GUROBISHIFT{} is computed. As discussed above, we may solve the SDP \eqref{eqn:shift-sdp} to compute a valid shift that is ``tightest'' under some reasonable objective measure. While this approach is reasonable for the boxQP instances studied here, it may not be practical for larger-scale instances due to the scalability of the SDP solver. Therefore, we compare this shift against a simpler ``eigenvalue'' shift $D = -\lambda_{\operatorname{min}} I$, where $I \in \bbR^{n \times n}$ is the identity matrix and $\lambda_{\operatorname{min}}$ is the smallest eigenvalue of $Q$. This minimum eigenvalue can be readily computed, and the resulting shift is conceptually similar to the convexification process used in $\alpha$BB~\cite{Androulakis1995}, for example.

We perform a similar experiment as in \cref{sec:baseline}, focusing on comparing our two shifts head-to-head for each method that utilizes it. We summarize our results in \cref{tab:computations-shift}. We observe that the tighter shift provided by the SDP \eqref{eqn:shift-sdp} offers a substantial improvement over the eigenvalue shift across the board. On the ``solved'' instances, we observe an order of magnitude reduction in solve time for all methods, as well as a significant increase in best bound attainment for the relaxation methods. On the ``contested'' instances, we observe a similar degradation when using the shift. The difference is perhaps most striking for \GUROBI{}: on the 22 instances, produces the best bound on 16 of 22 with the SDP shift, but with the eigenvalue shift only attains it on only one, and times out on remaining 21. Finally, for the ``unsolved'' instances we observe that the SDP shift provides 2-3x smaller gaps than the eigenvalue shift for each method.

\begin{table}[htbp]
    \centering
    \begin{tabular}{cccrrrr} \toprule
        family & method & shift & time & gap & BB & TO \\ \midrule
        \multirow{11.5}{*}{solved}  & \multirow{2}{*}{\GUROBI{}{}} & eigen & 1.81 & 0.00\% & \nicefrac{51}{51} & - \\
                                    &                              & sdp   & 0.19 & 0.00\% & \nicefrac{51}{51} & - \\ \cmidrule{2-7}
                                    & \multirow{2}{*}{\HL{}}       & eigen & 3.57 & 0.14\% & \nicefrac{17}{51} & - \\
                                    &                              & sdp   & 0.47 & 0.07\% & \nicefrac{34}{51} & - \\ \cmidrule{2-7}
                                    & \multirow{2}{*}{\NN{}}       & eigen & 1.22 & 0.01\% & \nicefrac{14}{51} & - \\
                                    &                              & sdp   & 0.14 & 0.00\% & \nicefrac{16}{51} & - \\ \cmidrule{2-7} 
                                    & \multirow{2}{*}{\NMDT{}}     & eigen & 1.73 & 0.04\% & \nicefrac{12}{51} & - \\
                                    &                              & sdp   & 0.26 & 0.01\% & \nicefrac{19}{51} & - \\ \cmidrule{2-7}
                                    & \multirow{2}{*}{\TNMDT{}}    & eigen & 3.85 & 0.02\% & \nicefrac{19}{51} & - \\
                                    &                              & sdp   & 0.40 & 0.00\% & \nicefrac{24}{51} & - \\ \midrule
        \multirow{11.5}{*}{contested} & \multirow{2}{*}{\GUROBI{}{}} & eigen & 582.97 & 3.52\% &  \nicefrac{1}{22} & \nicefrac{21}{22} \\
                                    &                              & sdp   & 130.76 & 0.05\% & \nicefrac{16}{22} &  \nicefrac{6}{22} \\ \cmidrule{2-7}
                                    & \multirow{2}{*}{\HL{}}       & eigen & 600.00 & 5.02\% & \nicefrac{0}{22} & \nicefrac{22}{22} \\
                                    &                              & sdp   & 126.65 & 0.23\% & \nicefrac{0}{22} & \nicefrac{5}{22} \\ \cmidrule{2-7}
                                    & \multirow{2}{*}{\NN{}}       & eigen & 562.82 & 2.03\% &  \nicefrac{0}{22} & \nicefrac{19}{22} \\
                                    &                              & sdp   &  43.24 & 0.02\% & \nicefrac{6}{22} &  \nicefrac{0}{22} \\ \cmidrule{2-7}
                                    & \multirow{2}{*}{\NMDT{}}     & eigen & 577.73 & 2.84\% &  \nicefrac{0}{22} & \nicefrac{20}{22} \\
                                    &                              & sdp   &  62.64 & 0.11\% & \nicefrac{1}{22} &  \nicefrac{2}{22} \\ \cmidrule{2-7}
                                    & \multirow{2}{*}{\TNMDT{}}    & eigen & 594.63 & 3.72\% & \nicefrac{0}{22} & \nicefrac{22}{22} \\
                                    &                              & sdp   & 116.27 & 0.05\% & \nicefrac{1}{22} & \nicefrac{4}{22} \\ \midrule
        \multirow{11.5}{*}{unsolved}  & \multirow{2}{*}{\GUROBI{}{}} & eigen & - & 8.64\% & \nicefrac{0}{26} & - \\
                                    &                              & sdp   & - & 4.17\% & \nicefrac{0}{26} & - \\ \cmidrule{2-7}
                                    & \multirow{2}{*}{\HL{}}       & eigen & - & 9.22\% & \nicefrac{0}{26} & - \\
                                    &                              & sdp   & - & 4.23\% & \nicefrac{0}{26} & - \\ \cmidrule{2-7}
                                    & \multirow{2}{*}{\NN{}}       & eigen & - & 8.31\% & \nicefrac{0}{26} & - \\
                                    &                              & sdp   & - & 3.12\% & \nicefrac{26}{26} & - \\ \cmidrule{2-7}
                                    & \multirow{2}{*}{\NMDT{}}     & eigen & - & 8.47\% & \nicefrac{0}{26} & - \\
                                    &                              & sdp   & - & 3.42\% & \nicefrac{0}{26} & - \\ \cmidrule{2-7}
                                    & \multirow{2}{*}{\TNMDT{}}    & eigen & - & 9.01\% & \nicefrac{0}{26} & - \\
                                    &                              & sdp   & - & 3.95\% & \nicefrac{0}{26} & - \\
 \bottomrule
    \end{tabular}
    \caption{Computational results with two different diagonal shifts.}
    \label{tab:computations-shift}
\end{table}

\subsection{A (simple) problem with quadratic constraints}\label{ssec:qcp-example}
In this section, we present a unique class of instances on which our model displays suprisingly strong performance compared to Gurobi.  In this model, we minimize a 1-norm with respect to box constraints and an additional quadratic constraint stating that the 2-norm is greater than some bound.   The specific model considered is
\begin{equation} \label{eq:qcp-example}
    \begin{array}{rll}
        \min & \tfrac{100}{n}\sum_{i=1}^n |x_i - \varepsilon_i|\\
        \text{s.t.}  & x_i \in [-1,1] & i \in \lrbr{n}\\
             & \sum_{i=1}^n x_i^2 \ge n - 0.5
    \end{array}
\end{equation}
where $\varepsilon_i = \text{rand}(-1,1) \cdot 10^{-3}$, sorted in ascending order of $|\varepsilon_i|$. We note that this problem can be solved in closed form.

We compare against \GUROBI{}{}, \GUROBISHIFT{}, and \TNMDT{} for various values of $n$, with $L=10$. The results are shown in \cref{tab:qcp-example} below.

\begin{table}[htbp]
    \centering
    \begin{tabular}{ccrrrrr} \toprule
        $n$ & method & time (sec) & gap & nodes \\ \midrule
        \multirow{4.5}{*}{10}    &  \GUROBI{}{} & 0.30 & 0.00\% & 2047\\
                                 &  \GUROBISHIFT{} & 0.08 & 0.00\% & 2047\\\cmidrule{2-6}
                                 &  \NN{} & 0.10 & 0.00\% & 89\\
                                 &  \TNMDT{} & 16.89 & 0.00\% & 45306\\\midrule
        \multirow{4.5}{*}{15}    &  \GUROBI{}{} & 1.66 & 0.00\% & 66828\\
                                 &  \GUROBISHIFT{} & 1.22 & 0.00\% & 66828\\\cmidrule{2-6}
                                 &  \NN{} & 0.72 & 0.00\% & 3477\\
                                 &  \TNMDT{} & 318.95 & 0.00\% & 1494473\\\midrule
        \multirow{4.5}{*}{18}    &  \GUROBI{}{} & 8.88 & 0.00\% & 528270\\
                                 &  \GUROBISHIFT{} & 8.08 & 0.00\% & 528210\\\cmidrule{2-6}
                                 &  \NN{} & 1.20 & 0.00\% & 7757\\
                                 &  \TNMDT{} & (TO) & 0.73\% & 3451026\\\midrule
        \multirow{4.5}{*}{20}    &  \GUROBI{}{} & 37.15 & 0.00\% & 2099824\\
                                 &  \GUROBISHIFT{} & 35.49 & 0.00\% & 2099563\\\cmidrule{2-6}
                                 &  \NN{} & 1.17 & 0.00\% & 9746\\
                                 &  \TNMDT{} & (TO) & 1.04\% & 3976996\\\midrule
        \multirow{4.5}{*}{22}    &  \GUROBI{}{} & 202.63 & 0.00\% & 8389900\\
                                 &  \GUROBISHIFT{} & 309.72 & 0.00\% & 8389887\\\cmidrule{2-6}
                                 &  \NN{} & 1.66 & 0.00\% & 11226\\
                                 &  \TNMDT{} & (TO) & 0.91\% & 2937766\\
 \bottomrule
    \end{tabular}
    \caption{Computational results for a stylized quadratically constrained problem with $L=10$.}
    \label{tab:qcp-example}
\end{table}

The results indicate a strong performance advantage of \NN{} above the competing methods shown. Note that the number of nodes for \GUROBI{}{} and \GUROBISHIFT{} are consistently close to $2^{n+1}$, with computational times to match, while \TNMDT{} is even worse. On the other hand, while $\NN{}$ shows only moderate increases in computational time, with a max of about $1.66s$, with a far smaller node count to match.

Upon deeper investigation, we found that the primary computational advantage of the $\NN$ method stems from Gurobi's Gomory cuts. In fact, turning off presolve, heuristics, and all cuts except for Gomory cuts, the performance significantly improves over the baseline performance. For $n=22$ the problem solves in 0.09s with only 191 nodes and 39 Gomory cuts. For $n=250$, over an order of magnitude higher, the problem solves in 10.65s with only 13016 nodes and 378 Gomory cuts.

We expect that Gurobi is performing a spatial branching algorithm. However, the problem was constructed so that the feasible solutions are near corners of a hypercube, while at spatial branching relaxations, optimal solutions to the relaxations are close to the center. Moreover, this property is likely to hold in a spatial branch-and-bound algorithm, meaning that you will likely need to branch on all variables in order to identify the correct corner. This behavior would yield at least $O(2^n)$ spatial branching nodes, and give poor bounds throughout the branching process, as observed in the computational results. On the other hand, for the \NN{} formulation provides an alternative branching structure that, when combined with Gomory cuts, provide excellent computational performance for this collection of instances.

\subsection{More difficult problems with nonconvex quadratic constraints}\label{ssec:qcqp-difficult}
To conclude our computational section, we study a ``best nearest'' variant of the boxQP problem that is in the spirit of the problem from \cref{ssec:qcp-example}. In more detail, for some fixed $\hat{\gamma} \in \bbR$ and $\hat{x} \in [0,1]^n$, we solve the problem
\begin{align*}
    \min_{x}\quad& \|x - \hat{x}\|_1 \\
    \text{s.t.}\quad& x'Qx + c \cdot x \leq 0.95\hat{\gamma} \\
    & 0 \leq x \leq 1.
\end{align*}
Since each boxQP instance considered has negative objective cost, this will constrain the feasible region to those points which are ``close'' to optimal for the original boxQP instance. We construct 54 instances based on the the \texttt{basic} family of boxQP instances from Chen and Burer~\cite{Chen2012}. We set $\hat{x}$ as the vector of all $0.5$s, and set $\hat{\gamma}$ to be the best primal cost on the underlying boxQP instance that is found by Gurobi after 10 minutes. We summarize the results in \cref{tab:qcqp}.

\begin{table}[htbp]
    \centering
    \begin{tabular}{ccrrrr} \toprule
        family & method & time (sec) & gap & BB & TO \\ \midrule
        \multirow{9}{*}{solved}   & \BARON{}       & 12.64 & 0.00\% & \nicefrac{8}{8} & - \\
                                  & \GUROBI{}      &  5.27 & 0.00\% & \nicefrac{8}{8} & - \\ \cmidrule{2-6}
                                  & \BARONSHIFT{}  & 13.06 & 0.00\% & \nicefrac{8}{8} & - \\
                                  & \GUROBISHIFT{} & 32.31 & 0.00\% & \nicefrac{8}{8} & - \\ \cmidrule{2-6}
                                  & \HL{}          &  6.13 & 1.70\% & \nicefrac{0}{8} & - \\ 
                                  & \NN{}          &  4.50 & 0.06\% & \nicefrac{6}{8} & - \\
                                  & \NMDT{}        & 11.99 & 1.03\% & \nicefrac{0}{8} & - \\
                                  & \TNMDT{}       & 22.66 & 0.07\% & \nicefrac{6}{8} & - \\\midrule
        \multirow{9}{*}{contested}& \BARON{}       & 176.89 &  0.00\% & \nicefrac{30}{40} & \nicefrac{12}{40} \\
                                  & \GUROBI{}      &  66.53 &  0.04\% & \nicefrac{28}{40} & \nicefrac{12}{40} \\ \cmidrule{2-6}
                                  & \BARONSHIFT{}  & 173.21 &  0.01\% & \nicefrac{29}{40} & \nicefrac{12}{40} \\
                                  & \GUROBISHIFT{} & 566.15 & 10.07\% &  \nicefrac{1}{40} & \nicefrac{39}{40} \\ \cmidrule{2-6}
                                  & \HL{}          & 412.17 &  7.12\% &  \nicefrac{0}{40} & \nicefrac{34}{40} \\
                                  & \NN{}          & 383.73 &  2.92\% &  \nicefrac{5}{40} & \nicefrac{33}{40} \\ 
                                  & \NMDT{}        & 481.24 &  6.10\% &  \nicefrac{0}{40} & \nicefrac{35}{40} \\
                                  & \TNMDT{}       & 522.14 &  4.16\% &  \nicefrac{4}{40} & \nicefrac{36}{40} \\\midrule
        \multirow{9}{*}{unsolved} & \BARON{}       & - & 69.60\% & \nicefrac{0}{6} & - \\
                                  & \GUROBI{}      & - & 65.66\% & \nicefrac{0}{6} & - \\ \cmidrule{2-6}
                                  & \BARONSHIFT{}  & - & 48.23\% & \nicefrac{0}{6} & - \\
                                  & \GUROBISHIFT{} & - & 24.42\% & \nicefrac{0}{6} & - \\ \cmidrule{2-6}
                                  & \HL{}          & - & 12.29\% & \nicefrac{0}{6} & - \\
                                  & \NN{}          & - &  8.98\% & \nicefrac{6}{6} & - \\ 
                                  & \NMDT{}        & - & 10.10\% & \nicefrac{0}{6} & - \\
                                  & \TNMDT{}       & - & 10.31\% & \nicefrac{0}{6} & - \\\bottomrule
    \end{tabular}
    \caption{Baseline computational results on 54 ``best nearest'' boxQP instances.}
    \label{tab:qcqp}
\end{table}

On the ``solved'' instances, we observe that our \NN{} relaxation has the lowest mean solve time of all methods, and is able to prove optimality on 6 of 8 methods. We note that the optimality gaps for \CDA{} and \NMDT{} are nearly two orders of magnitude greater than what was observed on the baseline boxQP instances in \cref{tab:computations}. This is in keeping with the common knowledge in the global optimization (e.g. Dey and Gupte~\cite{Dey:2015}) that tight relaxations for quadratic functions in the constraints do not necessarily lead to tight relaxations in the objective.

Similar to the baseline boxQP instances, we observe that the ``contested'' instances are a transient class where the native methods are able to terminate within the time limit with greater frequency than the relaxations, leading to significantly smaller mean optimality gaps. On the hardest ``unsolved'' instances, we again see that our \NN{} method produces the smallest optimality gap across all methods on each of the 6 instances, outperforming all other methods.

\section{Conclusion}

We present a simple MIP model for relaxing quadratic optimization problems that competes with robust commercial solvers in terms of solve time and bound quality. There are a number of ways that our method could be further improved. For example, we could follow the strategy of Dong and Luo~\cite{Dong-Luo-2018} and implement an adaptive strategy that dynamically refines individual quadratic terms as-needed. Additionally, for boxQP instances we can potentially improve performance by leveraging the results of Hansen et al.~\cite{hansen}, or applying existing cutting plane procedures~\cite{Bonami2018}. Further, we could apply bound tightening on variables~\cite{Galli2018} or include a tail-end call to a nonlinear solver to produce an optimal primal feasible solution.

We also have performed preliminary analysis on a variant of this method to model higher-order monomials as opposed to quadratics. Fundamental results of Wei~\cite{Wei1999} show that our sawtooth functions form a basis for \emph{any} continuous functions. Unfortunately, we have observed that a comparable approximation for $x^3$ seem to require a relatively large number of basis functions. We summarize our preliminary results in \cref{app:sawtooth-basis}. We believe it would be interesting future work to observe if this seeming obstruction is fundamental, or if compact methods for higher-order monomials can be derived through our approach.

\bibliographystyle{spmpsci}  
\bibliography{mybibliography}

\begin{thebibliography}{10}
\providecommand{\url}[1]{{#1}}
\providecommand{\urlprefix}{URL }
\expandafter\ifx\csname urlstyle\endcsname\relax
  \providecommand{\doi}[1]{DOI~\discretionary{}{}{}#1}\else
  \providecommand{\doi}{DOI~\discretionary{}{}{}\begingroup
  \urlstyle{rm}\Url}\fi

\bibitem{Adjiman:1998a}
Adjiman, C.S., Androulakis, I.P., Floudas, C.A.: A global optimization method,
  $\alpha$bb, for general twice-differentiable constrained {NLPs}---{II}.
  {I}mplementation and computational results.
\newblock Computers and Chemical Engineering \textbf{22}(9), 1159--1179 (1998)

\bibitem{Adjiman:1998}
Adjiman, C.S., Dallwig, S., Floudas, C.A., Neumaier, A.: A global optimization
  method, $\alpha${BB}, for general twice-differentiable constrained
  {NLPs}---{I}. {T}heoretical advances.
\newblock Computers and Chemical Engineering \textbf{22}(9), 1137--1158 (1998)

\bibitem{Anderson:2019}
Anderson, R., Huchette, J., Tjandraatmadja, C., Vielma, J.P.: Strong
  mixed-integer programming formulations for trained neural networks.
\newblock In: A.~Lodi, V.~Nagarajan (eds.) Proceedings of the 20th Conference
  on Integer Programming and Combinatorial Optimization, pp. 27--42. Springer
  International Publishing, Cham (2019).
\newblock \url{https://arxiv.org/abs/1811.08359}

\bibitem{Huchette-2019}
Anderson, R., Huchette, J., Tjandraatmadja, C., Vielma, J.P.: Strong
  mixed-integer programming formulations for trained neural networks.
\newblock In: A.~Lodi, V.~Nagarajan (eds.) Integer Programming and
  Combinatorial Optimization, pp. 27--42. Springer International Publishing,
  Cham (2019)

\bibitem{Androulakis:1995}
Androulakis, I., Maranas, C.D.: {$\alpha$BB}: {A} global optimization method
  for general constrained nonconvex problems.
\newblock Journal of Global Optimization \textbf{7}(4), 337--363 (1995)

\bibitem{Androulakis1995}
Androulakis, I.P., Maranas, C.D., Floudas, C.A.: $\alpha$bb: A global
  optimization method for general constrained nonconvex problems.
\newblock Journal of Global Optimization \textbf{7}(4), 337--363 (1995)

\bibitem{Bader2018}
Bader, J., Hildebrand, R., Weismantel, R., Zenklusen, R.: Mixed integer
  reformulations of integer programs and the affine tu-dimension of a matrix.
\newblock Mathematical Programming \textbf{169}(2), 565--584 (2018)

\bibitem{Billionnet2012}
Billionnet, A., Elloumi, S., Lambert, A.: Extending the {QCR} method to general
  mixed-integer programs.
\newblock Mathematical Programming \textbf{131}(1-2), 381--401 (2012).
\newblock \doi{10.1007/s10107-010-0381-7}.
\newblock \urlprefix\url{https://doi.org/10.1007/s10107-010-0381-7}

\bibitem{Billionnet2016}
Billionnet, A., Elloumi, S., Lambert, A.: Exact quadratic convex reformulations
  of mixed-integer quadratically constrained problems.
\newblock Mathematical Programming \textbf{158}(1), 235--266 (2016).
\newblock \doi{10.1007/s10107-015-0921-2}.
\newblock \urlprefix\url{https://doi.org/10.1007/s10107-015-0921-2}

\bibitem{Bonami2018}
Bonami, P., G\"{u}nl\"{u}k, O., Linderoth, J.: Globally solving nonconvex
  quadratic programming problems with box constraints via integer programming
  methods.
\newblock Mathematical Programming Computation \textbf{10}(3), 333--382 (2018).
\newblock \doi{10.1007/s12532-018-0133-x}.
\newblock \urlprefix\url{https://doi.org/10.1007/s12532-018-0133-x}

\bibitem{Bunel:2019}
Bunel, R., Lu, J., Turkaslan, I., Torr, P.H., Kohli, P., Kumar, M.P.: Branch
  and bound for piecewise linear neural network verification (2019).
\newblock \url{https://arxiv.org/abs/1909.06588}

\bibitem{Burer:2012a}
Burer, S., Saxena, A.: The {MILP} road to {MIQCP}.
\newblock In: J.~Lee, S.~Leyffer (eds.) Mixed Integer Nonlinear Programming,
  pp. 373--405. Springer New York (2012)

\bibitem{CastilloCastillo2018}
Castillo, P.A.C., Castro, P.M., Mahalec, V.: Global optimization of {MIQCPs}
  with dynamic piecewise relaxations.
\newblock Journal of Global Optimization \textbf{71}(4), 691--716 (2018).
\newblock \doi{10.1007/s10898-018-0612-7}.
\newblock \urlprefix\url{https://doi.org/10.1007/s10898-018-0612-7}

\bibitem{Castro2015c}
Castro, P.M.: Normalized multiparametric disaggregation: an efficient
  relaxation for mixed-integer bilinear problems.
\newblock Journal of Global Optimization \textbf{64}(4), 765--784 (2015)

\bibitem{Castro2015-Chem}
Castro, P.M.: Tightening piecewise {McCormick} relaxations for bilinear
  problems.
\newblock Computers {\&} Chemical Engineering \textbf{72}, 300--311 (2015).
\newblock \doi{10.1016/j.compchemeng.2014.03.025}.
\newblock \urlprefix\url{https://doi.org/10.1016/j.compchemeng.2014.03.025}

\bibitem{Castro2021}
Castro, P.M., Liao, Q., Liang, Y.: Comparison of mixed-integer relaxations with
  linear and logarithmic partitioning schemes for quadratically constrained
  problems.
\newblock Optimization and Engineering  (2021).
\newblock \doi{10.1007/s11081-021-09603-5}.
\newblock \urlprefix\url{https://doi.org/10.1007/s11081-021-09603-5}

\bibitem{Chen2012}
Chen, J., Burer, S.: Globally solving nonconvex quadratic programming problems
  via completely positive programming.
\newblock Mathematical Programming Computation \textbf{4}(1), 33--52 (2012)

\bibitem{Croxton:2003}
Croxton, K.L., Gendron, B., Magnanti, T.L.: A comparison of mixed-integer
  programming models for nonconvex piecewise linear cost minimization problems.
\newblock Management Science \textbf{49}(9), 1268--1273 (2003)

\bibitem{Dantzig:1960}
Dantzig, G.B.: On the significance of solving linear programming problems with
  some integer variables.
\newblock Econometrica, Journal of the Econometric Society pp. 30--44 (1960)

\bibitem{Dey:2015}
Dey, S.S., Gupte, A.: Analysis of milp techniques for the pooling problem.
\newblock Operations Research \textbf{63}(2), 412--427 (2015)

\bibitem{dey_kazachkov_lodi_mu}
Dey, S.S., Kazachkov, A.M., Lodi, A., Mu, G.: Cutting plane generation through
  sparse principal component analysis.
\newblock
  \urlprefix\url{http://www.optimization-online.org/DB_HTML/2021/02/8259.html}

\bibitem{Dong:2018}
Dong, H.: Relaxing nonconvex quadratic functions by multiple adaptive diagonal
  perturbations.
\newblock SIAM Journal on Optimization \textbf{26}(3), 1962--1985 (2016)

\bibitem{Dong-Luo-2018}
Dong, H., Luo, Y.: Compact disjunctive approximations to nonconvex
  quadratically constrained programs (2018)

\bibitem{Dunning:2015a}
Dunning, I., Huchette, J., Lubin, M.: {JuMP}: {A} modeling language for
  mathematical optimization.
\newblock SIAM Review \textbf{59}(2), 295--320 (2017)

\bibitem{Elloumi2019}
Elloumi, S., Lambert, A.: Global solution of non-convex quadratically
  constrained quadratic programs.
\newblock Optimization Methods and Software \textbf{34}(1), 98--114 (2019).
\newblock \doi{10.1080/10556788.2017.1350675}.
\newblock \urlprefix\url{https://doi.org/10.1080/10556788.2017.1350675}

\bibitem{Fortet1960}
Fortet, R.: L'algebre de boole et ses applications en recherche operationnelle.
\newblock Trabajos de Estadistica \textbf{11}(2), 111--118 (1960).
\newblock \doi{10.1007/bf03006558}.
\newblock \urlprefix\url{https://doi.org/10.1007/bf03006558}

\bibitem{Foss1954}
Foss, F.A.: The use of a reflected code in digital control systems.
\newblock Transactions of the I.R.E. Professional Group on Electronic Computers
  \textbf{{EC}-3}(4), 1--6 (1954).
\newblock \doi{10.1109/irepgelc.1954.6499244}.
\newblock \urlprefix\url{https://doi.org/10.1109/irepgelc.1954.6499244}

\bibitem{Frangioni:2006}
Frangioni, A., Gentile, C.: Perspective cuts for a class of convex $0-1$ mixed
  integer programs.
\newblock Math. Program., Ser. A \textbf{106}, 225--236 (2006)

\bibitem{Frangioni:2007}
Frangioni, A., Gentile, C.: {SDP} diagonalizations and perspective cuts for a
  class of nonseparable {MIQP}.
\newblock Operations Research Letters \textbf{35}(2), 181--185 (2007)

\bibitem{Furini:2019}
Furini, F., Traversi, E., Belotti, P., Frangioni, A., Gleixner, A., Gould, N.,
  Liberti, L., Lodi, A., Misener, R., Mittelmann, H., Sahinidis, N.V.,
  Vigerske, S., Wiegele, A.: {QPLIB}: a library of quadratic programming
  instances.
\newblock Mathematical Programming Computation \textbf{11}(2), 237--265 (2019)

\bibitem{Galli2014}
Galli, L., Letchford, A.N.: A compact variant of the qcr method for
  quadratically constrained quadratic 0--1 programs.
\newblock Optimization Letters \textbf{8}(4), 1213--1224 (2014).
\newblock \doi{10.1007/s11590-013-0676-8}.
\newblock \urlprefix\url{https://doi.org/10.1007/s11590-013-0676-8}

\bibitem{Galli2018}
Galli, L., Letchford, A.N.: A binarisation heuristic for non-convex quadratic
  programming with box constraints.
\newblock Operations Research Letters \textbf{46}(5), 529--533 (2018).
\newblock \doi{10.1016/j.orl.2018.08.005}.
\newblock \urlprefix\url{https://doi.org/10.1016/j.orl.2018.08.005}

\bibitem{Glover1975}
Glover, F.: Improved linear integer programming formulations of nonlinear
  integer problems.
\newblock Management Science \textbf{22}(4), 455--460 (1975).
\newblock \doi{10.1287/mnsc.22.4.455}.
\newblock \urlprefix\url{https://doi.org/10.1287/mnsc.22.4.455}

\bibitem{Hammer1970}
Hammer, P., Ruben, A.: Some remarks on quadratic programming with 0-1
  variables.
\newblock Revue Francaise D Automatique Informatique Recherche Operationnelle
  \textbf{4}(3), 67--79 (1970)

\bibitem{hansen}
Hansen, P., Jaumard, B., Ruiz, M., Xiong, J.: Global minimization of indefinite
  quadratic functions subject to box constraints.
\newblock Naval Research Logistics (NRL) \textbf{40}(3), 373--392 (1993).
\newblock
  \doi{https://doi.org/10.1002/1520-6750(199304)40:3<373::AID-NAV3220400307>3.0.CO;2-A}.
\newblock
  \urlprefix\url{https://onlinelibrary.wiley.com/doi/abs/10.1002/1520-6750%28199304%2940%3A3%3C373%3A%3AAID-NAV3220400307%3E3.0.CO%3B2-A}

\bibitem{Huchette:2017}
Huchette, J., Vielma, J.P.: Nonconvex piecewise linear functions: Advanced
  formulations and simple modeling tools.
\newblock Operations Research  (To appear).
\newblock \url{https://arxiv.org/abs/1708.00050}

\bibitem{Huchette:2018}
Huchette, J.A.: Advanced mixed-integer programming formulations : methodology,
  computation, and application.
\newblock Ph.D. thesis, Massachusetts Institute of Technology (2018)

\bibitem{Kaibel2013}
Kaibel, V., Pashkovich, K.: Constructing Extended Formulations from Reflection
  Relations, pp. 77--100.
\newblock Springer Berlin Heidelberg, Berlin, Heidelberg (2013)

\bibitem{Lee:2001}
Lee, J., Wilson, D.: Polyhedral methods for piecewise-linear functions {I}: the
  lambda method.
\newblock Discrete Applied Mathematics \textbf{108}, 269--285 (2001)

\bibitem{Magnanti:2004}
Magnanti, T.L., Stratila, D.: Separable concave optimization approximately
  equals piecewise linear optimization.
\newblock In: D.~Bienstock, G.~Nemhauser (eds.) Lecture Notes in Computer
  Science, vol. 3064, pp. 234--243. Springer (2004)

\bibitem{Misener2012}
Misener, R., Floudas, C.A.: Global optimization of mixed-integer
  quadratically-constrained quadratic programs ({MIQCQP}) through
  piecewise-linear and edge-concave relaxations.
\newblock Mathematical Programming \textbf{136}(1), 155--182 (2012).
\newblock \doi{10.1007/s10107-012-0555-6}.
\newblock \urlprefix\url{https://doi.org/10.1007/s10107-012-0555-6}

\bibitem{Nagarajan:2019}
Nagarajan, H., Lu, M., Wang, S., Bent, R., Sundar, K.: An adaptive,
  multivariate partitioning algorithm for global optimization of nonconvex
  programs.
\newblock Journal of Global Optimization \textbf{74}, 639--675 (2019)

\bibitem{Padberg:2000}
Padberg, M.: Approximating separable nonlinear functions via mixed zero-one
  programs.
\newblock Operations Research Letters \textbf{27}, 1--5 (2000)

\bibitem{Pardalos:1991}
Pardalos, P., Vavasis, S.: Quadratic programming with one negative eigenvalue
  is {NP}-hard.
\newblock Journal of Global Optimization \textbf{1}(1), 15--22 (1991)

\bibitem{Hao:1982}
{Phan-huy-Hao}, E.: Quadratically constrained quadratic programming: {S}ome
  applications and a method for solution.
\newblock Zeitschrift f{\"u}r Operations Research \textbf{26}(1), 105--119
  (1982)

\bibitem{Savage:1997}
Savage, C.: A survey of combinatorial gray codes.
\newblock SIAM Review \textbf{39}(4), 605--629 (1997)

\bibitem{Saxena:2008}
Saxena, A., Bomani, P., Lee, J.: Convex relaxations of non-convex mixed integer
  quadratically constrained programs: {P}rojected formulations.
\newblock Mathematical Programming \textbf{130}, 359--413 (2011)

\bibitem{Serra:2018a}
Serra, T., Ramalingam, S.: Empirical bounds on linear regions of deep rectifier
  networks (2018).
\newblock \url{https://arxiv.org/abs/1810.03370}

\bibitem{Serra:2018}
Serra, T., Tjandraatmadja, C., Ramalingam, S.: Bounding and counting linear
  regions of deep neural networks.
\newblock In: Thirty-fifth International Conference on Machine Learning (2018)

\bibitem{Tjeng:2017}
Tjeng, V., Xiao, K., Tedrake, R.: Verifying neural networks with mixed integer
  programming.
\newblock In: International Conference on Learning Representations (2019)

\bibitem{Vielma:2010}
Vielma, J.P., Ahmed, S., Nemhauser, G.: Mixed-integer models for nonseparable
  piecewise-linear optimization: {U}nifying framework and extensions.
\newblock Operations Research \textbf{58}(2), 303--315 (2010)

\bibitem{Vielma2010}
Vielma, J.P., Ahmed, S., Nemhauser, G.: Mixed-integer models for nonseparable
  piecewise-linear optimization: Unifying framework and extensions.
\newblock Operations Research \textbf{58}(2), 303--315 (2010).
\newblock \doi{10.1287/opre.1090.0721}.
\newblock \urlprefix\url{https://doi.org/10.1287/opre.1090.0721}

\bibitem{Vielma2009}
Vielma, J.P., Nemhauser, G.L.: Modeling disjunctive constraints with a
  logarithmic number of binary variables and constraints.
\newblock Mathematical Programming \textbf{128}(1-2), 49--72 (2009).
\newblock \doi{10.1007/s10107-009-0295-4}.
\newblock \urlprefix\url{https://doi.org/10.1007/s10107-009-0295-4}

\bibitem{Wei1999}
Wei, Y.: Triangular function analysis.
\newblock Computers {\&} Mathematics with Applications \textbf{37}(6), 37--56
  (1999).
\newblock \doi{10.1016/s0898-1221(99)00075-9}.
\newblock \urlprefix\url{https://doi.org/10.1016/s0898-1221(99)00075-9}

\bibitem{sven-MIQCQP}
Wiese, S.: A computational practicability study of {MIQCQP} reformulations.
\newblock \url{https://docs.mosek.com/whitepapers/miqcqp.pdf} (2021).
\newblock Accessed: 2021-02-22

\bibitem{Xia2020}
Xia, W., Vera, J.C., Zuluaga, L.F.: Globally solving nonconvex quadratic
  programs via linear integer programming techniques.
\newblock {INFORMS} Journal on Computing \textbf{32}(1), 40--56 (2020).
\newblock \doi{10.1287/ijoc.2018.0883}.
\newblock \urlprefix\url{https://doi.org/10.1287/ijoc.2018.0883}

\bibitem{Yarotsky-2016}
Yarotsky, D.: Error bounds for approximations with deep relu networks.
\newblock Neural Networks \textbf{94}, 103 -- 114 (2017)

\end{thebibliography}

\appendix

\section{Normalized multi-parametric disaggregation technique}
\label{app:NMDT}

We present a standard approach to descretizing continuous variables for handling bilinear products in nonlinear models.  This approach is perhaps the most straightforward way to convert bilinear problems to MILPs and has been referred to as \emph{Normalized Multi-Parametric Disaggregation Technique} (NMDT)~\cite{Castro2015c}.  
We adapt the bilinear approach here to a squaring a single variable.

Consider $x \in [0,1]$, and let $L$ be a positive integer. We then use the representation
\begin{subequations}
    \begin{align}
        x &= \sum_{i=1}^p 2^{-i} \beta_i + \Delta x \label{eq:NMDT_rep}\\
        \beta_i &\in \{0,1\} && i \in \intsto{L} \\
        \Delta x &\in [0, 2^{-L}],
    \end{align}
\end{subequations}
where $L$ is the number of binary variables to use.

Multiplying \eqref{eq:NMDT_rep} by $x$, and  substituting the representation into the $x \Delta x$ term, we obtain 
\begin{subequations}
    \begin{align*}
        y &= x \cdot x &&= \sum_{i=1}^L 2^{-i} x\beta_i + x\Delta x
        & &&= \sum_{i=1}^L 2^{-i} x\beta_i + \lrp{\sum_{i=1}^L 2^{-i} \beta_i + \Delta x} \Delta x \\
        & &&= \sum_{i=1}^L 2^{-i} (x + \Delta x) \beta_i + \Delta x^2
    \end{align*}
\end{subequations}
Now, using the fact that $x + \Delta x \in [0, 1+2^{-L}]$, first lift the model by adding variables $u_i$ and $\Delta u$ such that $u_i = (x + \Delta x) \beta_i$ and $\Delta u = \Delta x^2$, and then we relax these equations using McCormick Envelopes.

Given bounds $x \in [\underbar{x}^{\min}, \underbar{x}^{\max}]$ and
$\beta \in [0, 1]$, The McCormick envelope $\mathcal{M}(x,\beta)$ is defined as the following relaxation of $u = x \beta$
\begin{equation}
\mathcal{M}(x, \beta) = \left\{ (x, \beta, y) \in [\underbar{x}^{\min},\underbar{x}^{\max}]\times [0,1] \times \R \ : 
    \cref{eq:McCormick-bin}
    \right\}.
    \label{eq:McCormick}
\end{equation}
\begin{equation}
    \label{eq:McCormick-bin}
    \begin{aligned}
         \underbar{x}^{\min} \cdot \beta \le  & \, u \le \underbar{x}^{\max}\cdot \beta \\
  x - \underbar{x}^{\max}\cdot (1-\beta)  \le &\, u \le  x - \underbar{x}^{\min}\cdot (1-\beta) 
    \end{aligned}
\end{equation}
To approximate $u = x^2$ with $x \in [0, \underbar{x}^{\max}]$, this becomes
\begin{equation}
\mathcal{M}(x) = \left\{ (x, u) \in [0,\underbar{x}^{\max}] \times \R \ : \  
    \begin{aligned}
       & \, u \ge 0 \\
       \underbar{x}^{\max} (2 x - \underbar{x}^{\max}) \le &\, u \le \underbar{x}^{\max}\cdot x
    \end{aligned}
    \right\}.
    \label{eq:McCormick-sq}
\end{equation}

We present two ways to use this approach.  The first is the most direct use of $\NMDT{}$, as used in \cite{Castro2015c}.  This model is
\begin{subequations}
    \begin{align}
        x &= \sum_{i=1}^L 2^{-i} \beta_i + \Delta x\\
        y &= \sum_{i=1}^L 2^{-i} u_i + \Delta u\\
        (x, \beta_i, u_i) &\in \mathcal{M}(x, \beta_i) && i \in \intsto{L}\\
        (\Delta x, x, \Delta u) &\in \mathcal{M}(\Delta x, x)\\
        \beta_i &\in \{0,1\} && i \in \intsto{L} \\
        \Delta x &\in [0, 2^{-L}]
    \end{align}
\end{subequations}
 Here, the only error introduced in the relaxation is from $\Delta u = x\Delta x$, yielding a maximum error of $2^{-L-2}$, again occurring when $\Delta x = 2^{-L-1}$.

Alternatively, we consider the expansion of the $x \Delta x$ term.
We thus obtain the \TNMDT{} relaxation for $y=x^2$.
\begin{subequations}
    \begin{align}
        x &= \sum_{i=1}^L 2^{-i} \beta_i + \Delta x\\
        y &= \sum_{i=1}^L 2^{-i} u_i + \Delta u\\
        (x+\Delta x, \beta_i, u_i) &\in \mathcal{M}(x+\Delta x, \beta_i) && i \in \intsto{L}\\
        (\Delta x, \Delta u) &\in \mathcal{M}(\Delta x)\\
        \beta_i &\in \{0,1\} && i \in \intsto{L} \\
        \Delta x &\in [0, 2^{-L}]
    \end{align}
\end{subequations}
Since $\beta_i$ is binary, $u_i = \beta_i (x + \Delta x)$ is represented exactly. Thus, the only possible error is introduced in the relaxation of $\Delta y = \Delta x^2$, which yields a maximum error of $2^{-2L-2}$, occurring when $\Delta x = 2^{-L-1}$.

 Now, the expected error of \TNMDT{} is the expected error from the relaxation of $\Delta y = \Delta x^2$. Modeling $\Delta x$ as a uniform random variable within its bounds $[0,2^{-L}]$, and noting that the only overestimator from \cref{eq:McCormick-sq} is $y \le 2^{-L} \Delta x$ we obtain expected overapproximation error
\begin{equation}
    \begin{array}{rl}
        \mathbb{E}(2^{-L} \Delta x - \Delta x^2) &= \int_{0}^{2^{-L}} 2^L (2^{-L} \Delta x - \Delta x^2) \mathrm d \Delta x\\
        &= 2^L \int_{0}^{2^{-L}} (2^{-L} \Delta x - \Delta x^2) \mathrm d \Delta x\\
        &= 2^L (\frac16 (2^{-L})^3)\\
        &= \frac 16 2^{-2L}.
    \end{array}
\end{equation}
Similarly, the expected underapproximation error can be computed as $\frac 1{12} 2^{-2L}$.

\section{Additional baseline computation summaries} \label{app:computations-by-size}
In \cref{tab:computations-by-size} we summarize the results of our baseline experiments stratified by the number of decision variables as in, e.g., Table 4 of Dey et al.~\cite{dey_kazachkov_lodi_mu}.

\begin{table}[htbp]
    \centering
    \begin{tabular}{ccrrrr} \toprule
        family & method & time (sec) & gap & BB & TO \\ \midrule
        \multirow{9}{*}{$n \in [20,30]$}  & \BARON{}       & 0.19 & 0.00\% & \nicefrac{18}{18} & \nicefrac{0}{33} \\
                                          & \CPLEX{}       & 0.20 & 0.00\% & \nicefrac{18}{18} & \nicefrac{0}{18} \\
                                          & \GUROBI{}      & 0.14 & 0.00\% & \nicefrac{18}{18} & \nicefrac{0}{18} \\ \cmidrule{2-6}
                                          & \BARONSHIFT{}  & 0.34 & 0.00\% & \nicefrac{18}{18} & \nicefrac{0}{18} \\
                                          & \GUROBISHIFT{} & 0.05 & 0.00\% & \nicefrac{18}{18} & \nicefrac{0}{18} \\ \cmidrule{2-6}
                                          & \HL{}          & 0.16 & 0.06\% & \nicefrac{2}{18} & \nicefrac{0}{18} \\
                                          & \NN{}          & 0.05 & 0.00\% & \nicefrac{9}{18} & \nicefrac{0}{18} \\ 
                                          & \NMDT{}        & 0.09 & 0.01\% & \nicefrac{7}{18} & \nicefrac{0}{18} \\
                                          & \TNMDT{}       & 0.11 & 0.00\% & \nicefrac{8}{18} & \nicefrac{0}{18} \\\midrule
        \multirow{9}{*}{$n \in [40,50]$}  & \BARON{}       & 0.46 & 0.00\% & \nicefrac{33}{33} & \nicefrac{0}{33} \\
                                          & \CPLEX{}       & 0.70 & 0.00\% & \nicefrac{33}{33} & \nicefrac{0}{33} \\
                                          & \GUROBI{}      & 0.37 & 0.00\% & \nicefrac{33}{33} & \nicefrac{0}{33} \\ \cmidrule{2-6}
                                          & \BARONSHIFT{}  & 0.87 & 0.00\% & \nicefrac{33}{33} & \nicefrac{0}{33} \\
                                          & \GUROBISHIFT{} & 0.54 & 0.00\% & \nicefrac{33}{33} & \nicefrac{0}{33} \\ \cmidrule{2-6}
                                          & \HL{}          & 1.00 & 0.07\% & \nicefrac{3}{33} & \nicefrac{0}{33} \\
                                          & \NN{}          & 0.36 & 0.00\% & \nicefrac{16}{33} & \nicefrac{0}{33} \\ 
                                          & \NMDT{}        & 0.61 & 0.03\% & \nicefrac{11}{33} & \nicefrac{0}{33} \\
                                          & \TNMDT{}       & 0.98 & 0.01\% & \nicefrac{15}{33} & \nicefrac{0}{33} \\\midrule
        \multirow{9}{*}{$n \in [60,80]$}  & \BARON{}       &  13.66 & 0.00\% & \nicefrac{15}{21} & \nicefrac{6}{21} \\
                                          & \CPLEX{}       &  15.10 & 0.00\% & \nicefrac{19}{21} & \nicefrac{3}{21} \\
                                          & \GUROBI{}      &   9.99 & 0.00\% & \nicefrac{16}{21} & \nicefrac{5}{21} \\ \cmidrule{2-6}
                                          & \BARONSHIFT{}  &  25.29 & 0.00\% & \nicefrac{15}{21} & \nicefrac{6}{21} \\
                                          & \GUROBISHIFT{} & 112.85 & 0.00\% & \nicefrac{12}{21} & \nicefrac{8}{21} \\ \cmidrule{2-6}
                                          & \HL{}          & 112.31 & 0.30\% & \nicefrac{0}{21} & \nicefrac{7}{21} \\
                                          & \NN{}          &  37.26 & 0.04\% & \nicefrac{4}{21} & \nicefrac{3}{21} \\ 
                                          & \NMDT{}        &  53.73 & 0.12\% & \nicefrac{2}{21} & \nicefrac{4}{21} \\
                                          & \TNMDT{}       & 110.42 & 0.07\% & \nicefrac{2}{21} & \nicefrac{6}{21} \\\midrule
        \multirow{9}{*}{$n \in [90,125]$} & \BARON{}       & 261.48 &   0.24\% &  \nicefrac{9}{27} & \nicefrac{18}{27} \\
                                          & \CPLEX{}       & 218.22 &   0.13\% & \nicefrac{13}{27} & \nicefrac{16}{27} \\
                                          & \GUROBI{}      & 170.56 &   0.20\% & \nicefrac{11}{27} & \nicefrac{16}{27} \\ \cmidrule{2-6}
                                          & \BARONSHIFT{}  & 375.45 &   0.55\% &  \nicefrac{7}{27} & \nicefrac{20}{27} \\
                                          & \GUROBISHIFT{} & 578.95 &   3.43\% &  \nicefrac{1}{27} & \nicefrac{26}{27} \\ \cmidrule{2-6}
                                          & \HL{}          & 569.53 &   3.84\% &  \nicefrac{0}{27} & \nicefrac{26}{27} \\
                                          & \NN{}          & 533.97 &   2.35\% & \nicefrac{14}{27} & \nicefrac{25}{27} \\ 
                                          & \NMDT{}        & 543.72 &   2.84\% &  \nicefrac{0}{27} & \nicefrac{25}{27} \\
                                          & \TNMDT{}       & 563.94 &   3.39\% &  \nicefrac{0}{27} & \nicefrac{26}{27} \\\bottomrule
    \end{tabular}
    \caption{Computational results with instances stratified based on number of variables.}
    \label{tab:computations-by-size}
\end{table}

\section{General representations with sawtooth bases}
\label{app:sawtooth-basis}
The premise our formulation is that the function $y=x^2$ can be arbitrarily closely approximated by a series of sawtooth functions.  We discuss here if such approximations could conveniently apply to other polynomials.

In \cite{Wei1999}, the authors present a Fourier series-like method that leverages orthogonal triangular functions to derive a convergent class of $L_2$-optimal approximations for general functions on the interval $[-\pi,\pi]$. Define the periodic triangular functions

\begin{equation}
    \begin{array}{rl}
         X(x) = \begin{cases} \tfrac{\pi^2 + 2\pi x}{8} & -\pi < x + 2\pi k \le 0, k \in \Z\\
                              \tfrac{\pi^2 - 2\pi x}{8} & 0 < x + 2\pi k \le \pi, k \in \Z
                \end{cases}\\
         Y(x) = \begin{cases} \tfrac{\pi x}{4} & -\tfrac{\pi}{2} < x + 2\pi k \le \tfrac{\pi}{2}, k \in \Z\\
                              \tfrac{\pi^2 - \pi x}{4} & \tfrac{\pi}{2} < x + 2\pi k \le \tfrac{3\pi}{2}, k \in \Z
                \end{cases}
    \end{array}
\end{equation}

The authors then build their orthogonal basis functions using an orthogonal linear transformation of the basis
\begin{equation*}
    1, X(x), Y(x), X(2x), Y(2x), \dots, X(nx), Y(nx).    
\end{equation*}
However, as with Fourier series approximations, this method has the limitation that all approximating functions are equal at the endpoints of the interval, resulting in a poor approximation for functions at which the endpoints are not equal. Thus, to obtain good approximations for $x^3$ on $[-\pi,\pi]$, we first add the linear function $-\pi^2 x$ to enforce equality at the endpoints. 

Then, applying this method to $x^2$ and $x^3 - \pi^2 x$ on the interval $x \in [-\pi,\pi]$, we obtain the following numbers for the ($L_1$-error). Note that almost all of the $Y(nx)$ functions are relevant for approximating $x^3 - \pi^2 x$ (and no $X(nx)$'s), while only a few $X(nx)$ functions (and no $Y(nx)'s$) are relevant for approximating $x^2$.

\begin{table}[htpb]
\begin{center}
{\footnotesize
\begin{tabular}{c|lllllll}
Function \ & $N=2$ & $N = 4$ & $N = 8$ & $N = 16$ & $N = 32$ \\
\hline
     $x^2$ & 0.994 \ & 0.249 ~ \textbf{(4)} \ & \ 0.0622 ~ \textbf{(4)} \ & \ 0.0155 ~ \textbf{(4)} \ & \ 0.0039 ~ \textbf{(3.97)}\\
     $x^3 - \pi^2 x$  & 7.23 \ & 2.07 ~ \textbf{(3.5)} \ & \ 0.626 ~ \textbf{(3.3)} \ & \ 0.304 ~ \textbf{(2.06)} \ & \ 0.108 ~ \textbf{(2.81)} 
\end{tabular}}
\end{center}
\caption{Comparison of $L_1$-error. Factor of improvement over the previous value for $L$ is shown in bold.}
\label{tab:volume-triangle-approx}
\end{table}       

To investigate the outlook of sparsely approximating $x^3$ with triangular functions directly, we solved the following MIP to obtain the $L_1$-optimal triangular approximation to $x^3$ on the interval $[0,1]$ using re-scaled versions of the basis functions above, and explicitly including a linear shift. We discretely approximate the $L-1$ error via the error at uniformly-spaced points $x_1, \dots, x_{N_p} \in [0,1]$, allowing the inclusion of only $N_f$ triangular functions.
\begin{equation}
    \begin{array}{rrll}
         \min && \tfrac{1}{Np}\dsum_{j = 1}^{N_p}t_j \\
            s.t.  & t_j &\ge \sum_{i = I}(\lambda_i f_i(x_j) + \lambda_0 x_j + f_c) - x_j^3 & \forall j\\
              & t_j &\ge -(\sum_{i \in I}(\lambda_i f_i(x_j) + \lambda_0 x_j + f_c) - x_j^3) & \forall j\\
              & -M \cdot \alpha_i &\le \lambda_i \le M \cdot \alpha_i & \forall i \ge 1\\
              & \dsum_{i=1}^N \alpha_i &\le N_f\\
              & \lambda_i &\in [-M,M] & \forall i\\
              & \alpha_i &\in \{0,1\} & \forall i
    \end{array}
\end{equation}
The result, shown in \cref{fig:xcub-sparse}, suggests that it is not possible to use this triangular basis to obtain a similar-quality sparse approximation for $x^3$ as for $x^2$: the best achievable error rate for $x^3$ is roughly $O(N_f^{-2})$, compared to $O(2^{-2N_f})$ for the quadratic.  See also Table~\ref{tab:volume-triangle-approx} where we compare the convergence of the two approximations.  
\begin{figure}
    \centering
    \includegraphics[scale=.6]{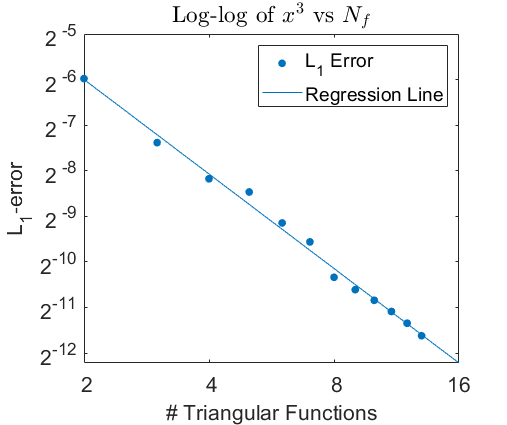}
    \caption{The $L_1$-error of $x^3$ vs. the number of approximating triangular functions $N_f$. The equation of the regression line suggests an asymtotic error rate of roughly $O(N_f^{-2})$, compared to $O(2^{-2N_f})$ for the quadratic.}
    \label{fig:xcub-sparse}
\end{figure}

\end{document}